\documentclass[reqno]{amsart}

\usepackage{xypic}
\xyoption{all}
\usepackage{epsfig}
\usepackage{color}
\usepackage{amsthm}
\usepackage{changepage}
\usepackage{amssymb}
\usepackage{mathdots}
\usepackage{amsmath}
\usepackage{amscd}
\usepackage{amsopn}
\usepackage{enumerate}
\usepackage{mathtools}
\usepackage{url}

\usepackage{hyperref}\hypersetup{colorlinks}

\usepackage{color} 

\definecolor{darkred}{rgb}{1,0,0} 
\definecolor{darkgreen}{rgb}{0,0.8,0}
\definecolor{darkblue}{rgb}{0,0,1}

\hypersetup{colorlinks,
linkcolor=darkblue,
filecolor=darkgreen,
urlcolor=darkred,
citecolor=darkgreen}

\newcommand\scalemath[2]{\scalebox{#1}{\mbox{\ensuremath{\displaystyle #2}}}}

\numberwithin{equation}{section}

\newtheorem {Theorem}{Theorem}

\numberwithin{Theorem}{section}

\newtheorem {Lemma}[Theorem]    {Lemma}
\newtheorem {Claim}[Theorem]    {Claim}
\newtheorem {Question}[Theorem]    {Question}
\newtheorem {Proposition}[Theorem]{Proposition}
\newtheorem {Corollary}[Theorem]{Corollary}
\theoremstyle{definition}
\newtheorem{Definition}[Theorem]{Definition}
\newtheorem{Remark}[Theorem]{Remark}
\newtheorem{Example}[Theorem]{Example}

\newtheoremstyle{TheoremForIntro} 
        {.6em}{.6em}              
        {\itshape}                      
        {}                              
        {\bfseries}                     
        {. }                             
        { }                             
        {\thmname{#1}\thmnote{ \bfseries #3}}
    \theoremstyle{TheoremForIntro}
    \newtheorem{TheoremIntro}[Theorem]{Theorem}
    \newtheorem{CorollaryIntro}[Theorem]{Corollary}
\newtheorem{PropositionIntro}[Theorem]{Proposition}

\expandafter\chardef\csname pre amssym.def at\endcsname=\the\catcode`\@
\catcode`\@=11
\def\undefine#1{\let#1\undefined}
\def\newsymbol#1#2#3#4#5{\let\next@\relax
 \ifnum#2=\@ne\let\next@\msafam@\else
 \ifnum#2=\tw@\let\next@\msbfam@\fi\fi
 \mathchardef#1="#3\next@#4#5}
\def\mathhexbox@#1#2#3{\relax
 \ifmmode\mathpalette{}{\m@th\mathchar"#1#2#3}%
 \else\leavevmode\hbox{$\m@th\mathchar"#1#2#3$}\fi}
\def\hexnumber@#1{\ifcase#1 0\or 1\or 2\or 3\or 4\or 5\or 6\or 7\or 8\or
 9\or A\or B\or C\or D\or E\or F\fi}

\font\teneufm=eufm10
\font\seveneufm=eufm7
\font\fiveeufm=eufm5
\newfam\eufmfam
\textfont\eufmfam=\teneufm
\scriptfont\eufmfam=\seveneufm
\scriptscriptfont\eufmfam=\fiveeufm
\def\frak#1{{\fam\eufmfam\relax#1}}
\catcode`\@=\csname pre amssym.def at\endcsname

\newcommand{\fsl}{{\frak {sl}}}
\newcommand{\fgl}{{\frak {gl}}}

\newcommand{\hook}{\hookrightarrow}

\newcommand{\Ff}{{\mathcal F}}

\newcommand{\Gg}{{\mathcal G}}

\newcommand{\Mm}{{\mathcal M}}

\newcommand{\Rr}{{\mathcal R}}

\newcommand{\Uu}{{\mathcal U}}

\newcommand{\up}[1]{^{^{#1}}}
\newcommand{\dwn}[1]{_{_{#1}}}

\newcommand{\mtrx}[1]{\left (\begin{matrix}#1\end{matrix}\right)}
\newcommand{\eqtns}[2]{\begin{equation}
\begin{dcases}#1\end{dcases}\label{#2}\end{equation}}

\def    \A      {{\mathbb A}}

\def    \C      {{\mathbb C}}
\def    \R      {{\mathbb R}}
\def    \Z      {{\mathbb Z}}

\def    \P    {{\mathbb P}}

\def    \ra     {{\rightarrow}}
\def    \lra     {{\longrightarrow}}
\def    \haf    {{\frac{1}{2}}}

\def    \p      {\partial}

\def    \bz     {\bar{z}}

\usepackage[margin=1in]{geometry}
\author{Brian Collier}
\author{Qiongling Li}
\title{Asymptotics of certain families of Higgs bundles in the Hitchin component}

\begin{document}

\setlength{\smallskipamount}{6pt}
\setlength{\medskipamount}{10pt}
\setlength{\bigskipamount}{16pt}

\begin{abstract}
Using Hitchin's parameterization of the Hitchin-Teichm\"uller component of the $SL(n,\R)$ representation variety, we study the asymptotics of certain families of representations.
In fact, for certain Higgs bundles in the $SL(n,\R)$-Hitchin component, we study the asymptotics of 
the Hermitian metric solving the Hitchin equations. This analysis is used to estimate the asymptotics of the corresponding family of flat connections as we scale the differentials by a real parameter. We consider Higgs fields that have only one holomorphic differential $q_n$ of degree $n$ or $q_{n-1}$ of degree $(n-1).$ The asymptotics of the corresponding parallel transport operator is calculated, and used to prove a special case of a conjecture of Katzarkov, Knoll, Pandit and Simpson \cite{harmonicbuildingWKB} on the Hitchin WKB problem.
\end{abstract}
\maketitle  
\tableofcontents
\section{Introduction}
For a closed, connected, oriented surface $S$ of genus $g\geq2$, consider the space of group homomorphisms $\rho:\pi_1(S)\ra G$ from the fundamental group $\pi_1(S)$ to a reductive Lie group $G.$
Through the nonabelian Hodge correspondence \cite{selfduality,localsystems}, the representation variety
\[\Rr(\pi_1,G):=Hom(\pi_1(S),G)//G\ \cong\ \Mm_{Higgs}(G)\]
is diffeomorphic to the moduli space of semistable $G$-Higgs bundles.

Some representations are of particular geometric interest; for $G=PSL(2,\R),$ there are two isomorphic connected components of $\Rr(\pi_1,PSL(2,\R))$ that are open cells of complex dimension $3(g-1).$ 
The representations in these components are called Fuchsian, and both components can be identified with Teichm\"uller space \cite{TopologicalComponents}. 
The unique irreducible representation $PSL(2,\R)\hookrightarrow PSL(n,\R),$ singles out a component of $\Rr(\pi_1,PSL(n,\R)),$ namely the component containing representations which factor through Fuchsian representations. 
Such a construction suggests that these representations are geometrically interesting; however, it gives no information on the structure of this component.

In \cite{liegroupsteichmuller}, Hitchin used Higgs bundles to show that this component of the representation variety
$\Rr(\pi_1,PSL(n,\R))$
is an open cell of complex dimension $(n^2-1)(g-1).$ This component has since been called the Hitchin component or Hitchin-Teichm\"uller component and will be denoted $Hit_n(S).$
Furthermore, if we fix a Riemann surface structure $\Sigma$ on $S,$ with canonical bundle $K,$ then $Hit_n(S)$ is parameterized by the space
\[Hit_n(S)\cong H^0(\Sigma,K^2)\oplus H^0(\Sigma,K^3)\oplus\cdots\oplus H^0(\Sigma,K^n)\]
 of holomorphic differentials. In this parameterization, the embedded copy of Teichm\"uller space, i.e. the representations factoring through $PSL(2,\R),$ is realized by setting all but the quadratic differential to zero. 
 A similar construction holds for any split real form of a complex semisimple Lie group. 
 For the split real forms $Sp(2n,\R)$ and $SO(n,n+1),$ the corresponding Hitchin components can be realized as the zero locus of the differentials of odd degree in $Hit_{2n}(S)$ and $Hit_{2n+1}(S)$ respectively \cite{GuichardZarclosureHitchinRep}.

To introduce the main objects, we now briefly recall the nonabelian Hodge correspondence for $SL(n,\C)$.
An $SL(n,\C)$-Higgs bundle is a pair $(E,\phi),$ where $E\ra\Sigma$ is a rank $n$ holomorphic vector bundle with trivial determinant bundle and $\phi$ is a traceless holomorphic bundle map $\phi:E\ra E\otimes K.$
For a stable Higgs bundle $(E,\phi),$ there is a unique Hermitian metric $h$ on $E$, with Chern connection $A_h$, solving the Hitchin equations
\[F_{A_h}+[\phi,\phi^{*_h}]=0,\]
where $F_{A_h}$ is the curvature of $A_h$ and $\phi^{*_h}$ denotes the the hermitian adjoint. In his foundational paper \cite{selfduality}, Hitchin proved this for $n=2$ and later, Simpson \cite{localsystems} proved it for general $n.$

Such a solution $(A_h,\phi)$ gives rise to a flat connection $A+\phi+\phi^{*_h},$ and hence defines a map
\[ \Mm_{Higgs}(SL(n,\C))\longrightarrow \Rr(\pi_1,SL(n,\C))\]
from the $SL(n,\C)$-Higgs bundle moduli space to the $SL(n,\C)$ representation variety. The fact that this map has an inverse was proven for $n=2$ by Donaldson \cite{harmoicmetric} and by Corlette \cite{canonicalmetrics} in general. Corlette's Theorem states that given an irreducible flat $SL(n,\C)$ connection, there exists a unique \textit{harmonic} metric on the corresponding flat $SL(n,\C)$-bundle. The equations for a harmonic metric and a flat connection can be viewed as a `real' version of the Hitchin equations.

We will restrict our study to Higgs bundles in $Hit_n(S);$ as noted above, such Higgs bundles are parameterized by $\bigoplus\limits_{j=2}^n H^0(\Sigma,K^j).$ To describe the Hitchin component we first fix a square root $K^\haf$ of $K.$ The Higgs bundle associated to $(q_2,\dots,q_n)\in \bigoplus\limits_{j=2}^n H^0(\Sigma,K^j)$ is
\[E=K^\frac{n-1}{2}\oplus K^{\frac{n-3}{2}}\oplus\cdots\oplus K^{-\frac{n-3}{2}}\oplus K^{-\frac{n-1}{2}} \] and
\[\phi=\mtrx{0&\frac{n-1}{2}q_2&q_3&\dots&q_{n-1}&q_n\\
1&0&\frac{n-3}{2}q_2&\dots&q_{n-2}&q_{n-1}\\
&\ddots&&\ddots&&\\
&&&&\frac{n-3}{2}q_2&q_3\\
&&&1&0&\frac{n-1}{2}q_2\\
&&&&1&0}:E\longrightarrow E\otimes K.
\]
Such a $\phi$ will be denoted by $\tilde e_1+q_2e_1+q_3e_2+\dots+q_ne_{n-1}.$
It can be shown that the flat connections corresponding to such Higgs bundles have holonomy in $SL(n,\R).$

Given a stable Higgs bundle $(E,\phi)$, consider the family of stable Higgs bundles $(E,t\phi),$ where $t \in \C^*.$ Solving the Hitchin equations yields a family of harmonic metrics $h_t$ on $E$ and thus a family of flat connections $\nabla_t$ with corresponding representations $\rho_t.$ 
If $\widetilde\Sigma\ra\Sigma$ is the universal cover, then for any $P,P'\in\widetilde\Sigma,$ let $T_{P,P'}(t)$ be the parallel transport matrices of the family of flat connections. 
The family of metrics $h_t$ also gives a corresponding family of $\rho_t$ equivariant harmonic maps 
\[f_t:\widetilde\Sigma\ra SL(n,\R)/SO(n,\R).\]
In a recent preprint \cite{harmonicbuildingWKB}, Katzarkov, Noll, Pandit, and Simpson asked the following question:
\begin{Question}\label{Question}
	What is the asymptotic behavior of $T_{P,P'}(t),$ $\rho_t,$ and $f_t$ as $t\ra\infty$?
\end{Question}

They call this the Hitchin WKB problem. In fact, in \cite{harmonicbuildingWKB} they study the asymptotics of the following related, but different, family of flat connections. 
Let $E$ be a fixed holomorphic vector bundle and $\nabla_0$ be a flat holomorphic $SL(n,\C)$ connection. 
Choose a Higgs field $\phi\in H^0(\Sigma,End(E)\otimes K)$ so that the Higgs bundle $(E,\phi)$ is polystable and not in the nilpotent cone. 
Consider the family of flat holomorphic connections 
\begin{equation}\label{CWKB}\nabla_t=\nabla_0+t\phi.\end{equation}
The asymptotics of this family of flat connections and the corresponding representations $\rho_t$ is called the complex or Riemann-Hilbert WKB problem. 
If we fix a hermitian metric on $E,$ the family of holomorphic flat connections $\nabla_t$ gives rise to a family of $\rho_t$-equivariant (not harmonic) maps  $f_t:\widetilde\Sigma\ra SL(n,\C)/SU(n).$
Katzarkov et al answer the analog of Question \ref{Question} for the complex WKB problem (discussed further at the end of this introduction and in section \ref{HarmonicSection}) and conjecture that the same is true for the Hitchin WKB problem. 

Note that these two asymptotic families are fundamentally different. For instance, for the complex WKB problem, the Higgs field is holomorphic with respect to the \em{flat connection}, \em{} not the (non-flat) Chern connection $\nabla_A$ as in the Hitchin WKB problem. Furthermore, there is no PDE to solve in the complex WKB problem.
Since the Hitchin WKB problem involves asymptotically solving the Hitchin equations, it seems to be a difficult problem. 

In this paper we restrict to the following situation
\begin{itemize}
	\item $(E,\phi)$ is in the Hitchin component
	\item  $t\in\R^+$
	\item $\phi=\tilde e_1+q_ne_{n-1}$ and $\phi=\tilde e_1+q_{n-1}e_{n-2}.$
\end{itemize}
Instead of $t\phi$, we use $tq_n$ and $tq_{n-1},$ which are equivalent to $t^\frac{1}{n}\phi$ and $t^\frac{1}{n-1}\phi$ after a gauge transformation.

For the Hitchin component, such asymptotics are related to the compactification of $Hit_n(S)$ due to Parreau \cite{CompactificationHit} for general $n$ and Kim \cite{conmpactifyRP2stuct} for $n=3.$
When $n=2,$ the Hitchin component is the classical Teichm\"uller space. In \cite{RaysInTeich}, Wolf considered rays of the form $(\Sigma,tq_2)$ inside Teichm\"uller space, and studied the asymptotics of the associated family of harmonic maps $f_t:\widetilde\Sigma\ra\mathbb{H}^2.$ 
He realized the appropriate limit of the family $f_t$ as a harmonic map into the $\R$-tree dual to the horizontal measured foliation determined by the holomorphic quadratic differential $q_2$ \cite{HarmRtree}. This compactification was generalized to the $SL(2,\C)$ character variety in the work of Daskalopoulos, Dostoglou, and Wentworth \cite{DDW}.

For $n=3,$ the Hitchin component was identified with the space of convex real projective structures on $S$ by Choi and Goldman \cite{ChoiGoldman}.
Using this identification, Labourie \cite{Labourie} and Loftin \cite{Lof01} independently showed that the space of convex real projective structures on $S$ is in bijection with the space of pairs $(\Sigma,q_3)$ where $\Sigma$ is a Riemann surface structure on $S$ and $q_3$ is a holomorphic cubic differential on $\Sigma.$
Away from zeros of $q_3,$ Loftin \cite{flatmetriccubicdiff} considered the asymptotic holonomy of the convex real projective structures corresponding to the ray $(\Sigma,tq_3)$ as $t\ra \infty.$
Recently, Dumas and Wolf \cite{DM} gave a correspondence between cubic polynomial differentials and polygons in $\R\P^2.$ As $t\ra\infty,$ such polygons serve as local models for neighborhoods of the zeros of the cubic differentials. 

In \cite{TaubesAsymptotics3manifolds}, Taubes considers sequences of connections on principal $PSL(2,\C)$-bundles over compact two and three dimensional manifolds. He studies, in detail, the limiting behavior of such sequences as an extension of Uhlenbeck's compactness theorem. For $PSL(2,\C)$-Higgs bundles $(E,\phi)$ such that $det(\phi)$ has simple zeros, Mazzeo, Swoboda, Weiss, and Witt \cite{EndsHiggs} have recently given a constructive proof of the existence and uniqueness of solutions to Hitchin's equation in a neighborhood of infinity. In particular, they are able to describe the limiting metric around the zeros of the $det(\phi).$ 

We now describe our main results. 
\begin{TheoremIntro}[\ref{MetricSplitting}]
	Let $(E,\phi)$ be in the Hitchin component with $\phi=\tilde e_1+\sum\limits_{j=0\ mod\ k}q_je_{j-1},$ for some $k\leq n.$ Then the harmonic metric solving the Hitchin equation splits as a direct sum metric on $E=E_1\oplus\cdots\oplus E_k$ where
	\[E_j=K^\frac{n+1-2j}{2}\oplus K^{\frac{n+1-2j}{2}-k}\oplus K^{\frac{n+1-2j}{2}-2k}\oplus\cdots  \]
\end{TheoremIntro}
Note that the dots for the holomorphic bundle $E_j$ go on as long as it makes sense. For instance, when $k=n,$ there is only one summand for each $j.$ In the restricted setting $\phi=\tilde e_1+q_ne_{n-1}$ and $\phi=\tilde e_1+q_{n-1}e_{n-2},$ we obtain:
\begin{CorollaryIntro}[\ref{CorMetricSplit}](Metric Splitting)
	For $k=n$ and $k=n-1$ the Higgs fields are $\phi=\tilde e_1+q_ne_{n-1}$ and $\phi=\tilde e_1+q_{n-1}e_{n-2}$, and the harmonic metric splits as
	\[h_1\oplus h_2\oplus\dots\oplus h_2^{-1}\oplus h_1^{-1}\]
	on the line bundles
	\[K^{\frac{n-1}{2}}\oplus K^{\frac{n-3}{2}}\oplus\dots\oplus K^{-\frac{n-3}{2}}\oplus K^{-\frac{n-1}{2}}.
	\]
\end{CorollaryIntro}
For $k=n,$ this was proven by Baraglia \cite{cyclichiggsaffinetoda, g2geometry}, and used to study, amongst other things, the relation between the Hitchin equations and the affine Toda equations. Higgs fields of the form $\phi=\tilde e_1+\sum\limits_{j=0\ mod\ k}q_je_{j-1}$  are fixed points of an action by the $k^{th}$ roots of unity, this will be crucial in the proof of Theorem \ref{MetricSplitting}.

Corollary \ref{CorMetricSplit} significantly simplifies the Hitchin equations from $n^2$ equations to $\lfloor \frac{n}{2}\rfloor$ equations. We first obtain estimates for the solution metric $h_t$ of the Hitchin equations as $t\ra\infty$ by repeatedly using the maximum principle and a standard ``telescope" trick.

\begin{TheoremIntro}[\ref{metrictheorem}] (Metric Asymptotics)
For every point $p\in\Sigma$ away from the zeros of $q_n$ or $q_{n-1},$ as $t\ra\infty$
\begin{enumerate}[1.]
		\item For $(\Sigma,\ \tilde e_1+tq_ne_{n-1})\in Hit_n(S),$ the metric $h_j(t)$ on $K^\frac{n+1-2j}{2}$ admits the expansion
		\begin{equation*}
		h_j(t)=(t|q_n|)^{-\frac{n+1-2j}{n}}\left(1+O\left(t^{-\frac{2}{n}}\right)\right) \ \ \ \  \text{for\ all\  } j
		\end{equation*}
		\item For $(\Sigma,\ \tilde e_1 +tq_{n-1}e_{n-2})\in Hit_n(S),$ the metric $h_j(t)$ on $K^\frac{n+1-2j}{2}$ admits the expansion
		\begin{equation*}
			h_j(t)=\begin{dcases}
				(t|q_{n-1}|)^{-\frac{n+1-2j}{n-1}}\left(1+O\left(t^{-\frac{2}{n-1}}\right)\right) & \text{for\ } j=1\ \ \text{and}\ \ j=n\\
				(2t|q_{n-1}|)^{-\frac{n+1-2j}{n-1}}\left(1+
				O\left(t^{-\frac{2}{n-1}}\right)\right)
				& \text{for\ } 1< j <n
			\end{dcases}
		\end{equation*}
		\end{enumerate}
\end{TheoremIntro}

Using the asymptotic estimates of the solution metric and error estimates, we integrate the ODE defined by the flat connection.
This yields an estimate of the parallel transport matrices $T_{P,P'}(t)$ as $t\ra\infty.$
For $(\Sigma,(0,\dots,0,tq_n))\in Hit_n(S)$, let $P\in\widetilde{\Sigma}$ be a point at which $q_n$ does not vanish. Choose a neighborhood $\Uu\dwn{P}$ centered at $P$, with coordinate $z,$ so that $q_n=dz^n.$ Any $P'\in\mathcal{U}\dwn{P}$ may be written in polar coordinate as $P'=Le^{i\theta}$. Suppose $\gamma(s)$ is a $|q_n|^{\frac{2}{n}}-$geodesic from $P$ to $P'$ parametrized by arc length $s$. 

In section \ref{ODESetup}, we translate the problem of computing the parallel transport operator along the path $\gamma$ into computing $\Phi(L),$ where $\Phi$ solves the ODE
\begin{equation}\label{PhieqIntro}
\Phi(0)=I, \ \ \ \ \ \ \ \ \
\dfrac{d\Phi}{ds}=\left[t^{\frac{1}{n}} M_n(\theta)
+R\right]\Phi
\end{equation}
where $M_n(\theta)=\mtrx{\mu_1&&\\&\ddots&\\&&\mu_n}$ and $R$ is the error matrix.

With an extra condition on the path, we obtain the entire set of eigenvalues of the parallel transport operator along the path asymptotically. 
\begin{TheoremIntro}[\ref{TransportAsymp}](Parallel Transport Asymptotics) Suppose $P,\ P'$ and the path $\gamma(s)$ are as above. If $P'$ has the property that for every $s$, 
\[ s<d(\gamma(s)):=min\{d(\gamma(s), z_0)| \  \text{for all zeros $z_0$ of $q_n$}\},\]
then there exists a constant unitary matrix $S$, not depending on the pair $(P, P')$, so that as $t\ra\infty$,
    \[T_{P,P'}(t)=\left(Id+O\left(t^{-\frac{1}{2n}}\right)\right)S\mtrx{
        e^{-Lt^{\frac{1}{n}}\mu_1}&&\\
        &\ddots&\\
        &&e^{-Lt^{\frac{1}{n}}\mu_n}}S^{-1}\]
where $\mu_j=2cos\left(\theta+\frac{2\pi {(j-1)}}{n}\right)$. 
\end{TheoremIntro}
\begin{Remark}
	The extra condition on the path is necessary for our method of proof, as the distance from the zeros of the holomorphic differential $q_n$ controls the decay rate of error terms. However, for sufficiently short paths, the extra condition is automatically satisfied. Thus, for each point $P$ away from the zeros of $q_n$, there is a neighborhood $\Uu_P$ for which all $|q_n|^\frac{2}{n}$-geodesics in $\Uu_P$ satisfy the extra condition. 
\end{Remark}

\begin{Remark}
 For $(\Sigma,(0,\cdots,0,tq_{n-1},0))\in Hit_n(S)$, we have similar results in Theorem \ref{subTransportAsymp}. In particular, in this case, $\mu_1=0$ and for $j>1$, $\mu_j=2cos\left(\theta+\frac{2\pi {(j-2)}}{n-1}\right)$.
When $P$ and $P'$ both project to the same point in $\Sigma,$ the projected path is a loop. In this case, the above asymptotics correspond to the values of the associated family of representations on the homotopy class of the loop. \end{Remark}

For $(\Sigma,(0,\cdots,tq_{2n}))\in Hit_{2n}(S),$ the eigenvalues in the parallel transport asymptotics Theorem \ref{TransportAsymp} come in pairs $\lambda,\lambda^{-1};$ thus, the corresponding representation is valued in $Sp(2n,\R).$ Similarly, for $(\Sigma,(0,\cdots,tq_{2n},0))\in Hit_{2n+1}(S)$ the eigenvalues in Theorem \ref{subTransportAsymp} come in pairs $\lambda,\lambda^{-1}$ with $1$ occurring as an eigenvalue; thus, the representation is valued in $SO(n,n+1).$
This is consistent with the description of the Hitchin components for $Sp(2n,\R)$ and $SO(n,n+1)$ as the zero locus of the differentials of odd degree in $Hit_{2n}(S)$ and $Hit_{2n+1}(S)$ mentioned above.

Also, the above eigenvalues are real and positive; this is a general phenomenon for a broader class of representations called Anosov representations, see \cite{AnosovFlows,Anosov}.
Although this is not new, such properties of Hitchin representations have not previously been observed using Higgs bundle techniques. \\

We now discuss the key ideas in the proof of the parallel transport asymptotics Theorem \ref{TransportAsymp}. \\

By the metric asymptotics Theorem \ref{metrictheorem}, the error estimates for the solution metric is $O\left(t^{-\frac{2}{n}}\right),$ these bounds are not strong enough to prove the parallel transport asymptotics Theorem \ref{TransportAsymp} directly. 
To overcome this, we prove very precise bounds for the matrix elements of the error term $R$ in the above ODE (\ref{PhieqIntro}), on a disk of radius $r$. More precisely, the $(k,l)$-entry is shown to have exponential decay 
\[R_{kl}\sim O\left(t^{-\frac{3}{2n}}e^{-2|1-\zeta\dwn{n}^{k-l}|t^\frac{1}{n}(r-|z|)}\right).\] 

To obtain such estimates, the entries of the error matrix are shown to be ``eigensolutions" for a cyclic Toda lattice, and these eigensolutions are estimated. 
The cyclic $SL(n,\C)$ Toda lattice is the following equation system
\eqtns{
\Delta d\up{1}=e^{d\up{1}-d\up{2}}-e^{d\up{n}-d\up{1}}\\
\Delta d\up{2}=e^{d\up{2}-d\up{3}}-e^{d\up{1}-d\up{2}}\\
\ \ \vdots\\
\Delta d\up{n}=e^{d\up{n}-d\up{1}}-e^{d\up{n-1}-d\up{n}}
}{}
Denote $\zeta\dwn{n}=e^{\frac{2\pi i}{n}},$ for $0\leq k\leq n-1,$ define the eigensolutions
\[w_k=\frac{1}{\sqrt{n}}\sum\limits_{i\in\Z_n}\zeta\dwn{n}^{ik}(d\up{i}-d\up{i+1}).\]
Using notation from the proof of the metric asymptotic Theorem \ref{metrictheorem}, set
$\tilde u\up{j}=u\up{j}-ln\left|tq_n\right|^\frac{n+1-2j}{n}.$ In terms of the $\tilde u\up{j}$'s, the Hitchin equations are
\eqtns{\Delta \tilde u\up{1}=4t^{\frac{2}{n}}(e^{\tilde u\up{1}-\tilde u\up{2}}-e^{-2\tilde u\up{1}})\\
\Delta \tilde u\up{2}=4t^{\frac{2}{n}}(e^{\tilde u\up{2}-\tilde u\up{3}}-e^{\tilde u\up{1}-\tilde u\up{2}})\\
\ \ \vdots\\
\Delta (-\tilde u\up{1})=4t^{\frac{2}{n}}(e^{-2\tilde u\up{1}}-e^{\tilde u\up{1}-\tilde u\up{2}})
}{ErrHitn}
As observed by Baraglia \cite{cyclichiggsaffinetoda}, this system is a cyclic Toda lattice given by $(d^1,\cdots,d^n)$ satisfying the extra symmetry $d^{n+1-i}=-d^{i}.$ Our analysis does not depend on this additional symmetry. In section \ref{ErrorToda}, the $(k,l)$-entry of $R$ is related with the eigensolution $w_{k-l}$ of the Toda lattice with $(\tilde u\up{1},\cdots,-\tilde u\up{1})=(d^1,\cdots,d^n)$. We show that the $(k,l)$-entry of the error term $R$ satisfies
\[R_{kl}=C(w_{k-l})_z+Ct^{-\frac{1}{n}}\Delta w_{k-l}\]
(equation (\ref{errortermcyclic_w_k}) below).  

In section \ref{ErrorEstimate}, we make use of maximum principle and an induction scheme repeatedly to prove estimates for the eigensolution $w_k,$ and hence, also obtain estimates of $\Delta w_k$ and $\p_z(w_k)$.   
The induction process relies on the following recursive formula on Toda lattice; since the recursive formula holds for a general cyclic Toda lattice, we believe it is of its own interest. 
\begin{PropositionIntro}[\ref{w_kProposition}](Recursive formula on cyclic Toda lattice)
The $w_k$'s satisfy a recursive formula
	\begin{equation*}
		\Delta w_k= \sum\limits_{\substack{r\equiv k\\ \text{mod}\ n}}\  \sum\limits_{r_1+\cdots+r_s=r} \frac{1}{s!n^{\frac{s-1}{2}}}\binom{r}{r_1,\cdots,r_s}|1-\zeta\dwn{n}^k|^2 w_{r_1}w_{r_2}\cdots w_{r_s}
	\end{equation*}
\end{PropositionIntro}
Finally, the parallel transport asymptotic Theorem \ref{TransportAsymp} is obtained by showing the $w_k$'s have the following asymptotics. 
\begin{TheoremIntro}[\ref{ErrorEstimateTheorem}]
 Let $d\up{i}$ be the error functions $\tilde u\up{i}$ in the Hitchin equations for the $n$-cyclic case, and define $w_k=\frac{1}{\sqrt{n}}\sum\limits_{i\in\Z_n}\zeta\dwn{n}^{ik}(d\up{i}-d\up{i+1})$ on the disk $D$ of radius $r$, then
	\[w_k(z)= O\left(t^{-\frac{3}{2n}}e^{-2|1-\zeta\dwn{n}^k|t^\frac{1}{n}(r-|z|)}\right).\] 
\end{TheoremIntro}
With the above theorem, it is not hard to show that, as $t\ra\infty,$ the Hitchin equation decouples. This is consistent with the asymptotic studies of \cite{EndsHiggs,TaubesAsymptotics3manifolds}.
\begin{CorollaryIntro} [\ref{decoupleCor}] For $\phi=\tilde e_1+tq_ne_{n-1}$, away from the zeros of $q_n,$ the Hitchin equation $F_{A_t}+[\phi,\phi^{*_{h_t}}]=0$ decouples as $t\ra\infty$
\[\begin{dcases}
	F_{A_t}=0\\
	[\phi,\phi^{*_{h_t}}]=0
\end{dcases}\] 
\end{CorollaryIntro}
\noindent The analogous corollary for $\phi=\tilde e_1+tq_{n-1}e_{n-2}$ is also valid. 

In section \ref{WKBsection}, generalizing Loftin's techniques in \cite{flatmetriccubicdiff,LoftinNotes}, we obtain an elementary proof of the special case of parallel transport asymptotic Theorem \ref{TransportAsymp} which only involves the highest eigenvalue, or WKB exponent, of $T_{P,P'}.$ 	

The family of metrics $h_t$ solving the Hitchin equations for $\phi=\tilde e_1+tq_n$, gives a family of $\rho_t$-equivariant harmonic maps 
\[f_t:\widetilde{\Sigma}\rightarrow SL(n,\R)/SO(n,\R)\subset (SL(n,\C)/SU(n),d),\]
where $d$ is the metric on the symmetric space.
Studying the asymptotics of the above maps provides the bridge between the above results, and the results of \cite{harmonicbuildingWKB} on the complex WKB problem (i.e. asymptotics of the family (\ref{CWKB}) of holomorphic flat connection $\nabla_t=\nabla_0+t\phi$).

Given two points $P,P'$ in the symmetric space $SL(n,\R)/SO(n,\R)$, the vector distance between them is defined by $\overset{\ra}{d}(P,P')=P-P',$ where the difference is taken in a \em{flat}\em{} (isometric to $\A^{n-1}$) containing both points. One can show $\overset{\ra}{d}(P,P')$ is independent of the choice of flat. In section \ref{HarmonicSection}, we show an asymptotic formula for the family of maps $f_t$ (Equation (\ref{asympft})) which implies the following theorem.

\begin{TheoremIntro}[\ref{harmMapTheorem}] With the same assumptions as the parallel transport asymptotitcs Theorem \ref{TransportAsymp}, for a path $\gamma$ satisfying
\[ s<d(\gamma(s)):=min\{d(\gamma(s), z_0)| \  \text{for all zeros $z_0$ of $q_n$}\},\]
we have
\[\lim\limits_{t\ra\infty}\frac{1}{t^\frac{1}{n}}\overset{\ra}{d}(f_t(\gamma(0)),f_t(\gamma(1)))=\left(-2L\cos\left(\theta\right), -2L\cos\left(\theta+\frac{2\pi }{n}\right),\dots,-2L\cos\left(\theta+\frac{2\pi {(n-1)}}{n}\right)\right).\]
\end{TheoremIntro}
A similar result holds for the $(n-1)$-cyclic case. In \cite{harmonicbuildingWKB}, a similar result is proven for the asymptotics of the complex WKB problem, i.e. the family of holomorphic flat connections $\nabla_t=\nabla_0+t\phi$ (\ref{CWKB}). Thus, Theorem \ref{harmMapTheorem} answers the conjecture in \cite{harmonicbuildingWKB} on the `Hitchin WKB problem' in the special cases where the Higgs bundle is in the Hitchin component and either $n$-cyclic or $(n-1)$-cyclic. 

To obtain better information about the behavior of the maps $f_t$ as $t\ra \infty,$ we rescale the metric on the symmetric space and consider the family of maps
\[f_t:\widetilde\Sigma\ra (SL(n,\R)/SO(n,\R),\frac{1}{t^\frac{1}{n}}d). \]

 By the work of Parreau \cite{CompactificationHit}, one can obtain a version of the limit map
\[f_{\omega}:\widetilde{\Sigma}\rightarrow Cone_{\omega},\] 
which is equivariant with respect to a limiting action $\rho_{\omega}$ of $\pi_1(S)$ on $Cone_{\omega}.$  The metric space $Cone_\omega$ is called the asymptotic cone of the symmetric space; it an affine building \cite{KeinerLeeb,Parreau2}.

In this language, the asymptotic expression (\ref{asympft}) of $f_t$ implies that for the families of rays 
\[(\Sigma,0,\cdots,0,tq_n),(\Sigma,0,\cdots,tq_{n-1},0)\in Hit_n(S)\] 
and for any $P$ away from the zeros of $q_n$ and $q_{n-1},$ there exists a neighborhood $\mathcal{U}_P$ so that the $\rho_{\omega}$-equivariant map
\[f_{\omega}:\widetilde{\Sigma}\rightarrow Cone_{\omega},\] 
sends $\mathcal{U}_P$ into a single apartment of the building $Cone_{\omega}$.


\smallskip

\textbf{Acknowledgements:}\ \ The authors would like to thank their advisors Steve Bradlow and Mike Wolf for encouraging this collaboration and for their instructive comments. We also thank John Loftin and Jakob Blaavand for their many helpful comments, and Andy Sanders and Daniele Alessandrini for numerous enlightening conversations.
Both authors acknowledge the support from U.S. National Science Foundation grants DMS 1107452, 1107263, 1107367 ``RNMS: GEometric structures And Representation varieties'' (the GEAR Network). Our collaboration would not have been possible without this support.

\section{Higgs bundle background and metric splitting}\label{HiggsSection}
The theory of Higgs bundles has been developed from many different perspectives; we will only give a brief review of the objects necessary for our results. 
Fix a Riemann surface structure $\Sigma$ on a closed surface $S$ of genus $g\geq2,$ denote the canonical bundle of $\Sigma$ by $K\ra\Sigma$ and fix a square root $K^\haf$ of the canonical bundle.

\begin{Definition} An $SL(n,\C)$-Higgs bundle over $\Sigma$ is a pair $(E,\phi),$ where $E\ra\Sigma$ is a rank $n$ holomorphic vector bundle with trivial determinant and $\phi\in H^0(\Sigma,End_0(E)\otimes K)$ is a holomorphic traceless $K$-twisted endomorphism.
\end{Definition}
To form the moduli space $\Mm_{Higgs}(SL(n,\C))$ of Higgs bundles, we need a notion of stability.
\begin{Definition} $(E,\phi)$ is \textbf{stable} if for any holomorphic subbundle $F\subset E$ with $\phi|_F:F\ra F\otimes K,$ we have $deg(F)<0;$ it is called \textbf{polystable} if it is a direct sum of stable Higgs bundles $\oplus(E_i,\phi_i).$ 
\end{Definition}
 A smooth $SL(n,\C)$ bundle isomorphism $g:E\ra E,$ also known as an $SL(n,\C)$-gauge transformation, acts on a Higgs bundle by pulling back both the holomorphic structure and the Higgs field. 
 The moduli space $\Mm_{Higgs}(SL(n,\C))$ consists of isomorphism classes of polystable $SL(n,\C)$-Higgs bundles.
\begin{Remark}
There is a $\C^*$ action on $\Mm_{Higgs}(SL(n,\C))$ given by $\lambda\cdot[(E,\phi)]=[(E,\lambda\phi)].$
A point $[(E,\phi)]$ is a fixed point of this action if for all $\lambda\in\C^*,$ there exists an automorphism $g_\lambda:E\ra E$ so that $Ad_{g_\lambda}\phi=\lambda\phi.$ 
We will be interested in the restriction of this action to two subgroups of $\C^*,$ these are $U(1)$ and the $k^{th}$ roots of unity $\langle \zeta\dwn{k}\rangle\subset U(1).$
\end{Remark}

One key ingredient in the nonabelian Hodge correspondence is the following theorem, proven by Hitchin \cite{selfduality} in the rank 2 case and Simpson in the general case \cite{localsystems}.

\begin{Theorem} Let $(E,\phi)$ be a stable $SL(n,\C)$-Higgs bundle, then there exists a unique hermitian metric $h,$ with Chern connection $A_h,$ solving the Hitchin equations
\begin{equation}\label{Hit}F_{A_h}+[\phi,\phi\up{*_h}]=0\end{equation}
where $F_{A_h}$ is the curvature of $A_h$ and $\phi\up{*_h}$ is the hermitian adjoint. Conversely,  if $(A_h,\phi)$ is a solution then the corresponding Higgs bundle is polystable Higgs bundle. 
\end{Theorem}
Such a solution gives rise to the flat $SL(n,\C)$-connection $A+\phi+\phi\up{*_h}.$ This gives a map to the representation variety
\[\Mm_{Higgs}(SL(n,\C))\lra \Rr(\pi_1(S),SL(n,\C)).\]

On a complex vector bundle $E,$ a Hermitian metric $h$ defines the unitary gauge group $\Gg_h,$ consisting of all bundle isomorphisms which preserve the metric. If we denote the Hermitian adjoint of a bundle isomorphism $g$ by $g^{*_h},$ then  
\[\Gg_h=\{g:E\ra E\ |\  g^{*_h}g=Id\}.\] 
It is well know, for instance see chapter 6 of \cite{DiffGeomCompVectBun}, that any two metrics $h$ and $h'$ on $E$ are related by $h'=hv$, where $v\in\Omega^0(\Sigma,End(E))$ is positive and self adjoint with respect to $h.$ Furthermore, the bundle endomorphism $v$ can be decomposed as $v=g^{*_h}g$, where $g$ is a $SL(n,\C)$-gauge transformation. This decomposition is unique up to a unitary gauge transformation.  
\begin{Remark}\label{MetricRemark}
	If we denote the holomorphic structure on $E$ by $\bar \p_E,$ then given a stable Higgs bundle $(\bar \p_E,\phi),$ the above theorem says there is a unique metric $h$ solving the Hitchin equations (\ref{Hit}). For any $SL(n,\C)$-gauge transformation $g,$ the pair $(g^{-1}\bar \p_E g,g^{-1}\phi g)$ also has a unique metric $h'$ solving (\ref{Hit}). The metrics $h$ and $h'$ are related by $h'=hg^{*_h}g.$ This follows from general gauge theoretic arguments, for example see section 3 of \cite{VorticesInLineBundles}; it will be crucial in the proof of Theorem \ref{MetricSplitting}.
\end{Remark}
\begin{Remark}\label{U1Remark}
	Note that if $h$ is a solution metric for a Higgs bundle $(E,\phi),$ then for all $\lambda\in U(1),$ $h$ is also the solution metric for $(E,\lambda\phi).$ This gives a $U(1)$ action on the moduli space of solutions to (\ref{Hit}). This action and its restriction to the $k^{th}$ roots of unity $\langle \zeta\dwn{k}\rangle$ play a key role in the proof of Theorem \ref{MetricSplitting}.
\end{Remark}

A generalization of the following example will be our main object of study. It was studied in detail in Hitchin's original paper \cite{selfduality}.
\begin{Example}\label{teichex}Consider the following family of stable $SL(2,\C)$-Higgs bundles
\[E=K^\haf\oplus K^{-\haf}\ \ \ \ \ \phi=\mtrx{0&q_2\\ 1&0}\]
where $q_2\in H^0(\Sigma,K^2)$ is a holomorphic quadratic differential. When $q_2=0,$ a solution to (\ref{Hit}) is equivalent to finding a hyperbolic metric on $\Sigma$ in the conformal class of the complex structure. Furthermore, using $q_2,$ all hyperbolic metrics can be found this way; this gives a parametrization of Teichm\"uller space by holomorphic quadratic differentials using Higgs bundles.
\end{Example}
\begin{Remark}
For any real reductive Lie group $G,$ there are corresponding definitions and theorems for $G$-Higgs bundles. We will only need to consider $SL(n,\R)$-Higgs bundles, for which we give a definition below. For the general set up see \cite{UpqHiggs,HiggsbundlesSP2nR}. The more complicated set up for real $G$ ensures that the flat connection, which arises from solving the Hitchin equations, has holonomy in the real group $G.$
\end{Remark}
 
\begin{Definition} An $SL(n,\R)$-Higgs bundle over $\Sigma$ is a triple $(E, Q,\phi),$ where $(E,\phi)$ is an $SL(n,\C)$-Higgs bundle and $Q$ is an orthogonal structure on $E$ with the property that $\phi$ is $Q$-symmetric, i.e., $\phi^TQ=Q\phi.$
\end{Definition}

Example \ref{teichex} is actually an $SL(2,\R)$-Higgs bundle. To see this, consider the orthogonal structure  
\[Q=\mtrx{0&1\\1&0}:\xymatrix@=1.5em{K^\haf\oplus K^{-\haf}\ar[r]& K^{-\haf}\oplus K^{\haf}}\] 
thought of as a symmetric isomorphism $Q:E\ra E^*.$

For $SL(n,\R)$-Higgs bundle, the data of $Q$ gives the bundle $E$ an $SO(n,\C)$-structure. By an isomorphism of a $SL(n,\R)$-Higgs bundle, we will mean an $SO(n,\C)$-gauge transformation, i.e., a bundle isomorphism of $E$ which preserves this additional structure.

Let $p_2,\dots,p_{n}$ be a homogeneous basis for the $SL(n,\C)$-invariant polynomials $\C[\fsl(n,\C)]^{SL(n,\C)},$ with $deg(p_j)=j.$ Such a choice of basis defines a map (called the Hitchin fibration)
\[h:\Mm_{Higgs}(SL(n,\C)) \longrightarrow\bigoplus\limits_{j=2}^nH^0(\Sigma, K^j)\]
given by		
\[h([E,\phi])=(p_2(\phi),\dots,p_n(\phi)).\]

In \cite{liegroupsteichmuller}, Hitchin defines a section $s_h$ of this fibration whose image consists of stable Higgs bundles with corresponding flat connections having holonomy in $SL(n,\R).$
Furthermore, the section $s_h$ maps surjectively to the connected component of the $SL(n,\R)$-Higgs bundle moduli space which naturally contains an embedded copy of Teichm\"uller space (example \ref{teichex}).  Let
\[(E,Q)=S^{n-1}\left(K^\haf\oplus K^{-\haf},\mtrx{0&1\\1&0}\right)=\left(K\up{\frac{n-1}{2}}\oplus K\up{\frac{n-3}{2}}\oplus\dots\oplus K^{-\frac{n-3}{2}}\oplus K\up{-\frac{n-1}{2}},\mtrx{&&1\\ &\iddots&\\ 1&&}\right)\] 
be the $(n-1)$'st symmetric power of example \ref{teichex},
and $(q_2,q_3,\dots, q_n)\in\bigoplus\limits_{j=2}^nH^0(\Sigma, K^j).$ The Hitchin section is defined by
\[s_h(q_2,q_3,\dots,q_n)=\left [E,\mtrx{&&1\\&\iddots&\\1&&},\mtrx{
0&\haf(n-1)q_2 &q_3 & q_4& \dots&q_{n-1}&q_{n}\\
1&0 &\haf(n-3)q_2 &q_3&\dots&q_{n-2}&q_{n-1}\\ 
&\ddots&\ddots&\ddots&&&\\
&&&& 0&\haf(n-3)q_2 &q_3\\
&&&&1 &0 &\haf(n-1)q_2\\
&&&&&1 &0}\right].\]

The embedded copy of Teichm\"uller space results from setting $q_3=\dots=q_n=0$; it arises as the $(n-1)$st symmetric power of example \ref{teichex}. Through Kostant's work \cite{ptds} on the principal three-dimensional subalgebra, there exists a homogeneous basis $\{p_2,\dots,p_n\}$ of the invariant polynomials so that $p_j(\phi)=q_j,$ verifying that $s_h$ is a section. Because of its link with the principal three-dimensional subalgebra, we will denote the Higgs field associated to $s_h(q_2,q_3,\dots,q_n)$ by 
\[\phi=\tilde e_1+q_2e_1+q_3e_2+\dots +q_ne_{n-1}.\]
The constants on the $q_2$'s are necessary to make $\langle\tilde e_1,e_1,[e_1,\tilde e_1]\rangle$ a Lie subalgebra isomorphic to $\fsl(2,\C).$ 

We now prove the metric splitting theorem which generalizes results of Baraglia \cite{g2geometry}. 
\begin{Theorem}\label{MetricSplitting} Let $(E,\phi)$ be a Higgs bundle in the Hitchin component $Hit_n(S)$ with \[\phi=\tilde e_1+\sum\limits_{j=0\ mod\ k}q_je_{j-1}.\] Then the metric solving the Hitchin equation (\ref{Hit}) splits as a direct sum metric on $E=E_1\oplus\cdots\oplus E_k$, where
	\[E_j=K^\frac{n+1-2j}{2}\oplus K^{\frac{n+1-2j}{2}-k}\oplus K^{\frac{n+1-2j}{2}-2k}\oplus\cdots .  \] 
\end{Theorem}
\begin{proof}
Let $\zeta\dwn{k}=e^\frac{2\pi i}{k}$ and consider the gauge transformation of $E=K^\frac{n-1}{2}\oplus\dots\oplus K^{-\frac{n-1}{2}}$ given by 
\begin{equation}\label{gaugetrans}g\dwn{k}=\mtrx{
\zeta\dwn{2k}^{1-n}&&&\\
&\zeta^{3-n}\dwn{2k}&& \\
&& \ddots&\\
&&&\zeta^{n-1}\dwn{2k}}:E\lra E\end{equation}
Note that $g\dwn{k}^TQg\dwn{k}=Q,$ so $g\dwn{k}$ is indeed an $SO(n,\C)$-gauge transformation. 
The action of $g\dwn{k}$ on the Higgs field is  
\[Ad_{g\dwn{k}}(\tilde e_1+\sum\limits_{j=2}^nq_je_{j-1})=\zeta\dwn{k}\tilde e_1+\sum\limits_{j=2}^n\zeta\dwn{k}^{1-j}q_je_{j-1}.\]
Thus, for \[\phi=\tilde e_1+\sum\limits_{j=0\ mod\ k}q_je_{j-1},\] we have $Ad_{g\dwn{k}}\phi=\zeta\dwn{k}\phi,$ so 
$(E,Q,\phi)$ is a fixed point of the $k^{th}$ roots of unity action on $\Mm_{Higgs}(SL(n,\R)).$ 

One checks that, with respect to $g\dwn{k},$ the eigenbundle decomposition $E=E_1\oplus\cdots E_k$ is given by
\[E_j=K^{\frac{n+1-2j}{2}}\oplus K^{\frac{n+1-2j}{2}-k}\oplus K^{\frac{n+1-2j}{2}-2k}\oplus\cdots .\]
To see that the metric $h$ splits, we will show the gauge transformation $g\dwn{k}$ is unitary, that is $g\dwn{k}^*g\dwn{k}=Id.$ Since the triple $(\bar\p_E,\phi,h)$ solves the Hitchin equations (\ref{Hit}), by remark \ref{MetricRemark}, the triple $(g\dwn{k}^{-1}\bar \p_E g\dwn{k},g\dwn{k}^{-1}\phi g\dwn{k},hg\dwn{k}^*g\dwn{k})$ also solves (\ref{Hit}). We have computed 
\[(g\dwn{k}^{-1}\bar \p_E g\dwn{k},g\dwn{k}^{-1}\phi g\dwn{k})=(\bar \p_E,\zeta\dwn{k}\phi),\]
thus $(\bar \p_E,\zeta\dwn{k}\phi,hg\dwn{k}^*g\dwn{k})$ solves (\ref{Hit}) as well. Now, using the $U(1)$ action and remark \ref{U1Remark}, the triple $(\bar \p_E,\phi,hg\dwn{k}^*g\dwn{k})$ solves (\ref{Hit}). By uniqueness of the metric, 
\[h=hg\dwn{k}^*g\dwn{k}\]
proving that $g\dwn{k}$ is unitary. Since $g\dwn{k}$ is both unitary and preserves the eigenbundle splitting $E_1\oplus\cdots\oplus E_k$, the metric $h$ splits as $h_1\oplus\cdots\oplus h_k.$ 
\end{proof}

So far, we have only used $SL(n,\C)$ properties of the Higgs bundles in the Hitchin component, we now use the $SL(n,\R)$ nature of the Hitchin component to further constrain the metric. Since $h$ is a metric on an orthogonal bundle, it gives a reduction of structure from $SO(n,\C)$ to $SO(n,\R).$ As a result, it must be $Q$-orthogonal, that is \[h^TQh=Q.\]
This leads to the following corollary, which is known for $k=n$ \cite{g2geometry}.
\begin{Corollary}\label{CorMetricSplit}
For $k=n$ and $k=n-1,$ the Higgs fields are $\phi=\tilde e_1+q_ne_{n-1}$ and $\phi=\tilde e_1+q_{n-1}e_{n-2}$, and the harmonic metric splits as 
\[h_1\oplus h_2\oplus\dots\oplus h_2^{-1}\oplus h_1^{-1}\] 
on the direct sum of line bundles 
\[K^{\frac{n-1}{2}}\oplus K^{\frac{n-3}{2}}\oplus\dots\oplus K^{-\frac{n-3}{2}}\oplus K^{-\frac{n-1}{2}}.
\]
Here $h_j^{-1}$ denotes the induced metric on the dual bundle. 
	
\end{Corollary}
\begin{proof}
When $k=n$, $\phi=\tilde e_1+q_ne_{n-1}$ is a fixed point of the $n^{th}$ roots of unity action. By Theorem \ref{MetricSplitting}, the original holomorphic decomposition is the eigenbundle decomposition of (\ref{gaugetrans}), and $h=h_1\oplus\dots\oplus h_n.$  The constraint coming from $Q$ is 
\[\mtrx{h_1&&&\\&h_2&&\\&&\ddots&\\&&&h_n}\mtrx{&&&1\\&&1&\\&\iddots&&\\1&&&}\mtrx{h_1&&&\\&h_2&&\\&&\ddots&\\&&&h_n}=\mtrx{&&&1\\&&1&\\&\iddots&&\\1&&&},\]
thus $h_1h_n=h_2h_{n-1}=\dots=1$ and the metric splits as $h=h_1\oplus h_2\oplus\dots\oplus h_2^{-1}\oplus h_1^{-1}.$
When $k=n-1$, $\phi=\tilde e_1+q_{n-1}e_{n-2}$ is a fixed point of the $(n-1)^{st}$ roots of unity action, and the eigenbundle splitting of theorem \ref{MetricSplitting} is
\[(K^\frac{n-1}{2}\oplus K^{-\frac{n-1}{2}})\oplus K^\frac{n-3}{2}\oplus K^\frac{n-5}{2}\oplus\dots \oplus K^{-\frac{n-5}{2}}\oplus K^{-\frac{n-3}{2}}.\]
The condition $h^TQh=Q$ is 
\[\mtrx{h_1^{11}&h_1^{21}&&&\\h_1^{12}&h_1^{22}&&&\\&&h_2&&\\&&&\ddots&\\&&&&h_{n-1}}
\mtrx{0&1&&&\\1&0&&&\\&&&&1\\&&&\iddots&\\&&1&&}\mtrx{h_1^{11}&h_1^{12}&&&\\h_1^{21}
&h_1^{22}&&&\\&&h_2&&\\&&&\ddots&\\&&&&h_{n-1}}=\mtrx{0&1&&&\\1&0&&&\\&&&&1\\&&&\iddots&\\&&1&&}\]
thus $h_2h_{n-1}=h_3h_{n-2}=\dots=1.$ The constraint $det(h)=1$ together with 
\[\mtrx{h_1^{11}&h_1^{21}\\h_1^{12}
&h_1^{22}}\mtrx{0&1\\1&0}\mtrx{h_1^{11}
&h_1^{12}\\h_1^{21}
&h_1^{22}}=\mtrx{0&1\\1&0}\]
implies \[\mtrx{h_1^{11}&h_1^{12}\\h_1^{21}&h_1^{22}}=\mtrx{h_1^{11}&0\\0&(h_1^{11})^{-1}}\]
So, in the original splitting $E=K^\frac{n-1}{2}\oplus K^\frac{n-3}{2}\oplus \dots \oplus K^{-\frac{n-3}{2}}\oplus K^{-\frac{n-1}{2}},$ the metric splits as 
\[h=h_1\oplus h_2\oplus \dots \oplus h_2^{-1}\oplus h_1^{-1}.\]
\end{proof}

\begin{Remark}
For $\phi=\tilde e_1+q_{n-2}e_{n-3},$ the eigenbundle splitting from Theorem \ref{MetricSplitting} is 
\[(K^\frac{n-1}{2}\oplus K^{-\frac{n-3}{2}})\oplus(K^\frac{n-3}{2}\oplus K^{-\frac{n-1}{2}})\oplus K^\frac{n-5}{2}\oplus \dots \oplus K^{-\frac{n-5}{2}}\]
the metric splits as $h=h_1\oplus h_2\oplus h_3\oplus\dots\oplus h_{n-2}$  where $h_1$ and $h_2$ are metrics on rank 2 bundles. The condition $h^TQh=Q$ tells us $h_3h_{n-2}=\dots=1$ and 
$h_1^T\mtrx{0&1\\1&0}h_2=\mtrx{0&1\\1&0},$ which does not imply the metric splits on line bundles. Thus, looking at fixed points for smaller $k$ gives less information about the metric.
\end{Remark}

Since $\phi^*=h^{-1}\overline{\phi}^Th$, one checks that with respect to the eigenbundle splitting of Theorem \ref{MetricSplitting}, fixed points of $\langle\zeta\dwn{k}\rangle$ in the Hitchin component have the following form
\[\phi=\mtrx{0&&&\phi_k\\\phi_1&0&&\\&\ddots&\ddots&\\&&\phi_{k-1}&0}\ \ \ \ \ \ \ \ \text{and} \ \ \ \ \ \ \ \phi^{*_h}=\mtrx{0&\phi_1^*&&\\&&\ddots&\\&&&\phi_{k-1}^*\\\phi_k^*&&&0}\] 
with $\phi_j:E_j\ra E_{j+1}\otimes K$ and $\phi_j^*:E_{j+1}\ra E_j\otimes \overline{K},$ where $j+1$ is taken $mod\ k.$ The adjoint $\phi_j^*$ is defined by $\phi_j^*=h_{j}^{-1}\overline{\phi_j}^Th_{j+1}.$   

Thus, for fixed points of $\langle\zeta_k\rangle,$ the equation (\ref{Hit})
simplifies to the following system of coupled equations
\begin{equation}\label{fixedeq}
F_{A_j}+\phi_{j-1}\wedge\phi_{j-1}^*+\phi_j^*\wedge\phi_j=0.\end{equation}
These equations are a special case of the twisted quiver bundle equations considered in \cite{KHCquiversvortices}. 

The above results have generalizations to fixed points of the $k^{th}$ roots of unity actions on the moduli space of $G$-Higgs bundles, which we call $k$-cyclic Higgs bundles. 
This is the topic of the first author's thesis \cite{mythesis}.

\section{Equations, flat connections and metric asymptotics}\label{MetricAsympSection}
In this section, Corollary \ref{CorMetricSplit} will be used to write the Hitchin equations as a system of $\lfloor \frac{n}{2}\rfloor$ fully coupled nonlinear elliptic equations, and to give an explicit description of the corresponding flat connections. After this, we prove the main theorem concerning the asymptotics of the metric solving the Hitchin equations. The proof is quite long and can be skipped on first reading. Finally, an important bound on the metric's first derivative is proved.

There are slight differences when $n$ is even compared to when $n$ is odd. We will always work in the even case and mention what the differences are for the odd case. One obvious difference in the odd case is the middle line bundle of $E$ is a trivial bundle; for both $\phi=\tilde e_1+q_ne_{n-1}$ and $\phi=\tilde e_1+q_{n-1}e_{n-2},$ the metric on the trivial line bundle is the standard one on $\C$. 

\subsection{Equations}
Since the metric splits as $h=h_1\oplus h_2\oplus\dots\oplus h_2^{-1}\oplus h_1^{-1},$ the adjoints of the Higgs fields $\phi=\tilde e_1+q_ne_{n-1}$ and $\phi=\tilde e_1+q_{n-1}e_{n-2}$ are respectively
\[\phi^*=\mtrx{0&h\dwn{1}^{-1}h\dwn{2}&0&&&\\
&&h\dwn{2}^{-1}h\dwn{3}&&&\\
&&&&\ddots&
\\
0&&&&0&h\dwn{1}^{-1}h\dwn{2}\\
h\dwn{1}^{2}\bar q_n&0&&&&0}\ \ \ \ \ \ 
\phi^*=\mtrx{0&h\dwn{1}^{-1}h\dwn{2}&0&&&\\
&&h\dwn{2}^{-1}h\dwn{3}&&&\\
&&&&\ddots&
\\
h\dwn{1}h_2\bar q_{n-1}&&&&0&h\dwn{1}^{-1}h\dwn{2}\\
0&h\dwn{1}h\dwn{2}\bar q_{n-1}&&&&0}\]

We are interested in the corresponding family of flat connections as the differentials $q_n$ and $q_{n-1}$ are scaled by a real parameter $t$. Using (\ref{fixedeq}), the Hitchin equations for $n$-cyclic Higgs field $\phi=\tilde e_1+tq_ne_{n-1}$ become:
\eqtns{
F\dwn{A_1}+t^2h\dwn{1}^2q_n\wedge \bar q_n - h\dwn{1}^{-1}h\dwn{2}=0\\
F\dwn{A_j} + h_{j-1}^{-1}h\dwn{j} - h\dwn{j}^{-1}h\dwn{j+1}=0 & 1<j<\frac{n}{2}\\
F\dwn{A_{\frac{n}{2}}}+h\dwn{\frac{n}{2}-1}^{-1}h\dwn{\frac{n}{2}}-h_{\frac{n}{2}}^{-2}=0
}{}
Here all the metrics, and hence, all the curvature forms depend on $t.$ We will suppress the $t$ dependence from the notation.
When $n$ is odd, the last equation is changed to $F\dwn{A_{\frac{n-1}{2}}}+h\dwn{\frac{n-1}{2}-1}^{-1}h\dwn{\frac{n-1}{2}}-h_{\frac{n-1}{2}}^{-1}=0.$

To understand the flat connection we choose a local coordinate $z$ on $\Sigma.$ Such a choice gives a local holomorphic frame $(s_1,s_2,\dots,s_2^*,s_1^*)$ for $E,$
where $s_j=dz^{\frac{n+1-2j}{2}}$ is the local frame of $K^\frac{n+1-2j}{2}$ induced by the coordinate $z.$ With respect to this choice of coordinates, the Higgs field is locally given by
\[\phi=\mtrx{0&&&f\dwn{n}
\\
1&&&\\
&\ddots&&\\
&&1
&0
}\]
where $q_n=f\dwn{n}dz^n,$ for some function $f\dwn{n}.$

With respect to this frame, locally represent the metric $h_{j}$ by $e^{-\lambda\up{j}},$ here the $j$ is a superscript and \textbf{not} an exponent. 
Recall that in a holomorphic frame, the Chern connection has connection 1-form $A=H^{-1}\p H$ and curvature 2-form given by $F_A=\bar\p (H^{-1}\p H)$. Since $h_j$ is a metric on  a line bundle, the expressions simplify to   
\[A_j=-\lambda\up{j}\dwn{z}dz\ \ \ \ \ \ \text{and}\ \ \ \ \ \ F\dwn{A_j}=\lambda\up{j}\dwn{z\bz}dz\wedge d\bz.\]
The equations may be rewritten as:
\eqtns{
\lambda\up{1}\dwn{z\bz}+t^2e^{-2\lambda\up{1}}|q_n|^2 - e^{\lambda\up{1}-\lambda\up{2}}=0\\
\lambda\up{j}\dwn{z\bz} + e^{\lambda\up{j-1}-\lambda\up{j}} - e^{\lambda\up{j}-\lambda\up{j+1}}=0 & 1<j<\frac{n}{2}\\
\lambda\up{\frac{n}{2}}\dwn{z\bz} + e^{\lambda\up{\frac{n}{2}-1}-\lambda\up{\frac{n}{2}}} - e^{2\lambda\up{\frac{n}{2}}}=0
}{}
Similarly for $(n-1)$-cylcic Higgs field $\phi=\tilde e_1+tq_{n-1} e_{n-2},$ we may rewrite the Hitchin equations as
\eqtns{
\lambda\up{1}\dwn{z\bz}+t^2e^{-\lambda\up{1}-\lambda\up{2}}|q_{n-1}|^2 - e^{\lambda\up{1}-\lambda\up{2}}=0\\
\lambda\up{2}\dwn{z\bz}+t^2e^{-\lambda\up{1}-\lambda\up{2}}|q_{n-1}|^2 +e^{\lambda\up{1}-\lambda\up{2}}-e^{\lambda\up{2}-\lambda\up{3}}=0\\
\lambda\up{j}\dwn{z\bz} + e^{\lambda\up{j-1}-\lambda\up{j}} - e^{\lambda\up{j}-\lambda\up{j+1}}=0 & 2<j<\frac{n}{2}\\
\lambda\up{\frac{n}{2}}\dwn{z\bz} + e^{\lambda\up{\frac{n}{2}-1}-\lambda\up{\frac{n}{2}}} - e^{2\lambda\up{\frac{n}{2}}}=0
}{}
Again, in the odd case, the last equation is changed to $\lambda\up{\frac{n-1}{2}}\dwn{z\bz} + e^{\lambda\up{\frac{n-1}{2}-1}-\lambda\up{\frac{n-1}{2}}} - e^{\lambda\up{\frac{n-1}{2}}}=0.$

\subsection{Flat connections}
The flat connection is given by $D=A_h+\phi+\phi^*.$ If, in the holomorphic frame $(s_1,\dots,s_{\frac{n}{2}},s_{\frac{n}{2}}^*,\dots,s_1^*),$ we have 
\begin{equation}\label{locqb}
	q_n=f\dwn{n}dz^n\ \ \ \text{and}\ \ \ \ q_{n-1}=f\dwn{n-1}dz^{n-1},
\end{equation}
 then  the flat connection for the $n$-cyclic $\phi=\tilde e_1+tq_ne_{n-1}$ is given by 
\begin{equation}\label{flat1}D=\mtrx{
-\lambda\up{1}_zdz&&&&tf\dwn{n}\\
1&-\lambda\up{2}_zdz&&&\\
&\ddots&\ddots&&\\
&&1&\lambda_z\up{2}dz&0\\
&&&1&\lambda_z\up{1}dz
}
+\mtrx{0&e^{\lambda\up{1}-\lambda\up{2}}&&&\\
0&0&e^{\lambda\up{2}-\lambda\up{3}}&&\\
&&&\ddots&\\
&&&&e^{\lambda\up{1}-\lambda\up{2}}\\
te^{-2\lambda\up{1}}\bar f\dwn{n}&&&&0}
,\end{equation}
and the flat connection for the $(n-1)$-cyclic $\phi=\tilde e_1+tq_{n-1}e_{n-2},$ is
\begin{equation}\label{flat2}D=\mtrx{
-\lambda\up{1}_zdz&&&tf\dwn{n-1}&0\\
1&-\lambda\up{2}_zdz&&&tf\dwn{n-1}\\
0&1&-\lambda\up{3}_zdz&&\\
&&&\ddots&\\
&&&1&\lambda_z\up{1}dz
}
+\mtrx{0&e^{\lambda\up{1}-\lambda\up{2}}&&&\\
0&0&e^{\lambda\up{2}-\lambda\up{3}}&&\\
&&&\ddots&\\
te^{-\lambda\up{1}-\lambda\up{2}}\bar f\dwn{n-1}&&&&e^{\lambda\up{1}-\lambda\up{2}}\\
0&te^{-\lambda\up{1}-\lambda\up{2}}\bar f\dwn{n-1}&&&0}
.\end{equation}

We want to calculate the behavior of the flat connection in the limit $t\ra\infty.$ To do so, we need to understand the asymptotics of the $\lambda\up{j}$'s and the  asymptotics of their first derivatives $\lambda\up{j}_z$. In order to use the maximum principle, we will make a change of variables.
Let $\Omega_n\subset\Sigma$ be a compact set away from the zeros of $q_n$ and fix a background metric $g\dwn{n}$ on $\Sigma$ with the following properties:
\eqtns{g\dwn{n}=|q_n|^{\frac{2}{n}} & \text{on}\ \ \Omega_n\\
\dfrac{|q_n|^2}{(g\dwn{n})^n}\leq 1 & \text{on}\ \ \Sigma}{}
Using this metric, we make the following change of variables:
\[u\up{j}=\lambda\up{j}-\frac{n+1-2j}{2}\ln(g\dwn{n}).\]

For $\phi=\tilde e_1+q_{n-1}e_{n-2},$ we define the analogous compact set $\Omega_{n-1}$ and background metric $g\dwn{n-1}$ with the property 
\eqtns{g\dwn{n-1}=|q_{n-1}|^{\frac{2}{n-1}} & \text{on}\ \ \Omega_{n-1}\\
\dfrac{|q_{n-1}|^2}{(g\dwn{n-1})^{n-1}}\leq 1 & \text{on}\ \ \Sigma}{}
Using $g\dwn{n-1},$ we make the change of variables
\[v\up{j}=\lambda\up{j}-\frac{n+1-2j}{2}\ln(g\dwn{n-1}).\]

Recall that the Laplace-Beltrami operator of a conformal metric $g$ on a Riemann surface is given by $\Delta_g=\frac{4}{g}\p_{z\bz}$ and the scalar curvature is 
\[K_g=-\haf \Delta_g\ln(g)=-\frac{2}{g}\p_{z\bz}\ln(g).\] 
Because $q_n$ and $q_{n-1}$ are holomorphic, $K_{g\dwn{n}}=0=K_{g\dwn{n-1}}$ on $\Omega_n$ and $\Omega_{n-1}.$ 

With respect to $u\up{j},$ the equations for $\phi=\tilde e_1+tq_ne_{n-1}$ become
\eqtns{
(u\up{1}+\frac{n-1}{2}\ln(g\dwn{n}))\dwn{z\bz}+t^2e^{-2u\up{1}-(n-1)\ln(g\dwn{n})}|q_n|^2- e^{u\up{1}-u\up{2}+\ln(g\dwn{n})}=0\\
(u\up{j}+\frac{n+1-2j}{2}\ln(g\dwn{n}))\dwn{z\bz} + e^{u\up{j-1}-u\up{j}+\ln(g\dwn{n})} - e^{u\up{j}-u\up{j+1}+\ln(g\dwn{n})}=0 & 1<j<\frac{n}{2}\\
(u\up{\frac{n}{2}}+\haf \ln(g\dwn{n}))\dwn{z\bz} + e^{u\up{\frac{n}{2}-1}-u\up{\frac{n}{2}}+\ln(g\dwn{n})} - e^{2u\up{\frac{n}{2}}+\ln(g\dwn{n})}=0
}{}
Using our knowledge of $K_{g\dwn{n}}$ and $\Delta_{g\dwn{n}},$ we rewrite the equations as
\eqtns{
-\frac{1}{4}\Delta_{g\dwn{n}}u\up{1}=-\frac{n-1}{4}K_{g\dwn{n}}+\dfrac{t^2|q_n|^2}{g\dwn{n}^n}e^{-2u\up{1}}- e^{u\up{1}-u\up{2}}\\
-\frac{1}{4}\Delta_{g\dwn{n}}u\up{j}=-\frac{n+1-2j}{4}K_{g\dwn{n}} + e^{u\up{j-1}-u\up{j}} - e^{u\up{j}-u\up{j+1}} & 1<j<\frac{n}{2}\\
-\frac{1}{4}\Delta_{g\dwn{n}}u\up{\frac{n}{2}}=-\frac{1}{4}K_{g\dwn{n}} + e^{u\up{\frac{n}{2}-1}-u\up{\frac{n}{2}}} - e^{2u\up{\frac{n}{2}}}
}
{neq}

We will show 
\[\lim\limits_{t\ra\infty}e^{u\up{j}}= t^\frac{n+1-2j}{n}\ \ \ \  \  1\leq j\leq\frac{n}{2}.\]
Similarly, in terms of the $v\up{j}$'s, the equations for $\phi=\tilde e_1+tq_{n-1}e_{n-2}$ become 
\eqtns{ 
-\frac{1}{4}\Delta_{g\dwn{n-1}}v\up{1}=-\frac{n-1}{4}K_{g\dwn{n-1}}+\dfrac{t^2|q_{n-1}|^2}{g\dwn{n-1}^{n-1}}e^{-v\up{1}-v\up{2}}- e^{v\up{1}-v\up{2}}\\
-\frac{1}{4}\Delta_{g\dwn{n-1}}v\up{2}=-\frac{n-1}{4}K_{g\dwn{n-1}}+\dfrac{t^2|q_{n-1}|^2}{g\dwn{n-1}^{n-1}}e^{-v\up{1}-v\up{2}}+ e^{v\up{1}-v\up{2}}-e^{v\up{2}-v\up{3}}\\
-\frac{1}{4}\Delta_{g\dwn{n-1}}v\up{j}=-\frac{n+1-2j}{4}K_{g\dwn{n-1}} + e^{v\up{j-1}-v\up{j}} - e^{v\up{j}-v\up{j+1}} & 2<j<\frac{n}{2}\\
-\frac{1}{4}\Delta_{g\dwn{n-1}}v\up{\frac{n}{2}}=-\frac{1}{4}K_{g\dwn{n-1}} + e^{v\up{\frac{n}{2}-1}-v\up{\frac{n}{2}}} - e^{2v\up{\frac{n}{2}}}
}{n-1eq}
 
In this case, it will be shown that
\[\lim\limits_{t\ra\infty}e^{v\up{1}} = t\]
\[\lim\limits_{t\ra\infty}e^{v\up{j}} = (2t)^\frac{n+1-2j}{n-1}\ \ \ \ \ 1<j\leq\frac{n}{2}\]

\subsection{Estimates on asymptotics  of $\lambda\up{j}$ and $\lambda\up{j}_z$}
In order to understand the asymptotics of the family of flat connections above, we need to understand the asymptotics of the metric and its first derivative. For the metric, we have the following theorem. 

\begin{Theorem}\label{metrictheorem}
	For every point $p\in\Sigma$ away from the zeros of $q_n$ or $q_{n-1},$ as $t\ra\infty$
\begin{enumerate}[1.]
		\item For $(\Sigma,\ \tilde e_1+tq_ne_{n-1})\in Hit_n(S),$ the metric $h_j(t)$ on $K^\frac{n+1-2j}{2}$ admits the expansion
		\begin{equation*}
		h_j(t)=(t|q_n|)^{-\frac{n+1-2j}{n}}\left(1+O\left(t^{-\frac{2}{n}}\right)\right) \ \ \ \  \text{for\ all\  } j
		\end{equation*}
		\item For $(\Sigma,\ \tilde e_1 +tq_{n-1}e_{n-2})\in Hit_n(S),$ the metric $h_j(t)$ on $K^\frac{n+1-2j}{2}$ admits the expansion
		\begin{equation*}
			h_j(t)=\begin{dcases}
				(t|q_{n-1}|)^{-\frac{n+1-2j}{n-1}}\left(1+O\left(t^{-\frac{2}{n-1}}\right)\right) & \text{for\ } j=1\ \ \text{and}\ \ j=n\\
				(2t|q_{n-1}|)^{-\frac{n+1-2j}{n-1}}\left(1+
				O\left(t^{-\frac{2}{n-1}}\right)\right)
				& \text{for\ } 1< j <n
			\end{dcases}
		\end{equation*}
		\end{enumerate}
\end{Theorem}

Unfortunately, the proof is long and technical, and so may be skipped on first reading. 
\begin{proof}
While the statement of the theorem for our two cases is similar, due to differences in the equations, the details of the proofs are different. 
We start with a series of lemmas. 
The following notation will be used:
\[M_n=\frac{1}{4}\underset{\Sigma}{max}\ |K_{g\dwn{n}}|\ \ \ \ \ \ \ A_j=\underset{\Sigma}{max}\ e^{u\up{j}}\ \ \ \ \ \ \ B_j=\underset{\Sigma}{max}\ e^{u\up{j}-u\up{j+1}}\]
\[M_{n-1}=\frac{1}{4}\underset{\Sigma}{max}\ |K_{g\dwn{n-1}}|\ \ \ \ \ \ \ D_j=\underset{\Sigma}{max}\ e^{v\up{j}}\ \ \ \ \ \ \ E_j=\underset{\Sigma}{max}\ e^{v\up{j}-v\up{j+1}},\]
note that $\dfrac{A_j}{A_{j+1}}\leq B_j$ and $\dfrac{D_j}{D_{j+1}}\leq E_j.$ 

The signs in our equations are set up for applications of the maximum principle.
At the maxima of $e^{u\up{j}},$ we have $0\leq -\frac{1}{4}\Delta_gu\up{j}.$ After rearranging the equations (\ref{neq}) we get the following estimates at the maxima of $e^{u\up{j}}$ respectively
\eqtns{
0\leq -\frac{1}{4}\Delta_{g\dwn{n}}u\up{1}\leq (n-1)M_n+t^2A_1^{-2}-\dfrac{A_1}{A_2}
\\
0\leq -\frac{1}{4}\Delta_{g\dwn{n}}u\up{j}\leq(n+1-2j)M_n+\dfrac{A_{j-1}}{A_j}-\dfrac{A_j}{A_{j+1}} & 1<j<\frac{n}{2}
\\
0\leq -\frac{1}{4}\Delta_{g\dwn{n}}u\up{\frac{n}{2}}\leq M_n+\dfrac{A_{\frac{n-1}{2}}}{A_\frac{n}{2}}-A_{\frac{n}{2}}^2.
}{maxuj}
Similarly, using equations (\ref{n-1eq}), at the maxima of $e^{v\up{j}}$ respectively
\eqtns{
0\leq -\frac{1}{4}\Delta_{g\dwn{n-1}}v\up{j}\leq(n+1-2j)M_{n-1}+\dfrac{D_{j-1}}{D_j}-\dfrac{D_j}{D_{j+1}} & 2<j<\frac{n}{2}
\\
0\leq -\frac{1}{4}\Delta_{g\dwn{n-1}}v\up{\frac{n}{2}}\leq(n+1-2j)M_{n-1}+\dfrac{D_{\frac{n}{2-1}}}{D_\frac{n}{2}}-D_{\frac{n}{2}}^2.}{maxvj}

\begin{Remark} We \textit{cannot} write 
\[0\leq -\frac{1}{4}\Delta_{g\dwn{n-1}}v\up{1}\leq (n-1)M_{n-1}+\dfrac{t^2}{D_1D_2}-\dfrac{D_1}{D_2},\]
 nor the corresponding inequality for $\Delta_{g\dwn{n-1}}v\up{2}$ because at the maximum of $e^{v\up{1}},$ it may be the case that $e^{-v\up{1}-v\up{2}}$ is larger than $(D_1D_2)^{-1}.$
\end{Remark}
In what follows, $C$ will be a constant, and $C$'s on different lines should not be assumed to be related. 
\begin{Lemma}\label{technicallemma} Set $f_n=\ln\left(\dfrac{|q_n|^2}{g\dwn{n}^ne^{2u\up{1}}}\right)$ and $f_{n-1}=\ln\left(\dfrac{|q_{n-1}|^2}{g\dwn{n-1}^{n-1}e^{2v\up{1}}}\right)$ then
\[\text{max}(\ t^2e^{f_n})\leq B_1+2M_n\ \ \ \ \ \text{and} \ \ \ \ \ \ e^{f_{n-1}}\leq1\]
\end{Lemma}
\begin{proof}
For the first inequality, we compute
\[\frac{1}{4}\Delta_{g\dwn{n}}(f_n)=\frac{1}{4}\Delta_{g\dwn{n}}\ln\left(\dfrac{|q_n|^2}{g\dwn{n}^ne^{2u\up{1}}}\right)=\frac{1}{g\dwn{n}}\p_z\p_{\bz} \ln|q_n|^2-\frac{1}{g\dwn{n}}\p_z\p_{\bz} \ln\left(g\dwn{n}^n\right)-\frac{1}{g\dwn{n}}\p_z\p_{\bz} \ln\left(e^{2u\up{1}}\right).\]
The first term vanishes since $q_n$ is holomorphic, the second term involves the curvature of $g\dwn{n}$ and the third term involves the Laplacian of $u\up{1}.$ We have
\[\frac{1}{4}\Delta_{g\dwn{n}}(f_n)=\frac{1}{4}\Delta_{g\dwn{n}}\left(\dfrac{|q_n|^2}{g^ne^{2u\up{1}}}\right)=\frac{n}{2}K_{g\dwn{n}}-\frac{1}{4}\Delta_{g\dwn{n}}(2u\up{1}).\] 
The equation for $\Delta_{g\dwn{n}}(u\up{1})$ yields
\[\frac{1}{4}\Delta_{g\dwn{n}}(f_n)=2t^2\dfrac{|q_n|^2}{g\dwn{n}^ne^{2u\up{1}}}-2e^{u\up{1}-u\up{2}}+\haf K_{g\dwn{n}}.\]
By the maximum principle, at the maximum of $f_n$ we have
\[0\geq\frac{1}{4}\Delta_{g\dwn{n}}(f_n)=2t^2e^{f_n}-2e^{u\up{1}-u\up{2}}\]
which gives
\[t^2e^{f_n}\leq e^{u\up{1}-u\up{2}}\leq B_1+2M_{n}.\]
Similarly, for the second inequality we compute
\[\Delta_{g\dwn{n-1}}f_{n-1}=\Delta_{g\dwn{n-1}}\ln\left(t^2|q_{n-1}|\right)-(n-1)\Delta_{g\dwn{n-1}}\ln\left(g\dwn{n-1}\right)-2\Delta_{g\dwn{n-1}}v\up{1}\]
\[=2(n-1)K_{g\dwn{n-1}}-2\Delta_{g\dwn{n-1}}v\up{1}\]
since $q_{n-1}$ is holomorphic. Now using the first equation
\[=2(n-1)K_{g\dwn{n-1}}+(e^{-v\up{1}-v\up{2}}\dfrac{|q_{n-1}|t^2}{g\dwn{n-1}^{n-1}}-e^{v\up{1}-v\up{2}}-\frac{n-1}{4}K_{g\dwn{n-1}})\]
\[=8e^{v\up{1}-v\up{2}}(e^{f_{n-1}}-1).\]
At the maximal point of $f_{n-1}$ we get
\[0\geq\Delta_{g\dwn{n-1}}f_{n-1}=8e^{v\up{1}-v\up{2}}(e^{f_{n-1}}-1),\]
thus $e^{f_{n-1}}\leq1$ on $\Sigma.$
\end{proof}

\begin{Lemma} \label{upperboundlemma} We have the following upper bound on $A^2_{\frac{n}{2}}$ and $D^2_\frac{n}{2},$
\eqtns{A^2_{\frac{n}{2}}\leq C+t^\frac{2}{n}\\
D_\frac{n}{2}^2\leq \dfrac{D_j}{D_{j+1}}+C\leq2\dfrac{D_1}{D_2} +C & 1<j<\frac{n}{2}.}{}
\end{Lemma}
\begin{proof}
For the first inequality we use two telescoping sums, the first is
\[\dfrac{A_j}{A_{j+1}}=(\dfrac{A_j}{A_{j+1}}-\dfrac{A_{j+1}}{A_{j+2}})+\dots+(\dfrac{A_{\frac{n}{2}-1}}{A_{\frac{n}{2}}}-A_{\frac{n}{2}}^2)+A_{\frac{n}{2}}^2\geq-C(M_n)+A_{\frac{n}{2}}^2\]
by equations (\ref{maxuj}). Thus
\[A_1=\frac{A_1}{A_2}\frac{A_2}{A_3}\cdots\frac{A_{\frac{n}{2}-1}}{A_{\frac{n}{2}}}A_{\frac{n}{2}}\geq(-C+A^2_\frac{n}{2})^{\frac{n}{2}-1}(A_\frac{n}{2}^2)^{\haf}\geq (-C+A_\frac{n}{2}^2)^\frac{n-1}{2}.\]
The second telescoping sum is
\[\dfrac{t^2}{A_1^2}=(\dfrac{t^2}{A_1^2}-\dfrac{A_1}{A_2})+(\dfrac{A_1}{A_2}-\dfrac{A_2}{A_3})+\dots+(\dfrac{A_{\frac{n}{2}-1}}{A_{\frac{n}{2}}}-A_{\frac{n}{2}}^2)+A_{\frac{n}{2}}^2\geq-C(M_n)+A_{\frac{n}{2}}^2,\]
putting them together we obtain $t^2\geq(-C+A_\frac{n}{2}^2)^n,$ giving the desired upper bound.

Now for the second inequality, at the maximum of $e^{v\up{2}}$ 
\[0\leq-\frac{1}{4}\Delta_{g\dwn{n-1}}v\up{2}=\frac{1}{D_2e^{v\up{1}}}\frac{|q_{n-1}|^2t^2}{g^{n-1}\dwn{n-1}}+\frac{e^{v\up{1}}}{D_2}-\frac{D_2}{e^{v\up{3}}}-\frac{n-3}{4}K_{g\dwn{n-1}}\]
\[\leq e^{v\up{1}-v\up{2}}\left(\frac{|q_{n-1}|^2t^2}{g^{n-1}\dwn{n-1}e^{2v\up{1}}}+1\right)-e^{v\up{2}-v\up{3}}-(n-3)M_{n-1}.\]
By Lemma \ref{technicallemma}
\[0\leq2\dfrac{D_1}{D_2}-\frac{D_2}{D_3}+(n-3)M_{n-1},\]
thus $\dfrac{D_2}{D_3}\leq2\dfrac{D_1}{D_2}+(n-3)M_{n-1}.$
Using the inequalities (\ref{maxvj}) we have 
\eqtns{\frac{D_j}{D_{j+1}}\leq\frac{D_{j-1}}{D_j}+C & 2<j<\frac{n}{2}\\
 D_\frac{n}{2}^2\leq \dfrac{D_{\frac{n}{2}-1}}{D_\frac{n}{2}}+C.}{}
Putting the inequalities together gives
\[D_{\frac{n}{2}}^2\leq\frac{D_j}{D_{j+1}}+C\leq 2\frac{D_1}{D_2}+C\ \ \ \ \ \ \ for\ 1<j<\frac{n}{2}\]
as desired.
\end{proof}
\begin{Lemma}\label{Ajbounds}
The following upper bound on $A_j$ and $D_j$ hold,
\eqtns{A_j\leq (A_{\frac{n}{2}}^2+C)^{\frac{n+1-2j}{2}} & \text{for all } j\\
D_1\leq \haf (D^2_{\frac{n}{2}}+C)^\frac{n-1}{2} \\
D_j\leq (D_\frac{n}{2}^2+C)^\frac{n+1-2j}{2} & j>1}{}
\end{Lemma}
\begin{proof}
In both cases, we first prove an inequality on the $B_j$'s and the $E_j$'s, and then use Lemma \ref{upperboundlemma}. To do this we subtract the equation $j+1$ from equations $j.$ 

For the first inequality, when $j=1$ at the maximum of $u\up{1}-u\up{2}$ we have 
\[0\leq-\frac{1}{4}\Delta_{g\dwn{n}}(u\up{1}-u\up{2})=-2e^{u\up{1}-u\up{2}}+e^{u\up{2}-u\up{3}}+e^{-2u\up{1}}\frac{t^2|q_n|^2}{g\dwn{n}^n}-\haf K_{g\dwn{n}}\]
\[\leq -2B_1+B_2+e^{-2u\up{1}}\frac{t^2|q_n|^2}{g\dwn{n}^n}+2M_n\leq -B_1+B_2+4M_n\]
by Lemma \ref{technicallemma}. 
For $1<j<\frac{n}{2}-1$ at the maximum of $u\up{j}-u\up{j+1},$ 
\[0\leq-\frac{1}{4}\Delta_{g\dwn{n}}(u\up{j}-u\up{j+1})=-\haf K_{g\dwn{n}}+e^{u\up{j-1}-u\up{j}}-2e^{u\up{j}-u\up{j+1}}+e^{u\up{j+1}-u\up{j+2}}\]
\[\leq 2M_n+B_{j-1}-2B_j+B_{j+1}.\]
At the maximum of $u\up{\frac{n}{2}-1}-u\up{\frac{n}{2}},$ 
\[0\leq-\frac{1}{4}\Delta_{g\dwn{n}}(u\up{\frac{n}{2}-1}-u\up{\frac{n}{2}})=-\haf K_{g\dwn{g}}+e^{u\up{\frac{n}{2}-2}-u\up{\frac{n}{2}-1}}-2e^{u\up{\frac{n}{2}-1}-u\up{\frac{n}{2}}}+e^{2u\up{\frac{n}{2}}}\]
\[\leq 2M_n+B_{\frac{n}{2}-2}-2B_{\frac{n}{2}-1}+A_{\frac{n}{2}}^2.\]
Rewriting the $B_j$ inequalities slightly gives:
\eqtns{B_1-B_2\leq 4M_n\\
B_j-B_{j+1}\leq 2M_n+B_{j-1}-B_j & 1<j<\frac{n}{2}-1\\
B_{\frac{n}{2}-1}-A_{\frac{n}{2}}^2\leq2M_n+B_{\frac{n}{2}-2}-B_{\frac{n}{2}-1}.}{}
The first two inequalities give $B_j-B_{j+1}\leq C$ for $1\leq j<\frac{n}{2}-1,$ and the third gives $B_{\frac{n}{2}-1}-A_\frac{n}{2}^2\leq C.$ 
Putting these together we have
\begin{equation}\label{Bjbound}B_j\leq A_\frac{n}{2}^2+C\ \ \ \ \ \ \ \text{for all}\ j.\end{equation}
Since $A_j=\dfrac{A_j}{A_{j+1}}\cdots\dfrac{A_{\frac{n}{2}-1}}{A_\frac{n}{2}}\cdot A_\frac{n}{2}\leq B_jB_{j+1}\cdots B_{\frac{n}{2}-1}\cdot A_\frac{n}{2}$ we have the desired inequality
\[A_j\leq (C+A_\frac{n}{2}^2)^{\frac{n}{2}-j}A_\frac{n}{2}\leq (C+A_\frac{n}{2}^2)^{\frac{n+1-2j}{2}}.\]

The second inequality involves the $v\up{j}$'s; at the maximum of $v\up{1}-v\up{2},$
\[0\leq-\frac{1}{4}\Delta_{g\dwn{n-1}}(v\up{1}-v\up{2})=-2e^{v\up{1}-v\up{2}}+e^{v\up{2}-v\up{3}}-\haf K_{g\dwn{n-1}}\leq-2E_1+E_2+2M_{n-1}\]
thus $2E_1-E_2\leq C.$ At the maximum of $v\up{2}-v\up{3}$  
\[0\leq-\frac{1}{4}\Delta_{g\dwn{n-1}}(v\up{2}-v\up{3})=e^{-v\up{1}-v\up{2}}\frac{|q_{n-1}|^2t^2}{g\dwn{n-1}^{n-1}}+e^{v\up{1}-v\up{2}}-2e^{v\up{2}-v\up{3}}+e^{v\up{3}-v\up{4}}-\haf K_{g\dwn{n-1}}.\]
By Lemma \ref{technicallemma} and the definition of $E_j,$ 
\[0\leq e^{v\up{1}-v\up{2}}+E_1-2E_2+E_3+M_{n-1}\leq 2E_1-2E_2+E_3+2M_{n-1}\]
thus $E_2-E_3\leq C+2E_1-E_2.$

For $2<j<\frac{n}{2}-1,$ at the maximum of $v\up{j}-v\up{j+1},$ we have
\[0\leq-\frac{1}{4}\Delta_{g\dwn{n-1}}(v\up{j}-v\dwn{j+1})=e^{v\up{j-1}-v\up{j}}-2e^{v\up{j}-v\up{j+1}}+e^{v\up{j+1}-v\up{j+2}}-\haf K_{g\dwn{n-1}}\]
thus $E_j-E_{j+1}\leq E_{j-1}-E_j+C.$ 

Finally, at the maximum of $v\up{\frac{n}{2}-1}-v\up{\frac{n}{2}},$ we obtain
\[0\leq-\frac{1}{4}\Delta_{g\dwn{n-1}}(v\up{\frac{n}{2}-1}-v\up{\frac{n}{2}})=e^{v\up{\frac{n}{2}-2}-v\up{\frac{n}{2}-1}}-2e^{v\up{\frac{n}{2}-1} -v\up{\frac{n}{2}}}+e^{2v\up{\frac{n}{2}}}-\haf K_{g\dwn{n-1}}\]
thus $E_{\frac{n}{2}-1}-D^2_\frac{n}{2}\leq E_{\frac{n}{2}-2}-E_{\frac{n}{2}-1}+C.$

Combining the inequalities for the $E_j$'s yields $E_{\frac{n}{2}-1}-D^2_\frac{n}{2}\leq C+E_{j-1}-E_j\leq 2E_1-E_2+C,$ and hence
\eqtns{2E_1\leq D_\frac{n}{2}^2+C\\
E_j\leq D_\frac{n}{2}^2+C.}{Ejs}

Since $D_j=\dfrac{D_j}{D_{j+1}}\cdots\dfrac{D_{\frac{n}{2}-1}}{D_\frac{n}{2}}\cdot D_\frac{n}{2}\leq E_jE_{j+1}\cdots E_{\frac{n}{2}-1}\cdot D_\frac{n}{2},$ the desired inequalities are proved
\eqtns{D_1\leq\haf(D_\frac{n}{2}^2+C)^\frac{n-1}{2}\\
D_j\leq (D_\frac{n}{2}^2+C)^\frac{n+1-2j}{2}&j>1.}{}
\end{proof}
\noindent We are now ready to prove the theorem. 
\begin{proof}(\emph{Of  Theorem \ref{metrictheorem}})
Recall that we are proving that on $\Omega_n$ and $\Omega_{n-1}$
\eqtns{e^{u\up{j}}=t^{\frac{n+1-2j}{n}}\left(1+O\left(t^{-\frac{2}{n}}\right)\right)&\text{for all}\ j\\
e^{v\up{j}}=(2t)^{\frac{n+1-2j}{n-1}}\left(1+O\left(t^{-\frac{2}{n-1}}\right)\right) & 1<j\leq\frac{n}{2}\\
e^{v\up{1}}=t\left(1+O\left(t^{-\frac{2}{n-1}}\right)\right)}{}
First the $u\up{j}$'s, by Lemma \ref{technicallemma} we have $\dfrac{t^2|q_n|^2}{g\dwn{n}^ne^{2u\up{1
}}}\leq B_1.$ Thus on $\Omega_{n},$ the choice of metric $g\dwn{n}$ yields $t^2e^{-2u\up{1}}\leq B_1.$ Using (\ref{Bjbound}), on $\Omega_n$ we have
\[t^2e^{-2u\up{1}}\leq B_1\leq C+A_{\frac{n}{2}}^2\] which, together with Lemma \ref{Ajbounds}, gives the inequality
\[\left (\frac{t^2}{A_\frac{n}{2}^2+C}\right )^{\haf}\leq e^{u\up{1}}\leq A_1\leq (C+A_\frac{n}{2}^2)^\frac{n-1}{2}.\]
We may rewrite this inequality as $t^\frac{2}{n}\leq A_\frac{n}{2}^2+C.$

Since $e^{u\up{1}}=e^{u\up{1}-u\up{2}}e^{u\up{2}-u\up{3}}\cdots e^{u\up{j-1}-u\up{j}}e^{u\up{j}}$, by (\ref{Bjbound}) it follows that
\[\left (\frac{t^2}{A_\frac{n}{2}^2+C}\right)^{\haf}\leq e^{u\up{1}}\leq B_1\cdots B_{j-1}e^{u\up{j}}.\]
Lemma \ref{upperboundlemma} gives $B_j\leq A_\frac{n}{2}^2+C\leq t^{\frac{2}{n}}+C,$ and by Lemma \ref{Ajbounds} 
\[e^{u\up{j}}\leq A_j\leq(t^\frac{2}{n}+C)^\frac{n+1-2j}{2}.\]
Hence on $\Omega_n,$
\[\frac{t}{(t^{\frac{2}{n}}+C)^\frac{2j-1}{2}}\leq e^{u\up{j}}\leq (t^\frac{2}{n}+C)^\frac{n+1-2j}{2}.\]
Thus, on $\Omega_n,$ we have the desired
\[e^{u\up{j}}=t^\frac{n+1-2j}{n}\left(1+O\left(t^{-\frac{2}{n}}\right)\right).\]

Now for the second two equations, recall that Lemma \ref{technicallemma} says $\dfrac{t^2|q_{n-1}|^2}{g\dwn{n-1}^{n-1}e^{2v\up{1}}}\leq 1.$ By definition of $g\dwn{n-1},$ on $\Omega_{n-1},$ 
\[\frac{t^2}{e^{2v\up{1}}}\leq 1.\]
So on $\Omega_{n-1},$ $t\leq e^{v\up{1}}\leq D_1.$ By Lemma \ref{Ajbounds} 
\[(2t)^{\frac{2}{n-1}}-C\leq D_\frac{n}{2}^2\] 
on $\Omega_{n-1},$ thus by Lemma \ref{upperboundlemma}  
\[2^\frac{3-n}{n-1}t^\frac{2}{n-1}-C\leq\haf D_\frac{n}{2}^2-C\leq \frac{D_1}{D_2}.\] 

Rearranging yields
\[\frac{D_2}{D_1}\leq\frac{1}{2^\frac{3-n}{n-1}t^\frac{2}{n-1}-C}\] on $\Omega_{n-1}.$ Using the first equation and the definition of $g\dwn{n-1},$ at the maximum of $e^{v\up{1}}$ we have 
\[0\leq -\frac{1}{4}\Delta_{g\dwn{n-1}}v\up{1}=e^{-v\up{1}-v\up{2}}\frac{|q_{n-1}|^2t^2}{g^{n-1}\dwn{n-1}}-e^{v\up{1}-v\up{2}}-C\]
\[\leq \frac{t^2}{D_1e^{v\up{2}}}-\frac{D_1}{e^{v\up{2}}}+C\leq \frac{t^2-D_1^2+CD_2D_1}{D_1e^{v\up{2}}}.\] 
 After rearranging the inequality it becomes
\[1\leq \frac{t^2}{D_1^2}+C\frac{D_2}{D_1}.\]
Thus on $\Omega_{n-1},$ we have $1\leq \dfrac{t^2}{D_1^2}+Ct^{-\frac{2}{n-1}},$ simplifying yields
\[D_1\leq \frac{t}{1-Ct^{-\frac{2}{n-1}}}\leq t(1+Ct^{-\frac{2}{n-1}}).\] 
By Lemma \ref{Ajbounds}, $D_\frac{n}{2}^2\leq(2D_1)^\frac{2}{n-1}+C,$ thus on $\Omega_{n-1}$
\[D_\frac{n}{2}^2\leq 2^\frac{2}{n-1}(t(1+Ct^{-\frac{2}{n-1}}))^\frac{2}{n-1}\]
Using Lemma \ref{Ajbounds} once more, for $1<j<\frac{n}{2}$ 
\[D_j\leq 2^\frac{2}{n+1-2j}(t(1+Ct^{-\frac{2}{n-1}}))^\frac{2}{n+1-2j}.\]
Using (\ref{Ejs}) we have
\eqtns{E_j\leq 2^{\frac{2}{n-1}}(t(1+Ct^{-\frac{2}{n-1}}))^\frac{2}{n-1} & 1<j\\
2E_1\leq 2^\frac{2}{n-1}(t(1+Ct^{-\frac{2}{n-1}}))^\frac{2}{n-1}
.}{Ejbound}
As in the proof of the first equation, 
\[e^{v\up{1}}=e^{v\up{1}-v\up{2}}e^{v\up{2}-v\up{3}}\cdots e^{v\up{j-1}-v\up{j}}\cdot e^{v\up{j}}\leq E_1E_2\cdots E_{j-1}\cdot e^{v\up{j}},\] thus restricting to $\Omega_{n-1}$
\[\frac{t}{E_1\cdots E_{j-1}}\leq e^{v\up{j}}.\]
Applying (\ref{Ejbound}) yields
\[\frac{2t}{(t(1+Ct^{-\frac{2}{n-1}}))^\frac{2(j-1)}{n-1}}\leq e^{v\up{j}}.\]
Finally, since $e^{v\up{j}}\leq D_j$ we have 
\eqtns{t\leq e^{v\up{1}}\leq t(1+Ct^{-\frac{2}{n-1}})\\
\frac{2t}{(t(1+Ct^{-\frac{2}{n-1}}))^\frac{2(j-1)}{n-1}}\leq e^{v\up{j}}\leq 2^\frac{2}{n+1-2j}(t(1+Ct^{-\frac{2}{n-1}}))^\frac{2}{n+1-2j}.}{}
Hence we have shown that on $\Omega_{n-1},$ 
\[e^{v\up{1}}=t\left(1+O\left(t^{-\frac{2}{n-1}}\right)\right)\] 
\[e^{v\up{j}}=(2t)^\frac{n+1-2j}{n-1}\left(1+O\left(t^{-\frac{2}{n-1}}\right)\right)\ \ \ \ \ \ for\  1<j\leq\frac{n}{2}.\]
\end{proof} 
\end{proof}

In terms of the $u\up{j}$'s and $v\up{j}$'s, Theorem \ref{metrictheorem} says the asymptotics of the metric solving the Hitchin equations on $\Omega_{n}$ are 
\[e^{u\up{j}}=t^{\frac{n+1-2j}{n}}\left(1+O\left(t^{-\frac{2}{n}}\right)\right)\ \ \ \ \  1\leq j\leq\frac{n}{2},\]
and for $\phi=\tilde e_1+tq_{n-1}e_{n-2},$ the asymptotics of the metric solving the Hitchin equations on $\Omega_{n-1}$ are 
\[e^{v\up{j}}=(2t)^{\frac{n+1-2j}{n-1}}\left(1+O\left(t^{-\frac{2}{n-1}}\right)\right)\ \ \ \ \ 1<j\leq\frac{n}{2}\]
\[e^{v\up{1}}=t\left(1+O\left(t^{-\frac{2}{n-1}}\right)\right).\]
Using our understanding of the $u\up{j}$'s, the $v\up{j}$'s, and their Laplacians, we gain control of their first derivatives.

\begin{Proposition}\label{derivprop} Let $z$ be a local coordinate so that $q_n=dz^n,$ then there is a constant $C_n=C_n(\Sigma,q_n,\Omega_{n})$ so that 
\[|u_z\up{j}|\leq C_nt^{-\frac{1}{n}}\ .\]
Similarly, let $z$ be a local coordinate so that $q_{n-1}=dz^{n-1},$ then there is a constant $C_{n-1}=C_{n-1}(\Sigma,q_{n-1},\Omega_{n-1})$ so that
\[|v_z\up{j}|\leq C_{n-1}t^{-\frac{1}{n-1}}\ .\] 
\end{Proposition}
\begin{proof}
Recall that the $u\up{j}$'s are functions of $t$, we continue to suppress this from the notation. Theorem \ref{metrictheorem} implies 
\eqtns{
e^{u\up{j}-u\up{j+1}}=t^{\frac{2}{n}}\left(1+O\left(t^{-\frac{2}{n}}\right)\right)
\\\dfrac{t^2}{e^{2u\up{1}}}=t^{\frac{2}{n}}\left(1+O\left(t^{-\frac{2}{n}}\right)\right)}{}
For $p\in\Omega_{n}$, choose a local coordinate $z$ centered at $p$ with $q_n=dz^n.$ Consider the functions $\alpha\up{j}$ defined by 
\[\alpha\up{j}(z)=u\up{j}(t^{-\frac{1}{n}}z)-\frac{n+1-2j}{n}\ln(t)\ \ \ for\ 1\leq j\leq \frac{n}{2},\]
where $u\up{j}(t^{-\frac{1}{n}}z)$ is just a rescaling of $u\up{j}.$
The $\alpha\up{j}$'s then satisfy the following two properties
\eqtns{|\alpha\up{j}|=|u\up{j}-\frac{n+1-2j}{n}\ln(t)|\leq Ct^{-\frac{2}{n}}\\
\alpha\up{j}_{z\bz}=t^{-\frac{2}{n}}u\up{j}_{z\bz}.}{alphaprop}
By choice of the background metric $g\dwn{n}$ and our coordinate system, on $\Omega_{n}$ we have $q_n=dz^n$, $K_{g\dwn{n}}=0$, and $\frac{|q_n|^2}{(g\dwn{n})^n}=1.$ Hence for $j=1,$ using the equations we have
\[|\alpha\up{1}_{z\bz}|=|t^{-\frac{2}{n}}u\up{1}_{z\bz}|=|t^{-\frac{2}{n}}(\dfrac{t^2}{e^{2u\up{1}}}-e^{u\up{1}-u\up{2}})|\]
\[=|t^{-\frac{2}{n}}\left(t^{\frac{2}{n}}O\left(t^{-\frac{2}{n}}\right)\right)|\leq Ct^{-\frac{2}{n}}.\]
For $1<j<\frac{n}{2},$ using the equations we have
\[|\alpha\up{j}_{z\bz}|=|t^{-\frac{2}{n}}u\up{j}_{z\bz}|=|t^{-\frac{2}{n}}(e^{u\up{j-1}-u\up{j}}-e^{u\up{j}-u\up{j+1}})|\]
\[=|t^{-\frac{2}{n}}\left(t^{\frac{2}{n}}O\left(t^{-\frac{2}{n}}\right)\right)|\leq Ct^{-\frac{2}{n}}.\]
For $j=\frac{n}{2},$
\[|\alpha\up{\frac{n}{2}}_{z\bz}|=|t^{-\frac{2}{n}}u\up{\frac{n}{2}}_{z\bz}|=|t^{-\frac{2}{n}}(e^{u\up{\frac{n}{2}-1}-u\up{\frac{n}{2}}}-e^{2u\up{\frac{n}{2}}})|\]
\[=|t^{-\frac{2}{n}}\left(t^{\frac{2}{n}}O\left(t^{-\frac{2}{n}}\right)\right)|\leq Ct^{-\frac{2}{n}}.\]
Hence for all $j,$ we have both $|\alpha\up{j}|\leq Ct^{-\frac{2}{n}}$ and $|\alpha\up{j}_{z\bz}|\leq Ct^{-\frac{2}{n}}.$

Applying Schauder's estimate gives
\[|\alpha\up{j}|_{C^1}\leq C|\alpha\up{j}|_{W^{2,p}}\leq Ct^{-\frac{2}{n}}.\]
Thus $|\alpha\up{j}_z|\leq Ct^{-\frac{2}{n}},$ proving \[|u\up{j}_z|=t^{\frac{1}{n}}|\alpha\up{j}_z|\leq Ct^{-\frac{1}{n}}.\]

The $\phi=\tilde e_1+q_{n-1}e_{n-2}$ case works similarly.
For $p\in \Omega_{n-1}$, choose the local coordinate $z$ centered at $p,$ and consider the functions $\beta\up{j}$ defined by 
\[\beta\up{j}(z)=v\up{j}(t^{-\frac{1}{n-1}}z)-\frac{n+1-2j}{n}\ln(t)\ \ \ \text{for}\ \ \ 1\leq j\leq \frac{n}{2}.\]
The $\beta\up{j}$'s also satisfy (\ref{alphaprop}), thus as above, our equations simplify to \[|\beta_{z\bz}\up{j}|=|t^{-\frac{2}{n-1}}v_{z\bz}\up{j}|=\begin{dcases}
|t^{-\frac{2}{n-1}}(t^2e^{-v\up{1}-v\up{2}}-e^{v\up{1}-v\up{2}})|\leq Ct^{-\frac{2}{n-1}} & j=1 \\
|t^{-\frac{2}{n-1}}(t^2e^{-v\up{1}-v\up{2}}+e^{v\up{1}-v\up{2}}-e^{v\up{2}-v\up{3}})|\leq Ct^{-\frac{2}{n-1}} & j=2 \\ 
|t^{-\frac{2}{n-1}}(e^{v\up{j-1}-v\up{j}}-e^{v\up{j}-v\up{j+1}})|\leq Ct^{-\frac{2}{n-1}} & 2<j<\frac{n}{2}\\
|t^{-\frac{2}{n-1}}(e^{v\up{\frac{n}{2}-1}-v\up{\frac{n}{2}}}-e^{2v\up{\frac{n}{2}}})|\leq Ct^{-\frac{2}{n-1}} & j=\frac{n}{2}
\end{dcases}
\]
Thus, with the same argument as the $\phi=\tilde e_1+q_ne_{n-1}$ case, we get 
\[|v_z\up{j}|=t^{\frac{1}{n-1}}|\beta_z\up{j}|\leq Ct^{-\frac{1}{n-1}}\]
%
%
%
\end{proof}

\section{Parallel transport asymptotics}\label{ODESetup}
In this section, the parallel transport ODE we wish to integrate is setup. To avoid some redundancy, we will sometimes use a subscript or superscript $b$ will be used to denote objects corresponding to the $b$-cyclic Higgs field $\phi_b=\tilde e_1 +tq_{b}e_{b-1}$ for $b=n,n-1.$ We will also work in the universal cover $\widetilde\Sigma$ of $\Sigma,$ all objects should be pulled back to the universal cover.

Let $P\in\widetilde{\Sigma}$ be away from the zeros of the differential $q_b,$ and choose a neighborhood $\Uu\dwn{P}$ centered at $P,$ with coordinate $z,$ so that $q_b=dz^b.$ Note that for this to make sense, $\Uu_P$ must be disjoint from the zero set of $q_b.$ In this neighborhood, $u\up{j}=\lambda\up{j}$ for $b=n$ and $v\up{j}=\lambda\up{j}$ for $b=n-1.$

As before, the choice of local coordinate $z$ defines a local holomorphic frame $(s_1,\dots,s_{\frac{n}{2}},s_{\frac{n}{2}}^*,\dots,s_1^*)$ for
\[E=K^{\frac{n-1}{2}}\oplus K^\frac{n-3}{2}\oplus\cdots\oplus K^{-\frac{n-1}{2}},\]
where $s_j=dz^{\frac{n+1-2j}{2}}.$
In this frame, the connection $1$-form of the corresponding flat connection is given by (\ref{flat1}) and (\ref{flat2}).
By our choice of coordinates, the $f\dwn{b}$ in (\ref{locqb}) is identically 1.

Using our estimates from Theorem \ref{metrictheorem} and Proposition \ref{derivprop}, we will solve for the transport matrix $T_{P,P'}(t)$ along paths starting at P and ending at a point $P'$ in the neighborhood $\Uu\dwn{P}$. In fact, $T_{P,P'}(t)$ will be calculated along geodesics of the background metric $g_b=|dz|^\frac{2}{b}$ which start at $P$ and end at $P'.$ Since the connection is flat, the value of $T_{P,P'}(t)$ is path independent in $\Uu\dwn{P}$.

We rescale the holomorphic frame $(s_1,\cdots,s_1^*)$ so that it stays bounded away from $0$ and $\infty$ as $t\ra\infty.$
For $\phi=\tilde e_1+tq_ne_{n-1}$, the rescaled frame is given by $F_n=(\sigma_1,\dots,\sigma_1^*)$ where
\begin{equation}\label{RescaledFrame}\sigma_j=t^{\frac{n+1-2j}{2n}}s_j\ \ \ \ \ \ \ \ \ \ \sigma^*_j=t^{-\frac{n+1-2j}{2n}}s^*_j.\end{equation}
\begin{Remark}\label{FmetricRemark}
By Theorem \ref{metrictheorem}, in the rescaled frame, the metric $h=Id\left(1+O\left(t^{-\frac{2}{n}}\right)\right).$ To see this, consider
\[h(s_i,s_j)=\delta_{ij}t^\frac{i+j-n-1}{n}\left(1+O\left(t^{-\frac{2}{n}}\right)\right)\]
thus
\[h(\sigma_i,\sigma_j)=h(t^{\frac{n+1-2i}{2n}}s_i,t^{\frac{n+1-2j}{2n}}s_j)\]
\[=t^{\frac{n+1-(i+j)}{n}}h(s_i,s_j)=\delta_{ij}\left(1+O\left(t^{-\frac{2}{n}}\right)\right).\]
\end{Remark}
For $\phi=\tilde e_1+tq_{n-1}e_{n-2},$ the rescaled frame is denote by $F_{n-1}=(\sigma_1,\dots,\sigma_1^*),$ it is given by
\[\sigma_1=t^{\haf}s_1\ \ \ \ \ \ \sigma^*_1=t^{-\frac{1}{2}}s^*_1\]\[\sigma_j=(2t)^{\frac{n+1-2j}{2(n-1)}}s_j\ \ \ \ \ \ \sigma^*_j=(2t)^{-\frac{n+1-2j}{2(n-1)}}s_j^*\ \ \ \ \ \  j=2,\dots,\frac{n}{2}.\]
As in the previous case, the harmonic metric in this frame is $h=Id\left(1+O\left(t^{-\frac{2}{n-1}}\right)\right).$

If we denote the flat connection by $D_b=U_bdz+V_b d\bz,$ then, by the estimates from Theorem \ref{metrictheorem} and Proposition \ref{derivprop}, the matrices in the connection $1$-form are given by:
\begin{enumerate}
\item For $\phi=\tilde e_1+q_n e_{n-1}$,
\[U_n=\mtrx{
-u\up{1}_z&&&t^{\frac{1}{n}}\\
t^{\frac{1}{n}}&-u\up{2}_z&&\\
&\ddots&\ddots&\\
&&t^{\frac{1}{n}}&u_z\up{1}}=
t^\frac{1}{n} \mtrx{
&&&1\\
1&&&\\
&\ddots&&\\
&&1&
}+O\left(t^{-\frac{1}{n}}\right)
\]

\[V_n=\mtrx{&t^{-\frac{1}{n}}e^{u\up{1}-u\up{2}}&&\\
&&\ddots&\\
&&&t^{-\frac{1}{n}}e^{u\up{1}-u\up{2}}\\
t^{-\frac{2n-1}{n}}e^{-2u\up{1}}&&&}=
t^{\frac{1}{n}}\mtrx{&1&&\\
&&\ddots&\\
&&&1\\
1&&&}+O\left(t^{-\frac{1}{n}}\right)
\]
\noindent where $O\left(t^{-\frac{1}{n}}\right)$ is uniform as $t\ra\infty$ for all points in $\Omega_n.$
\item For $\phi=\tilde e_1+q_{n-1} e_{n-2}$,
\[U_{n-1}=\mtrx{-v\up{1}_z&&&2^{-\frac{n-3}{2(n-1)}}t^\frac{1}{n-1}&\\
2^{-\frac{n-3}{2(n-1)}}t^\frac{1}{n-1}&-v_z\up{2}&&&2^{-\frac{n-3}{2(n-1)}}t^{\frac{1}{n-1}}\\
&2^{\frac{1}{n-1}}t^{\frac{1}{n-1}}&&&\\
&&\ddots&&\\
&&&2^{-\frac{n-3}{2(n-1)}}t^\frac{1}{n-1}&v\up{1}_z}\]

\[=
(2t)^\frac{1}{n-1}\mtrx{&&&\frac{1}{\sqrt{2}}&\\
\frac{1}{\sqrt{2}}&&&&\frac{1}{\sqrt{2}}\\
&1&&&\\
&&\ddots&&\\
&&&\frac{1}{\sqrt{2}}&}+O\left(t^{-\frac{1}{n-1}}\right)
\]

\smallskip
\[V_{n-1}=
\mtrx{&e^{v\up{1}-v\up{2}}2^{\frac{n-3}{2(n-1)}}t^{-\frac{1}{n-1}}&&\\
&&&\\
&&\ddots&\\
e^{-v\up{1}-v\up{2}}2^{\frac{n-3}{2(n-1)}}t^\frac{2n-3}{n-1}&&&e^{v\up{1}-v\up{2}}2^{\frac{n-3}{2(n-1)}}t^{-\frac{1}{n-1}}\\
&e^{-v\up{1}-v\up{2}}2^{\frac{n-3}{2(n-1)}}t^\frac{2n-3}{n-1}&&}\]
\[=
(2t)^{\frac{1}{n-1}}\mtrx{&\frac{1}{\sqrt{2}}&&&\\
&&1&&\\
&&&\ddots&\\
\frac{1}{\sqrt{2}}&&&&\frac{1}{\sqrt{2}}\\
&\frac{1}{\sqrt{2}}&&&}+O\left(t^{-\frac{1}{n-1}}\right)\]
\noindent
where $O\left(t^{-\frac{1}{n-1}}\right)$ is uniform as $t\ra\infty$ for all points in $\Omega_{n-1}.$\end{enumerate}
\noindent

As noted above, we will integrate the initial value problem along geodesics of the metric $|q_b|^{\frac{2}{b}}$ which avoid the zeros of $q_b.$ Any $P'\in\Uu\dwn{P},$ can be expressed in polar coordinates $P'=Le^{i\theta};$ the geodesic $\gamma$ of the metric $|q_b|^{\frac{2}{b}}$ which starts at $P$ and ends at $P'$ is the straight line
\[\gamma(s)=se^{i\theta}\ \ \ \ \ \ \ \text{for}\  s\in[0,L].\]

To avoid an overload of notation, when there is no confusion, the $b$ will be dropped from the notation. We start at $P$ with the initial rescaled holomorphic frame $F(P).$ For a fixed $t, $ parallel transportation along the geodesic $\gamma(s):[0,L]\rightarrow \widetilde{\Sigma}$ with respect to the flat connection yields a family of frames $G(\gamma(s))(t)$ along $\gamma$ given by
\[G(\gamma(s))(t)=T_{P,\gamma(s)}(t)(F(P))\ \ \ \ \ \ \text{with}\ \ \ \  T_{P,\gamma(0)}(t)=Id.\]

For each $t,$ consider the family of matrices $\Psi_t(s)$ satisfying  \[\Psi_t(0)=Id\ \ \ \ \text{and}\ \ \ \ \  \Psi_t(s)G(\gamma(s))(t)=F(\gamma(s)).\]
Since $G(\gamma(s))(t)$ is parallel along $\gamma,$ rewriting $\nabla_{\frac{\p}{\p s}} F(\gamma(s))$ in terms of $G(\gamma(s))(t)$ yields
\[\nabla_{\frac{\p}{\p s}} F(\gamma(s))=\frac{d\Psi_t}{ds}G(\gamma(s))(t).\]
Also,
\begin{equation*} \nabla_{\frac{\p}{\p s}} F(\gamma(s))=(e^{i\theta}U+e^{-i\theta}V)F(\gamma(s))= (e^{i\theta}U+e^{-i\theta}V)\Psi_t G(\gamma(s))(t),\end{equation*}
hence,
\[\frac{d\Psi_t}{ds}=\left(e^{i\theta}U+e^{-i\theta}V\right)\Psi_t.\]
Rewriting $T_{P,\gamma(s)}(t)$ in terms of $\Psi_t$ gives
\begin{equation}\label{TvPsi} T_{P,\gamma(s)}(t)(F(P))=G(\gamma(s))(t)=\Psi_t(\gamma(s))^{-1}F(\gamma(s)).
\end{equation}
Thus $T_{P,\gamma(s)}(t)=\Psi_t(\gamma(s))^{-1},$ and we obtain the following proposition.
%
%
%
%
\begin{Proposition}\label{FmetricProp} With respect to the frame $(\sigma_1,\dots,\sigma_{\frac{n}{2}},\sigma_{\frac{n}{2}}^*,\dots,\sigma_1^*),$ parallel transport along the geodesic from $P$ to $P'$ for the flat connection is given by $\Psi_t(L)^{-1}\left(1+O\left(t^{-\frac{2}{b}}\right)\right),$ where $\Psi_t$ solves the initial value problem
\[\Psi_t(0)=I\ \ \ \ \ \ \ \ \frac{d\Psi_t}{ds}=\left(e^{i\theta}U+e^{-i\theta}V\right)\Psi_t\]
\end{Proposition}
\smallskip

\noindent Explicitly, we have \begin{enumerate}[1.]
	\item  For $\phi=\tilde e_1+tq_n e_{n-1}$,
\[\frac{d\Psi_t}{ds}=\left[ t^{\frac{1}{n}}\mtrx{
0&e^{-i\theta}&&&0&e^{i\theta}\\
e^{i\theta}&0&e^{-i\theta}&&&\\
&\ddots&&\ddots&&\\
&&&&&\\
0&&&e^{i\theta}&0&e^{-i\theta}\\
e^{-i\theta}&&&&e^{i\theta}&0}
+O\left(t^{-\frac{1}{n}}\right)\right]\Psi_t\]
\item For $\phi=\tilde e_1+tq_{n-1} e_{n-2}$,
\[\dfrac{d\Psi_t}{ds}=\left[ (2t)^{\frac{1}{n-1}}
\mtrx{
0&\frac{1}{\sqrt{2}}e^{-i\theta}&&&\frac{1}{\sqrt{2}}e^{i\theta}&0\\
\frac{1}{\sqrt{2}}e^{i\theta}&0&e^{-i\theta}&&&\frac{1}{\sqrt{2}}e^{i\theta}\\
&\ddots&&\ddots&&\\
&&&&&\\
\frac{1}{\sqrt{2}}e^{-i\theta}&&&e^{i\theta}&0&\frac{1}{\sqrt{2}}e^{-i\theta}\\
0&\frac{1}{\sqrt{2}}e^{-i\theta}&&&\frac{1}{\sqrt{2}}e^{i\theta}&0}
+O\left(t^{-\frac{1}{n-1}}\right)\right]\Psi_t\]
\end{enumerate}

In the above expressions, the matrix inside the bracket may be diagonalized by a constant unitary matrix $S$, and thus can be written as
\[S \mtrx{\mu_1&&\\
&\ddots&\\
&&\mu_n}S^{-1}\]
where the set $\{\mu_j\}$ is the set of roots of the characteristic polynomial $det(\mu I-(e^{i\theta}U+e^{-i\theta}V)).$ More precisely,
\begin{enumerate}[1.]
	\item For the case $\phi=\tilde e_1+q_n e_{n-1}$, $\mu_j=2\cos(\theta+\frac{2\pi j}{n})$.
	\item For the case $\phi=\tilde e_1+q_{n-1} e_{n-2}$, $\mu_1=0$, and for $j\geq 2$, $\mu_j=2\cos(\theta+\frac{2\pi j}{n-1})$.
\end{enumerate} 

To compute $\Psi(L)$, we compute $\Phi=S^{-1}\Psi S$
\begin{equation}\label{Phieq}
\Phi(0)=I, \ \ \ \ \ \ \ \ \
\dfrac{d\Phi}{ds}=\left[t^{\frac{1}{b}} M(\theta)
+R\right]\Phi
\end{equation}
where $M(\theta)=\mtrx{\mu_1&&\\&\ddots&\\&&\mu_n},$ and $S^{-1}RS$ is the error term in Proposition \ref{FmetricProp}.

To integrate this initial value problem, we employ the following strategy:

\noindent Consider the solution $\Phi_0$ to the initial value problem 
\[\xymatrix{\Phi_0(0)=I,& \dfrac{d\Phi_0}{ds}=t^{\frac{1}{n}} M(\theta)\Phi_0.}\]
Hence $\Phi_0(s)=\mtrx{
        e^{st^{\frac{1}{b}}\mu_1}&&&\\
        &e^{st^{\frac{1}{b}}\mu_2}&&\\
        &&\ddots&\\
        &&&e^{st^{\frac{1}{b}}\mu_n}}.$ 
Instead of solving for $\Phi$ asymptotically, we solve for $(\Phi_0)^{-1}\Phi$. Note that $(\Phi_0)^{-1}\Phi$ solves the initial value problem
\begin{equation}
	\label{star}
	(\Phi_0)^{-1}\Phi(0)=I,\ \ \ \ \ \ \ \ \ \dfrac{d((\Phi_0)^{-1}\Phi)}{ds}=(\Phi_0)^{-1}R\Phi_0 \cdot(\Phi_0)^{-1}\Phi.
\end{equation}
This can be seen by using the product rule
\begin{eqnarray*}
\dfrac{d((\Phi_0)^{-1}\Phi)}{ds}&=&
\dfrac{d\Phi_0}{ds}\Phi+(\Phi_0)^{-1}\dfrac{d\Phi}{ds}\\
&=&-(\Phi_0)^{-1}\dfrac{d\Phi_0}{ds}(\Phi_0)^{-1}\Phi+(\Phi_0)^{-1}\dfrac{d\Phi}{ds}\\
&=&-(\Phi_0)^{-1}t^{\frac{1}{b}}M(\theta)\Phi+(\Phi_0)^{-1}(t^{\frac{1}{b}}M(\theta)+R)\Phi\\
&=&(\Phi_0)^{-1}R\Phi\\
&=&(\Phi_0)^{-1}R\Phi_0\cdot(\Phi_0)^{-1}\Phi.
\end{eqnarray*}
For the initial value problem (\ref{star}), we will show $(\Phi_0)^{-1}R\Phi_0$ is $o(1),$ and that $(\Phi_0)^{-1}\Phi$ is $Id+o(1);$ hence
\[\Phi=\Phi_0(Id+o(1)).\]
Before doing this, we need a more in-depth understanding of the error term.  

The estimate of the error term for the ODE relies mainly on the error estimate of the $u\up{j}$'s and $v\up{j}$'s. For the $n$-cyclic case, we introduce the following notation for the error term for $u\up{j}$ coming from Theorem \ref{metrictheorem}
\[\tilde u\up{j}=u\up{j}-ln\left|tq_n\right|^\frac{n+1-2j}{n}.\]
Similarly for the $(n-1)$-cyclic case set 
\[\tilde v\up{j}=\begin{dcases}
v\up{j}-ln\left|tq_{n-1}\right|&j=1\\
v\up{j}-ln\left|2tq_{n-1}\right|^\frac{n+1-2j}{n-1}&otherwise
\end{dcases}\]
For the $n$-cyclic case, writing the error term $R$ for the ODE (\ref{Phieq}) in terms of $\tilde u\up{j}$ gives 
\begin{equation}
S^{-1}\left(e^{i\theta}\mtrx{
\tilde u\up{1}_z&&&\\
&\tilde u\up{2}_z&&\\
&&\ddots&\\
&&&-\tilde u\up{1}_z}+
t^\frac{1}{n}e^{-i\theta}\mtrx{
	0&e^{\tilde u\up{1}-\tilde u\up{2}}-1&&\\
	&\ddots&\ddots&\\
	&&0&e^{ \tilde u\up{1}-\tilde u\up{2}}-1\\
	e^{-2\tilde u\up{1}}-1&&&0
}
\right)S\label{nError}
\end{equation}
which we will write as $R=B^1\dwn{n}+t^\frac{1}{n}B^2\dwn{n}.$
In a similar fashion, the error term for the $(n-1)$-cyclic case is 
\begin{equation*}S^{-1}\left(e^{i\theta}\scalemath{.85}{\mtrx{
\tilde v\up{1}_z&&&&\\
&\tilde v\up{2}_z&&&\\
&&\ddots&&\\
&&&-\tilde v\up{2}_z&\\
&&&&-\tilde v\up{1}_z}}\right.
\end{equation*}
\begin{equation}\left.+(2t)^{\frac{1}{n-1}}e^{-i\theta}\scalemath{.85}{
\mtrx{
	&\frac{1}{\sqrt{2}}(e^{\tilde v\up{1}-\tilde v\up{2}}-1)&&&&\\
	&&e^{\tilde v\up{2}-\tilde v\up{3}}-1&&&\\
	&&&\ddots&&\\
	&&&&e^{\tilde v\up{2}-\tilde v\up{3}}-1&\\
	\frac{1}{\sqrt{2}}(e^{-\tilde v\up{1}-\tilde v\up{2}}-1)&&&&&\frac{1}{\sqrt{2}}(e^{\tilde v\up{1}-\tilde v\up{2}}-1)\\
	&\frac{1}{\sqrt{2}}(e^{-\tilde v\up{1}-\tilde v\up{2}}-1)&&&&
}} \right) S\label{(n-1)Error}
\end{equation}
\subsection{The $n$-cyclic case}
The following theorem concerning estimates of the errors will be crucial.
\begin{Theorem}\label{error} Let $d(p)$ be the minimum distance from a point $p$ to the zeros of $q_n.$ Then for any $d<d(p),$ as $t\rightarrow +\infty$, the $(k,l)$-entry of $R$ satisfies
\[R_{kl}(p)= O\left(t^{-\frac{1}{2n}}e^{-2|1-\zeta\dwn{n}^{k-l}|t^\frac{1}{n}d}\right).\] 
\end{Theorem}
Due to the length and level of technicality of the proof, we will spend the next two sections proving Theorem \ref{error}. In section \ref{ErrorToda} we relate the error term with solutions to the $\fsl(n,\C)$-cyclic Toda lattice, and prove a recursive formula on the Toda lattice. Then in section \ref{ErrorEstimate}, the recursive formula will be applied to prove the desired estimate. 

Assuming Theorem \ref{error}, we can now prove the main theorem concerning the asymptotics of the parallel transport operator with an extra condition on the path.

\begin{Theorem} \label{TransportAsymp} Suppose $P,\ P'$ and the path $\gamma(s)$ are as above. If $P'$ has the property that for every $s$, 
\[ s<d(\gamma(s)):=min\{d(\gamma(s), z_0)| \  \textit{for all zeros $z_0$ of $q_n$}\},\]
then there exists a constant unitary matrix $S$, not depending on the pair $(P, P')$, so that as $t\ra\infty$,
    \[T_{P,P'}(t)=\left( Id+O\left(t^{-\frac{1}{2n}}\right) \right) S \mtrx{
        e^{-Lt^{\frac{1}{n}}\mu_1}&&&\\
        &e^{-Lt^{\frac{1}{n}}\mu_2}&&\\
        &&\ddots&\\
        &&&e^{-Lt^{\frac{1}{n}}\mu_n}}S^{-1}\]
where $\mu_j=2cos\left(\theta+\frac{2\pi {(j-1)}}{n}\right)$.    
\end{Theorem}
 \begin{Remark}\label{smallnbhd}The extra condition on the path is necessary for our method of proof, as the distance from the zeros of the holomorphic differential $q_n$ controls the decay rate of the error terms. However, for sufficiently short paths, the extra condition is automatically satisfied. Thus, for each point $z$ away from the zeros of $q_n$, there is a neighborhood $U$ for which all $|q_n|^\frac{2}{n}$-geodesics in $U$ satisfy the extra condition. Furthermore, if, for all zeros $z_0$ of $q_{n},$ the angle $<_{z_0}(P,P')$ is less than $\pi/3,$ then the $|q_{n}|^{\frac{2}{n}}$-geodesic from $P$ to $P'$ satisfies the condition.
\end{Remark}
When $P$ and $P'$ both project to the same point in $\Sigma,$ the projected path is a loop. In this case, the above asymptotics correspond to the values of the associated family of representations on the homotopy class of the loop.

\begin{proof}
By Theorem \ref{error}, the $(k,l)$-entry of the error term $(\Phi^b_0)^{-1}R\Phi^b_0$ is \[R_{k,l}(\gamma(s)) e^{(\mu_k-\mu_l)st^{\frac{1}{n}}}=O\left(t^{-\frac{1}{2n}}e^{-2|1-\zeta\dwn{n}^{k-l}|t^\frac{1}{n}d(\gamma(s))}e^{(\mu_k-\mu_l)st^{\frac{1}{n}}}\right).\]
Observe that
 \begin{eqnarray*}|\mu_k-\mu_l|&=&\left|2\cos\left(\theta+\frac{2\pi(k-1)}{n}\right)-2\cos\left(\theta+\frac{2\pi(l-1)}{n}\right)\right|\\
&=&\left|4\sin\left(\frac{\pi(k-l)}{n}\right)\sin\left(\theta+\frac{\pi(k+l-2)}{n}\right)\right|\\
&\leq&\left|4\sin\left(\frac{\pi(k-l)}{n}\right)\right|\\&=&2|1-\zeta\dwn{n}^{k-l}|.\end{eqnarray*}
Hence, the $(k,l)$-entry of $(\Phi^b_0)^{-1}R\Phi^b_0$ is $O\left(t^{-\frac{1}{2n}}e^{2|1-\zeta\dwn{n}^{k-l}|t^\frac{1}{n}(s-d(\gamma(s)))}\right)$.
Since $\gamma(s)$ satisfies the condition that for every $s$, $ s<d(\gamma(s)),$ 
we obtain $(\Phi^b_0)^{-1}R\Phi^b_0=O\left(t^{-\frac{1}{2n}}\right)$.

We make use of the following classical theorem in ODE theory, for a nice proof, see appendix B of \cite{DM}. 
\begin{Lemma} \label{ODE}
Let $A:[a,b]\ra\fgl_n(\R)$ be a continuous function. For the equation $F'(s)=F(s)A(s)$ on an interval $[a,b]\subset \R,$ there exists $C,\delta_0>0$ such that if $\|A(t)\|<\delta<\delta_0$ for all $s\in[a,b]$, then the solution $F$ with $F(a)=I$ satisfies $|F(s)-I|<C\delta$ for all $s\in[a,b]$.
\end{Lemma}

Applying Lemma \ref{ODE} and $(\Phi^b_0)^{-1}R\Phi^b_0=O\left(t^{-\frac{1}{2n}}\right)$ to the ODE 
\begin{equation*}
	(\Phi^b_0)^{-1}\Phi^b(0)=I,\ \ \ \ \ \ \ \ \ \dfrac{d((\Phi_0^b)^{-1}\Phi^b)}{ds}=(\Phi_0^b)^{-1}R\Phi_0^b \cdot(\Phi_0^b)^{-1}\Phi^b,
\end{equation*} 
we obtain \[(\Phi^b_0)^{-1}\Phi^b=Id+O\left(t^{-\frac{1}{2n}}\right).\]
Therefore $\Phi^b=\Phi^b_0\left(Id+O\left(t^{-\frac{1}{2n}}\right)\right)$.
\end{proof}

\subsection{The $(n-1)$-cyclic case}
For the $(n-1)$-cyclic case, the crucial error estimate theorem is the following.
\begin{Theorem}\label{suberror} Let $d(p)$ be the minimum distance from a point $p$ to the zeros of $q_{n-1}.$ Then for any $d<d(p),$ as $t\rightarrow +\infty$, the $(k,l)$-entry of $R$ satisfies
\[R_{kl}(p)=\begin{dcases} O\left(t^{-\frac{1}{2(n-1)}}e^{-2|1-\zeta\dwn{n-1}^{k-l}|(2t)^{\frac{1}{n-1}}d}\right)& k,l\geq 2\\
							0&         k=l=1\\
							O\left(t^{-\frac{1}{2(n-1)}}e^{-2(2t)^{\frac{1}{n-1}}d}\right)& otherwise
\end{dcases}\]

\end{Theorem}
As with the $n$-cyclic case, we will assume Theorem \ref{suberror} for now and prove the main theorem concerning the asymptotic of the parallel transport operator with an extra condition on the path.
\begin{Theorem}\label{subTransportAsymp}Suppose $P,\ P'$ and the path $\gamma(s)$ are as above. If $P'$ has the property that for every $s$
 \[ s<d(\gamma(s)):=min\{d(\gamma(s), z_0)| \  \textit{for all zeros $z_0$ of $q_{n-1}$}\},\]
then there exists a constant unitary matrix $S$, not depending on the pair $P$ and $P'$, so that as $t\ra\infty$,
    $$T_{P,P'}(t)=\left(Id+O\left(t^{-\frac{1}{2(n-1)}}\right)\right)S\begin{pmatrix}
        e^{-Lt^{\frac{1}{n-1}}\mu_1}&&&\\
        &e^{-Lt^{\frac{1}{n-1}}\mu_2}&&\\
        &&\ddots&\\
        &&&e^{-Lt^{\frac{1}{n-1}}\mu_n}\end{pmatrix}S^{-1}$$
where $\mu_1=0,$ for $j\geq 2, \mu_j=2cos\left(\theta+\frac{2\pi {(j-2)}}{n-1}\right)$.    
\end{Theorem}
\begin{Remark}The extra condition on the path is necessary for our method of proof, as the distance from the zeros of the holomorphic differential $q_n$ controls the decay rate of error terms. However, for sufficiently short paths, the extra condition is automatically satisfied. Thus, for each point $z$ away from the zeros of $q_n$, there is a neighborhood $U$ for which all $|q_n|^\frac{2}{n}$-geodesics in $U$ satisfy the extra condition. Furthermore, if the angle $<_{z_0}(P,P')$ satisfies $<_{z_0}(P,P')<\pi/3$ for all zeros $z_0$ of $q_{n-1}$, then the $|q_{n-1}|^{\frac{2}{n-1}}$-geodesic from $P$ to $P'$ satisfies the extra condition in Theorem \ref{subTransportAsymp}. \end{Remark}

When $P$ and $P'$ both project to the same point in $\Sigma,$ the projected path is a loop. In this case, the above asymptotics correspond to the values of the associated family of representations on the homotopy class of the loop.

\begin{proof}By Theorem \ref{suberror}, we have the $(k,l)$-entry of the error term $(\Phi^b_0)^{-1}R\Phi^b_0$ is \[R_{k,l}(\gamma(s)) e^{(\mu_k-\mu_l)st^{\frac{1}{n}}}.\]
For $k,l\geq 2$, similar to the proof of Theorem \ref{TransportAsymp}, $|\mu_k-\mu_l|\leq 2|1-\zeta\dwn{n-1}^{k-l}|.$ 
Hence for $k,l\geq 2$, the $(k,l)$-entry of $(\Phi^b_0)^{-1}R\Phi^b_0$ is $O\left(t^{-\frac{1}{2(n-1)}}e^{2|1-\zeta\dwn{n}^{k-l}|(2t)^{\frac{1}{n-1}}(s-d(\gamma(s))}\right)$.

For $k=l=1$, we have $\mu_1=0$, hence the $(1,1)$-entry of $(\Phi^b_0)^{-1}R\Phi^b_0$ is $O\left(t^{-\frac{1}{2(n-1)}}\right)$.
If $k=1$ and $ l\neq 1$, then \[|\mu_k-\mu_l|=|0-2\cos(\theta+\frac{2\pi(l-1)}{n-1})|\leq 2.\]
Also, if $l=1$ and $ k\neq 1$, we have $|\mu_k-\mu_l|=|2\cos(\theta+\frac{2\pi(k-1)}{n-1})-0|\leq 2.$ 
Thus for $k=1, l\neq 1$ or $l=1, k\neq 1$, the $(k,l)$-entry of $(\Phi^b_0)^{-1}R\Phi^b_0$ is 
\[O\left(t^{-\frac{1}{2(n-1)}}e^{2(2t)^{\frac{1}{n-1}}(s-d(\gamma(s)))}\right).\]
Since $\gamma(s)$ satisfies the condition that for every $s$, $ s<d(\gamma(s)),$ 
we obtain that $(\Phi^b_0)^{-1}R\Phi^b_0=O\left(t^{-\frac{1}{2(n-1)}}\right)$.
As in the $n$-cyclic case, we apply Lemma \ref{ODE} and obtain $(\Phi^b_0)^{-1}\Phi^b_0=Id+O\left(t^{-\frac{1}{2(n-1)}}\right),$ and
thus \[\Phi^b=\Phi^b_0\left(Id+O\left(t^{-\frac{1}{2(n-1)}}\right)\right).\]
\end{proof}

\section{The error terms and the Toda lattice}\label{ErrorToda}
To understand the asymptotics of the error terms we reinterpret them in terms of the cyclic Toda Lattice. For a good reference for the Toda lattice, see \cite{GuestIntSystem}.  
\subsection{The Toda lattice}
The cyclic $\fsl(n,\C)$ Toda lattice, or the affine $\fsl(n,\C)$ Toda equations solve the following system,\eqtns{
\Delta d\up{1}=e^{d\up{1}-d\up{2}}-e^{d\up{n}-d\up{1}}\\
\Delta d\up{2}=e^{d\up{2}-d\up{3}}-e^{d\up{1}-d\up{2}}\\
 \ \ \ \ \ \vdots\\
\Delta d\up{n}=e^{d\up{n}-d\up{1}}-e^{d\up{n-1}-d\up{n}}
}{Toda}
where $\Delta={\frac{{\partial}^2}{\partial x^2}}+{\frac{\partial^2}{\partial y^2}}=\frac{1}{4}\frac{\partial^2}{\partial z\partial\overline{z}}$.

Define 
\begin{equation}\label{wkdef}w_k=\frac{1}{\sqrt{n}}\sum\limits_{i\in\Z_n}\zeta\dwn{n}^{ik}(d\up{i}-d\up{i+1})\end{equation}
and note that since the sum is over $\Z_n,$ we have $w_0=0.$ 
\begin{Proposition}\label{w_kProposition}
	The term $w_k$ satisfies the follow properties. 
	\begin{enumerate}[1.]
		\item $\displaystyle w_k=\dfrac{1-\zeta\dwn{n}^{-k}}{\sqrt{n}}\sum\limits_i \zeta\dwn{n}^{ki}d\up{i}$

		\item $\displaystyle\Delta w_k=\dfrac{|1-\zeta\dwn{n}^k|^2}{\sqrt{n}}\sum\limits_i \zeta\dwn{n}^{(i-1)k}e^{d\up{i}-d\up{i+1}}$

		\item $\displaystyle\Delta w_k=|1-\zeta\dwn{n}^k|^2 \sum\limits_{\substack{r\equiv k\\ \text{mod}\ n}}\  \sum\limits_{r_1+\cdots+r_s=r} \frac{1}{s!n^{\frac{s-1}{2}}}\binom{r}{r_1,\cdots,r_s} w_{r_1}w_{r_2}\cdots w_{r_s}$
	\end{enumerate}
\end{Proposition}
\begin{proof}
	Items $1$ and $2$ follow immediately from reordering the summation. Proving $3$ takes a little more work, we start by linearizing the cyclic Toda system (\ref{Toda}) at zero as
	\eqtns{
	\Delta d\up{1}=-d\up{n}+2d\up{1}-d\up{2}\\
	\Delta d\up{2}= -d\up{1}+2d\up{2}-d\up{3}\\
	\ \ \ \ \ \vdots\\
	\Delta d\up{n}=-d\up{n-1}+2d\up{n}-d\up{1}
	}{LinToda}
Subtracting the $(i+1)^{th}$ equations from the $i^{th}$ equations in the linearization gives
\eqtns{
	\Delta(d\up{n}-d\up{1})=-(d\up{n-1}-d\up{n})+2(d\up{n}-d\up{1})-(d\up{1}-d\up{2})\\
	\Delta(d\up{1}-d\up{2})=-(d\up{n}-d\up{1})+2(d\up{1}-d\up{2})-(d\up{2}-d\up{3})\\
	\ \ \ \ \ \vdots\\
	\Delta(d\up{n-1}-d\up{n})=-(d\up{n-2}-d\up{n-1})+2(d\up{n-1}-d\up{n})-(d\up{n}-d\up{1})
}{LinDifToda}
Note that $w_k=\dfrac{1}{\sqrt{n}}\sum\limits_{i\in\Z_n}\zeta\dwn{n}^{ik}(d\up{i}-d\up{i+1})$ is an eigenfunction of equations (\ref{LinDifToda}),
\[\mtrx{w_0\\w_1\\\vdots\\w_{n-1}}=
\frac{1}{\sqrt{n}}\mtrx{1&1&1&\cdots&1\\
	  1&\zeta\dwn{n}&\zeta\dwn{n}^{2}&\cdots&\zeta\dwn{n}^{n-1}\\
	  	  \vdots&&\vdots&&\\
	  1&\zeta\dwn{n}^{n-1}&\zeta\dwn{n}^{2(n-1)}&\cdots&\zeta\dwn{n}^{(n-1)^2}}
\mtrx{d\up{n}-d\up{1}\\d\up{1}-d\up{2}\\\vdots\\d\up{n-1}-d\up{n}}
\]
Denote the above matrix by $S$ and note that it is unitary; thus
\begin{equation}\label{relation1}\mtrx{d\up{n}-d\up{1}\\\vdots\\d\up{n-1}-d\up{n}}=\bar S^T\mtrx{w_0\\\vdots\\w_{n-1}}.\end{equation}

Write $S$ in terms of \underline{row} vectors $S=\mtrx{S_0\\\vdots\\S_{n-1}},$ then
$S^{-1}=\mtrx{\bar S^T_0&\bar S^T_1&\cdots&\bar S^T_{n-1}}$
with $\bar S^T_i=\dfrac{1}{\sqrt{n}}\mtrx{1\\\zeta\dwn{n}^{-i}\\\vdots\\\zeta\dwn{n}^{(n-1)i}}.$ Exploiting the cyclicity of the cyclic Toda lattice, we have
\begin{equation}\label{relation2}\mtrx{d\up{1}-d\up{2}\\d\up{2}-d\up{3}\\\vdots\\d\up{n}-d\up{1
}}=\mtrx{\bar S^T_0&\zeta\dwn{n}^{-1}\bar S^T_1&\zeta\dwn{n}^{-2}\bar S^T_2&\cdots&\zeta\dwn{n}^{-(n-1)}\bar S^T_{n-1}}\mtrx{w_0\\w_1\\\vdots\\w_{n-1}}\end{equation}
and
\begin{equation}\label{relation3}\mtrx{d\up{n-1}-d\up{n}\\d\up{n}-d\up{1}\\\vdots\\d\up{n-2}-d\up{n-1}}=\mtrx{\bar S^T_0&\zeta\dwn{n}^{1}\bar S^T_1&\zeta\dwn{n}^{2}\bar S^T_2&\cdots&\zeta\dwn{n}^{(n-1)}\bar S^T_{n-1}}\mtrx{w_0\\w_1\\\vdots\\w_{n-1}}.\end{equation}
Since $\Delta w_k=S_k \Delta\mtrx{{d\up{n}-d\up{1}}\\{d\up{1}-d\up{2}}\\\vdots\\{d\up{n-1}-d\up{n}}},$ by the definition of the cyclic Toda equations (\ref{Toda}),
\[\Delta w_k=S_k\left(2\mtrx{e^{d\up{n}-d\up{1}}\\  e^{d\up{1}-d\up{2}}\\ \vdots\\ e^{d\up{n-1}-d\up{n}}}
-\mtrx{e^{d\up{1}-d\up{2}}\\e^{d\up{2}-d\up{3}}\\ \vdots\\ e^{d\up{n}-d\up{1}}}
-\mtrx{e^{d\up{n-1}-d\up{n}}\\e^{d\up{n}-d\up{1}}\\ \vdots\\ e^{d\up{n-2}-d\up{n-1}}}\right)
\]
The next step is to expand all the exponentials, we make use of the Hadamard product  $*$ on vectors, which is defined as $\mtrx{x_1\\ \vdots\\x_n}*\mtrx{y_1\\\vdots\\y_n}=\mtrx{x_1y_1\\\vdots\\x_ny_n}.$
\[\Delta w_k= S_k\sum\limits_{l\geq0}\frac{1}{l!}\left(2\mtrx{d\up{n}-d\up{1}\\d\up{1}-d\up{2}\\ \vdots\\d\up{n-1}-d\up{n}}^{*l}- 
\mtrx{d\up{1}-d\up{2}\\d\up{2}-d\up{3}\\\vdots\\d\up{n}-d\up{1}}^{*l}-\mtrx{d\up{n-1}-d\up{n}\\d\up{n}-d\up{1}\\ \vdots\\d\up{n-2}-d\up{n-1}}^{*l}
\right).\]
By equations (\ref{relation1}),(\ref{relation2}), and (\ref{relation3}), $\Delta w_k$ can be rewritten as
\[\Delta w_k=S_k\left(\sum\limits_{l\geq0}\frac{1}{l!}\left(2\left(\sum\limits_{i\in\Z_n} \bar S^T_iw_i\right)^{*l}-\left(\sum\limits_{i\in\Z_n} \bar S^T_iw_i\zeta\dwn{n}^i\right)^{*l}-\left(\sum\limits_{i\in\Z_n} \bar S^T_iw_i\zeta\dwn{n}^{-i}\right)^{*l}\right)\right).\]
Reindexing the sum by the index of the $w_j$'s we have
\[
\Delta w_k=S_k\sum\limits_{r\geq 0}\left( 2\sum\limits_{i_1+\cdots+i_s=r}\frac{1}{s!}\binom{r}{i_1,\cdots,i_s}\bar S^T_{i_1}*\cdots*\bar S^T_{i_s}w_{i_1}w_{i_2}\cdots w_{i_s}\right.\]
\[-\sum\limits_{i_1+\cdots+i_s=r}\frac{1}{s!}\zeta\dwn{n}^r\binom{r}{i_1,\cdots,i_s}\bar S^T_{i_1}*\cdots*\bar S^T_{i_s}w_{i_1}\cdots w_{i_s}
\left. -\sum\limits_{i_1+\cdots+i_s=r}\frac{1}{s!}\zeta\dwn{n}^r\binom{r}{i_1,\cdots,i_s}\bar S^T_{i_1}*\cdots*\bar S^T_{i_s}w_{i_1}\cdots w_{i_s}\right )
\]
\[=S_k\left(\sum\limits_{r\geq 0}(2-\zeta\dwn{n}^r+\zeta\dwn{n}^{-r})\sum\limits_{i_1+\cdots+i_s=r}\frac{1}{s!}\binom{r}{i_1,\cdots,i_s}\bar S^T_{i_1}*\dots*\bar S^T_{i_s}w_{i_1}\cdots w_{i_s}\right ).\]
Observe that $\bar S^T_i*\bar S_j^T=\frac{1}{\sqrt{n}}\bar S^T_a$ where $a=i+j\ \text{mod}\ n,$ so we may rewrite the above equation as
\[\Delta w_k=S_k\left (\sum\limits_{r\geq 0}|1-\zeta\dwn{n}^r|^2\sum\limits_{i_1+\cdots+i_s=r}\frac{1}{s!\sqrt{n^{s-1}}}\binom{r}{i_1,\cdots,i_s}\bar S^T_{r(\text{mod}\ n)}w_{i_1}\cdots w_{i_s}\right).\]
Since $S$ is unitary, $S_i\bar S^T_j=\delta_{ij},$ and 
thus
\[\Delta w_k=|1-\zeta\dwn{n}^k|^2\sum\limits_{\substack{r\equiv k\\ \text{mod}\ n}}\sum\limits_{i_1+\cdots+i_s=r}\frac{1}{s!\sqrt{n^{s-1}}}\binom{r}{i_1,\cdots,i_s}w_{i_1}\cdots w_{i_s}\]
as desired.
\end{proof}

For applications to the error term, we will need the following perturbed version of the Toda system
\eqtns{
\Delta d\up{1}=a(e^{d\up{1}-d\up{2}}-e^{d\up{n}-d\up{1}}+f_1)\\
\Delta d\up{2}=a(e^{d\up{2}-d\up{3}}-e^{d\up{1}-d\up{2}}+f_2)\\
\ \ \ \ \ \vdots\\
\Delta d\up{n}=a(e^{d\up{n}-d\up{1}}-e^{d\up{n-1}-d\up{n}}+f_n)
}{PerturbedToda}
where $a$ is a constant and $f_j=f_j(d\up{1},\dots,d\up{n})$ is a function. Again we define
\[w_k=\frac{1}{\sqrt{n}}\sum\limits_{i\in\Z_n}\zeta\dwn{n}^{ik}(d\up{i}-d\up{i+1}),\]
and have the following proposition analogous to Proposition \ref{w_kProposition}.
\begin{Proposition}\label{Perturbedw_kProposition}
	If $w_k$ is as above, then
	\begin{enumerate}[1.]
		\item $w_k=\dfrac{1-\zeta\dwn{n}^{-k}}{\sqrt{n}}\sum\limits_i \zeta\dwn{n}^{ki}d\up{i}$\vspace{.2cm}

		\item $\Delta w_k=\dfrac{a|1-\zeta\dwn{n}^k|^2}{\sqrt{n}}\sum\limits_i \zeta\dwn{n}^{(i-1)k}e^{d\up{i}-d\up{i+1}}+\frac{a}{\sqrt{n}}\sum\limits_i \zeta\dwn{n}^{ik}f_i$\vspace{.2cm}

		\item $\Delta w_k=a \sum\limits_{\substack{r\equiv k\\ \text{mod}\ n}}\  \sum\limits_{r_1+\cdots+r_s=r} \frac{1}{s!n^{\frac{s-1}{2}}}\binom{r}{r_1,\cdots,r_s}|1-\zeta\dwn{n}^k|^2 w_{r_1}w_{r_2}\cdots w_{r_s}+\frac{a}{\sqrt{n}}\sum\limits_i \zeta\dwn{n}^{ik}f_i$
	\end{enumerate}
\end{Proposition}
The proof of Proposition \ref{Perturbedw_kProposition} is a straightforward generalization of the proof of Proposition \ref{w_kProposition}.

\subsection{The $n$-cyclic case}
We now examine the error term and relate it to the $w_k$'s from the Toda lattice. Using the results in the previous section, we show that the matrix elements have the desired decay. 
Rewriting the Hitchin equations in terms of the $\tilde u\up{j}$ yields,
\eqtns{\Delta \tilde u\up{1}=4t^{\frac{2}{n}}(e^{\tilde u\up{1}-\tilde u\up{2}}-e^{-2\tilde u\up{1}})\\
\Delta \tilde u\up{2}=4t^{\frac{2}{n}}(e^{\tilde u\up{2}-\tilde u\up{3}}-e^{\tilde u\up{1}-\tilde u\up{2}})\\
\ \ \ \ \ \vdots\\
\Delta (-\tilde u\up{1})=4t^{\frac{2}{n}}(e^{-2\tilde u\up{1}}-e^{\tilde u\up{1}-\tilde u\up{2}})
}{ErrHitn}
This system is a special case of the cyclic Toda lattice, in fact, it is a real form of the $\fsl(n,\C)$ cyclic Toda lattice. Our techniques do not rely on this extra symmetry and we will think of $(\tilde u\up{1},\cdots,-\tilde u\up{1})$ as $(d^1,\cdots,d^n)$ satisfying $d\up{n+1-i}=-d\up{i}.$ Recall from (\ref{nError}), the error term $R$ is written as $B\dwn{n}^1+t^\frac{1}{n}B^2\dwn{n}.$ 

For ease of notation, write the diagonalizing matrix in terms of column vectors, $S^{-1}=\mtrx{\bar S^T_0&\bar S^T_1&\cdots&\bar S^T_{n-1}}.$
Thus, $B\dwn{n}^1$ is given by 
\[ \mtrx{\bar S^T_0&\bar S^T_1&\cdots&\bar S^T_{n-1}}\mtrx{d^1_z&&&\\&d^2_z&&\\&&\ddots&\\&&&d^n_z}\mtrx{S_0\\S_1\\\vdots\\S_{n-1}}
=\sum\limits_i\bar S_i^TS_id^{i+1}_z\]
Using the definition of $S_i,$ 
\[\bar S_i^TS_i=\frac{1}{n}\mtrx{1\\\zeta\dwn{n}^i\\\vdots\\\zeta\dwn{n}^{(n-1)i}}\mtrx{1&\zeta\dwn{n}^{-i}&\zeta\dwn{n}^{-2i}&\cdots&\zeta\dwn{n}^{-(n-1)i} }\]
so $(\bar S_i^TS_i)_{kl}=\frac{1}{n}\zeta\dwn{n}^{(k-1)i}\zeta\dwn{n}^{-(l-1)i}=\frac{1}{n}\zeta\dwn{n}^{(k-l)i}.$ Thus,
\[(B\dwn{n}^1)_{kl}=\frac{1}{n}\sum\limits_i \zeta\dwn{n}^{(k-l)i}d^{i+1}_z.\]
By Proposition \ref{w_kProposition} $w_k=\dfrac{1-\zeta\dwn{n}^{-k}}{\sqrt{n}}\sum\limits_i\zeta\dwn{n}^{ik}d^i,$ so we may rewrite $B\dwn{n}^1$ as
\[ (B\dwn{n}^1)_{kl}=\dfrac{\zeta\dwn{n}^{-(k-l)}}{(1-\zeta\dwn{n}^{-(k-l)})\sqrt{n}}(w_{k-l})_z\]
For $B\dwn{n}^2,$ we have
\[ \mtrx{\bar S^T_0&\cdots&\bar S^T_{n-1}}\mtrx{0&e^{d^1-d^2}-1&&&\\&0&e^{d^2-d^3}-1&&\\&&\ddots&\ddots&\\&&&0&e^{d^{n-1}-d^n}-1\\e^{d^n-d^1}-1&&&&0}\mtrx{S_0\\ S_1\\\vdots\\ S_{n-1}}\]
\[=\frac{1}{n}\sum\limits_i \bar S_i^T S_i (e^{d^i-d^{i+1}}-1).\]
Since 
\[\bar S_{i}^TS_{i+1} =\frac{1}{n}\mtrx{1\\\zeta\dwn{n}^{i}\\\vdots\\\zeta\dwn{n}^{(n-1)i}}\mtrx{1& \zeta\dwn{n}^{-i}&\zeta\dwn{n}^{-2i}&\cdots&\zeta\dwn{n}^{-(n-1)i}}\]
 we have 
 \[(\bar S_{i}^TS_{i+1})_{kl}=\frac{1}{n}\zeta\dwn{n}^{(k-1)i}\zeta\dwn{n}^{-(l-1)(i+1)}=\zeta\dwn{n}^{(k-l)i-(l-1)}.\] 
 This yields
 \[(B\dwn{n}^2)_{kl}=\frac{1}{n}\sum\limits_i \zeta\dwn{n}^{(k-l)i-(l-1)}(e^{d^i-d^{i+1}}-1)\]
Since the sum is over $\Z_n$, the constant terms sum to 0; thus
 \[(B\dwn{n}^2)_{kl}=\frac{\zeta\dwn{n}^{1-l}}{n}\sum\limits_i \zeta\dwn{n}^{(k-l)i}e^{d^i-d^{i+1}}.\]
 Using second part of Proposition \ref{Perturbedw_kProposition}, we have
  \[(B\dwn{n}^2)_{kl}=\frac{\zeta\dwn{n}^{1-l}}{n}\sum\limits_i \zeta\dwn{n}^{(k-l)i}e^{d^i-d^{i+1}}=\frac{C}{t^\frac{2}{n}}\Delta w_{k-l}.\]
In conclusion, we have the $(k,l)$-entry of the error term \begin{equation}\label{errortermcyclic_w_k}R_{kl}=(B\dwn{n}^1)_{kl}+t^\frac{1}{n}(B\dwn{n}^2)_{kl}=C(w_{k-l})_z+Ct^{-\frac{1}{n}}\Delta w_{k-l}.\end{equation}
\begin{Lemma}\label{PropertyEigensolutions}
For the $n$-cyclic case, the eigensolutions $w_k$'s satisfy 
\begin{enumerate}[1.]
	\item $w_k=w_{n-k},$
	\item $w_k$ is real.
\end{enumerate}

\end{Lemma}
\begin{proof}Both statements follow from the calculations below.
\begin{enumerate} [1.]
 	\item \begin{eqnarray*}w_k&=&\frac{1}{\sqrt{n}}\sum\limits_{i\in\Z_n}\zeta\dwn{n}^{ik}(d\up{i}-d\up{i+1})\\&=&\frac{1}{\sqrt{n}}\sum\limits_{i\in\Z_n}\zeta\dwn{n}^{(n-i)(n-k)}(d\up{i}-d\up{i+1})\end{eqnarray*}
 	 Since the $d^i$\ 's satisfy that $d\up{n+1-i}=-d\up{i}$,
 	\begin{eqnarray*} &=&\frac{1}{\sqrt{n}}\sum\limits_{i\in\Z_n}\zeta\dwn{n}^{(n-i)(n-k)}(d\up{n-i}-d\up{n-i+1})\\&=&\frac{1}{\sqrt{n}}\sum\limits_{i\in\Z_n}\zeta\dwn{n}^{i(n-k)}(d\up{i}-d\up{i+1})=w_{n-k}.\end{eqnarray*}
 	\item 
 	 \begin{eqnarray*}\overline{w_k}&=&\frac{1}{\sqrt{n}}\sum\limits_{i\in\Z_n}\zeta\dwn{n}^{-ik}(d\up{i}-d\up{i+1})\\&=&\frac{1}{\sqrt{n}}\sum\limits_{i\in\Z_n}\zeta\dwn{n}^{(n-i)k}(d\up{i}-d\up{i+1})
 	 \end{eqnarray*}
 	  Since the $d^i$'s satisfy that $d\up{n+1-i}=-d\up{i}$
 	  \begin{eqnarray*}&=&\frac{1}{\sqrt{n}}\sum\limits_{i\in\Z_n}\zeta\dwn{n}^{(n-i)k}(d\up{n-i}-d\up{n-i+1})\\&=&\frac{1}{\sqrt{n}}\sum\limits_{i\in\Z_n}\zeta\dwn{n}^{ik}(d\up{i}-d\up{i+1})=w_k.\end{eqnarray*}
 \end{enumerate} 

\end{proof}
 \begin{Remark}\label{w_ksmall}
 It is important to remember that the system we are interested in comes from the error terms $\tilde u\up{j},$ and $w_k=\frac{1}{\sqrt{n}}\sum\limits_i\zeta\dwn{n}^{ki}(\tilde u\up{i}-\tilde u\up{i+1}).$ By Theorem \ref{metrictheorem}, 
$\tilde u\up{j}\sim ln\left(1+O\left(t^{-\frac{2}{n}}\right)\right)\sim O\left(t^{-\frac{2}{n}}\right),$
 thus the $w_k$ are also small for large t. This will be important in the proof of Theorem \ref{TransportAsymp}.
 \end{Remark}
\subsection{The $(n-1)$-cyclic case}
Recall from (\ref{(n-1)Error}) the error term for the $(n-1)$-cyclic case is written as $B\dwn{n-1}^1+(2t)^\frac{1}{n-1}B\dwn{n-1}^2.$
 In this case, the system will be reduced to a perturbed version of a cyclic Toda lattice of rank  $(n-1)$. 
Writing the Hitchin equations in terms of the $\tilde v\up{j}$'s gives
\eqtns{\Delta \tilde v\up{1}=4(2t)^\frac{2}{n-1}(\haf e^{\tilde v\up{1}-\tilde v\up{2}}-\haf e^{-\tilde v\up{1}-\tilde v\up{2}})\\
\Delta \tilde v\up{2}=4(2t)^\frac{2}{n-1}(e^{\tilde v\up{2}-\tilde v\up{3}}-\haf e^{\tilde v\up{1}-\tilde v\up{2}}-\haf e^{-\tilde v\up{1}-\tilde v\up{2}})\\
\Delta \tilde v\up{3}=4(2t)^\frac{2}{n-1}( e^{\tilde v\up{3}-\tilde v\up{4}}- e^{\tilde v\up{2}-\tilde v\up{3}})\\
\ \ \ \ \ \vdots\\
\Delta(-\tilde v\up{1})=-\Delta \tilde v\up{1}=4(2t)^\frac{2}{n-1}(-\haf e^{\tilde v\up{1}-\tilde v\up{2}}+\haf e^{-\tilde v\up{1}-\tilde v\up{2}})
}{HitEq(n-1)}
If we set $f=\frac{1}{2}(e^{-\tilde v\up{1}}+e^{\tilde v\up{1}}-2)e^{-\tilde v\up{2}}$, then equations $2$ through $n-1$ are
\eqtns{
\Delta \tilde v\up{2}=4(2t)^{\frac{2}{n-1}}(e^{\tilde v\up{2}-\tilde v\up{3}}-e^{-\tilde v\up{2}})-f\\
\Delta \tilde v\up{3}=4(2t)^{\frac{2}{n-1}}(e^{\tilde v\up{3}-\tilde v\up{4}}-e^{\tilde v\up{2}-\tilde v\up{3}})\\
\ \ \ \ \ \vdots\\
\Delta(- \tilde v\up{2})=4(2t)^{\frac{2}{n-1}}(-e^{\tilde v\up{2}-\tilde v\up{3}}+e^{-\tilde v\up{2}})+f
}{Todan-1}
If we add the equation $
\Delta 0=4(2t)^\frac{2}{n-1}(e^{-\tilde v\up{2}}-e^{-\tilde v\up{2}}),$
 this system is a special case of the perturbed cyclic Toda lattice (\ref{PerturbedToda}) of rank $n-1$. 
As in the $n$-cyclic case, set $(\tilde v\up{2},\tilde v\up{3},\cdots,-\tilde v\up{2},0)=(d^1,d^2,\cdots,d^{n-1}).$ 
If $w_k=\frac{1}{\sqrt{n-1}}\sum\limits_i \zeta\dwn{n-1}^{ik}(d^i-d^{i+1})$ then 
\[\Delta w_k=4(2t)^{\frac{2}{n-1}}\frac{|1-\zeta\dwn{n-1}^k|^2}{\sqrt{n-1}}\sum\limits_{i\in\Z_{n-1}} \zeta\dwn{n-1}^{(i-1)k}e^{d^i-d^{i+1}}+C(2t)^{\frac{2}{n-1}}f.\] 
For the term $B\dwn{n}^1,$ recall that $S^{-1}=\bar S^T.$ A simple calculations shows that 
\[\bar S^T=\mtrx{\frac{1}{\sqrt{2}}&0&\cdots&0&-\frac{1}{\sqrt{2}}\\
				 \frac{1}{\sqrt{2}} \bar S^T_0&\bar S^T_1&\cdots&\bar S^T_{n-2}&\frac{1}{\sqrt{2}}\bar S^T_0}\]
where, as in the $(n-1)$-cyclic case for $SL(n-1,\R)$, the column vector $\bar S^T_j=\frac{1}{\sqrt{n-1}}\mtrx{1\\\zeta\dwn{n-1}^j\\\vdots\\\zeta\dwn{n-1}^{(n-2)j}}.$ By definition of $B\dwn{n-1}^1$,

\[B\dwn{n-1}^1=\mtrx{\frac{1}{\sqrt{2}}&0&\cdots&0&-\frac{1}{\sqrt{2}}\\
				 \frac{1}{\sqrt{2}} \bar S^T_0&\bar S^T_1&\cdots&\bar S^T_{n-2}&\frac{1}{\sqrt{2}}\bar S^T_0	
					}
\mtrx{
\tilde v\up{1}_z&&&&\\
&d^{1}_z&&&\\
&&\ddots&&\\
&&&d^{n-2}_z&\\
&&&&-\tilde v\up{1}_z}
\mtrx{\frac{1}{\sqrt{2}}&\frac{1}{\sqrt{2}} S_0\\
				0&S_1\\
				\vdots&\vdots\\
				0&S_{n-2}\\
				 -\frac{1}{\sqrt{2}}&\frac{1}{\sqrt{2}}S_0}
					\]
\[=\mtrx{\frac{1}{\sqrt{2}}\tilde v\up{1}_z&0&\cdots&0&\frac{1}{\sqrt{2}}\tilde v\up{1}_z\\
 \frac{1}{\sqrt{2}} \bar S^T_0\tilde v\up{1}_z&\bar S^T_1d^1_z&\cdots&\bar S^T_{n-2}d^{n-2}_z&-\frac{1}{\sqrt{2}}\bar S^T_0\tilde v\up{1}_z}
\mtrx{\frac{1}{\sqrt{2}}&\frac{1}{\sqrt{2}} S_0\\
				0&S_1\\
				\vdots&\vdots\\
				0&S_{n-2}\\
				 -\frac{1}{\sqrt{2}}&\frac{1}{\sqrt{2}}S_F0}\]
\[=\mtrx{0&\tilde v\up{1}_zS_0\\
		 \tilde v\up{1}_z\bar S_0^T&\bar S_1^T S_1d^1_z+\cdots+\bar S^T_{n-2}S_{n-2}d^{n-2}_z}=\mtrx{0&\tilde v\up{1}_zS_0\\
		 \tilde v\up{1}_z\bar S_0^T& \sum\limits_{i\in\Z_{n-1}}\bar S_i^TS_id^i_z} \]
In the matrix above, the $(1,1)$ entry is a $1\times 1$ matrix, the $(1,2)$ entry is a row vector of length $(n-1)$, the $(2,1)$-entry is a column vector of length $(n-1)$ and the $(2,2)$-entry is a $(n-1)\times(n-1)$-matrix.
Thus 
\[(B\dwn{n-1}^1)_{kl}=\begin{dcases}C(w_{k-l})_z& k\geq 2\ \text{and}\ l\geq 2\\
				0& k=l=1\\
				c\tilde v\up{1}_z& otherwise\end{dcases}\]
Now for $B\dwn{n-1}^2,$ as above, we have
\[B\dwn{n-1}^2=\bar S^T
\mtrx{&\frac{1}{\sqrt{2}}(e^{\tilde v\up{1}-d^1}-1)&&&\\
	&&e^{d^1-d^2}-1&&\\
	&&&\ddots&\\
	\frac{1}{\sqrt{2}}(e^{-\tilde v\up{1}-d^1}-1)&&&&\frac{1}{\sqrt{2}}(e^{\tilde v\up{1}-d^1}-1)\\
	&\frac{1}{\sqrt{2}}(e^{-\tilde v\up{1}-d^1}-1)&&&}S\]
\[=\scalemath{.9}{\mtrx{
			0&\haf(e^{\tilde v\up{1}-d^1}-e^{-\tilde v\up{1}-d^1})&0&\dots&0\\
			\frac{1}{\sqrt{2}}\bar S^T_{n-2}(e^{-\tilde v\up{1}-d^1}-1)& \haf \bar S_0^T(e^{\tilde v\up{1}-d^1}+e^{-\tilde v\up{1}-d^1}-2)&\bar S_1(e^{d^1-d^2}-1)&\cdots&\frac{1}{\sqrt{2}}\bar S^T_{n-2}(e^{\tilde v\up{1}-d^1}-1)
}\mtrx{\frac{1}{\sqrt{2}}&\frac{1}{\sqrt{2}} S_0\\
				0&S_1\\
				\vdots&\vdots\\
				0&S_{n-2}\\
				 -\frac{1}{\sqrt{2}}&\frac{1}{\sqrt{2}}S_0}}\]
\[=\mtrx{0&\haf S_1 (e^{\tilde v\up{1}-d^1}-e^{-\tilde v\up{1}-d^1})\\
	\haf \bar S^T_{n-2}(-e^{\tilde v\up{1}-d^1}+e^{-\tilde v\up{1}-d^1})&\star}\]
where $\star$ is the $(n-1)\times(n-1)$-matrix given by
\[\haf \bar S^T_{n-2}S_0(e^{-\tilde v\up{1}-d^1}-1)+\haf\bar S^T_0S_1(e^{\tilde v\up{1}-d^1}+e^{\tilde v\up{1}-d^1}-2)+\haf \bar S^T_{n-2}S_0(e^{\tilde v\up{1}-d^1}-1)+\sum\limits_{i=1}^{n-3}\bar S_i^TS_{i+1}(e^{d^{i}-d^{i+1}}-1).\]
Rewrite $\star$ as $I-II$ where $I$ contains all exponential terms and $II$ contains all constant terms. Then
\[II=2\times\haf \bar S^T_{n-2}S_0+\haf \bar S_0^TS_1\cdot 2+\bar S_1^TS_2+\cdots+\bar S_{n-3}^TS_{n-2}\]
\[=\sum\limits_{i\in\Z_{n-1}} \bar S_i^TS_{i+1}=0.\]
Writing $\haf\bar S^T_{n-2} S_0(e^{-\tilde v\up{1}-d^1}+e^{\tilde v\up{1}-d^1})= \haf \bar S^T_{n-2} S_0(2f+2e^{-d^1})$ as above, we have
\[\star= \bar S^T_{n-2} S_0(f+e^{-d^1})+\bar S_0^TS_1(f+e^{-d^1})+\sum\limits_{i=1}^{n-3}\bar S^T_iS_{i+1}e^{d^i-d^{i+1}} \]
\[=\sum \bar S_i^TS_{i+1}e^{d^i-d^{i+1}}+(\bar S^T_{n-2} S_0+\bar S_0^TS_1)f.\]
Thus, similar to the $n$-cyclic case,
\[(B\dwn{n-1}^2)_{kl}=\begin{dcases} ct^{-\frac{2}{n-1}}\Delta w_{k-l}+Cf & k,l\geq 2\\
							0& k=l=1\\
							c(e^{\tilde v\up{1}}-e^{-\tilde v\up{1}})e^{-d^{1}}&otherwise
\end{dcases}\]
In conclusion, we have the $(k,l)$-entry of the error term
\[R_{kl}=(B\dwn{n-1}^1)_{kl}+(2t)^\frac{1}{n-1}(B\dwn{n-1}^2)_{kl}=\begin{dcases}c(w_{k-l})_z+ct^{-\frac{1}{n-1}}\Delta w_{k-l}+cft^{\frac{1}{n-1}} & k,l\geq 2\\
							0& k=l=1\\
							c\tilde v\up{1}_z+ct^{\frac{1}{n-1}}(e^{\tilde v\up{1}}-e^{-\tilde v\up{1}})e^{-d^{1}}& otherwise
\end{dcases}\]
As in $n$-cyclic case, the $w_k$ satisfy extra symmetries.
\begin{Lemma}\label{PropertyEigensolutionssub}
For the $(n-1)$-cyclic case, the eigensolutions $w_k$'s satisfy 
\begin{enumerate}[1.]
	\item $w_{n-1-k}=\zeta_{n-1}\up{k}w_k$;
\item $\dfrac{w_k}{i(1-\zeta_{n-1}\up{-k})}$ is real.
\end{enumerate}
 
\end{Lemma}
 \begin{Remark}
 Again, it is important to remember that the system we are interested in comes from the error terms $\tilde v\up{j}.$ By Theorem \ref{metrictheorem}, $\tilde v\up{1}$ and $w_k$ are also small for large t. This will be important in the proof of Theorem \ref{TransportAsymp}.
 \end{Remark}

 \section{Error estimate}\label{ErrorEstimate}
In this section we prove the main error estimates, Theorem \ref{error} and Theorem \ref{suberror}. 
Let $P$ be a point away from the zeros of $q_b.$ Choose a local coordinate centered at $P$ so that $q_b=dz^b,$ and let $D$ be the disk of radius $R,$ in this coordinate chart, centered at $P.$ 

\subsection{The $n$-cyclic case}
The following proposition will be key.
\begin{Theorem}\label{ErrorEstimateTheorem}
 Let $d\up{i}$ be the error functions $\tilde u\up{i}$ in the Hitchin equations for the $n$-cyclic case, and define $w_k=\frac{1}{\sqrt{n}}\sum\limits_{i\in\Z_n}\zeta\dwn{n}^{ik}(d\up{i}-d\up{i+1})$ on the disk $D$, then
	\[w_k(z)= O\left(t^{-\frac{3}{2n}}e^{-2|1-\zeta\dwn{n}^k|t^\frac{1}{n}(R-|z|)}\right).\] 
\end{Theorem}
With the above theorem, we can show, as $t\ra\infty,$ the Hitchin equation decouples. This is consistent with the asymptotic studies of \cite{EndsHiggs,TaubesAsymptotics3manifolds}.
 \begin{Corollary}\label{decoupleCor}For $\phi=\tilde e_1+tq_ne_{n-1}$, away from the zeros of $q_n,$ the Hitchin equation $F_{A_t}+[\phi,\phi^{*_{h_t}}]=0$ decouples as $t\ra\infty$
\[\begin{dcases}
	F_{A_t}=0\\
	[\phi,\phi^{*_{h_t}}]=0
\end{dcases}\] 
 \end{Corollary}
\begin{proof}
Recall, in our local holomorphic frame $(\sigma_1,\cdots,\sigma_n)$ (\ref{RescaledFrame}), $F_{A_t}=Diag(\Delta u^1,\Delta u^2,\cdots,-\Delta u^1).$ 
By definition of the error term $\tilde u^j,$  we have (\ref{ErrHitn})
\[	|\Delta u^j|= |\Delta \tilde u^j|=4t^{\frac{2}{n}}\left|e^{\tilde u^{j-1}-\tilde u^j}-e^{\tilde u^j-\tilde u^{j+1}}\right|.\]
Also, with $(d^1,d^2,\cdots,d^n)=(\tilde u^1,\tilde u^2,\cdots,-\tilde u^1),$ we defined $w_k=\frac{1}{\sqrt{n}}\sum\limits_{i\in\Z_n}\zeta\dwn{n}^{ik}(d^{i}-d^{i+1});$ hence $\tilde u^{j}-\tilde u^{j+1}$ is a linear combination of $w_k$'s. By Theorem \ref{ErrorEstimateTheorem}, we obtain 
\[|\tilde u^{j}-\tilde u^{j+1}|= O\left(t^{-\frac{3}{2n}}e^{-2|1-\zeta\dwn{n}|t^\frac{1}{n}(R-|z|)}\right).\]
Thus,
\[|\Delta u^j|=O\left(t^{\frac{1}{2n}}e^{-2|1-\zeta\dwn{n}|t^\frac{1}{n}(R-|z|)}\right),\]
and the corollary follows. 
\end{proof}
Before proving Theorem \ref{ErrorEstimateTheorem} , we prove a series of lemmas. First, consider a radial function $\eta$ satisfying 
\begin{eqnarray}\label{I0}\Delta\eta=k\eta & \text{in $D$}\\
\eta=1 &  \text{on $\partial D$}.\end{eqnarray}
Then $\eta(r)$ satisfies 
\[\begin{dcases}
	\eta''+\frac{1}{r}\eta'-k\eta=0 & r<R\\ \eta(R)=1 &
\end{dcases}.\]
Hence $\eta(r)=\frac{I_0(\sqrt{k}r)}{I_0(\sqrt{k}R)}$, where $I_0$ is the modified Bessel function of second kind. We will need two basic facts about Bessel functions, see \cite{bessel}.
\begin{Lemma}\label{Bessel}
There exists $C>1,M>0$ such that $\forall x>M$, $$\frac{1}{C}\frac{e^x}{\sqrt{x}}\leq I_0(x)\leq C \frac{e^x}{\sqrt{x}}.$$
\end{Lemma}
\begin{Lemma}\label{monotone} 
$I_0'>0$, thus $I_0$ is an increasing function.
\end{Lemma}
Lemma \ref{Bessel} follows from the fact that the asymptotic expansion of $I_0(r)$ is $I_0(r)\sim\frac{e^x}{\sqrt{2\pi x}}$ as $x\rightarrow \infty$.
Let $y_k$ be the unique solution to the system (\ref{I0}), by Lemma \ref{monotone} 
\[y_k=\dfrac{I_0(\sqrt{k}|z|)}{I_0(\sqrt{k}R)}\leq \dfrac{I_0\left(max\{M,\sqrt{k}|z|\}\right)}{I_0(\sqrt{k}R)}.\]
For $k>0$, the function $y_k$ satisfies the following important property.
\begin{Lemma}\label{ykboundLemma} There exists a constant $C>0$ not depending on $k$ and $z$ such that
	\[e^{-\sqrt{k}(R-|z|)}\leq y_k(z)\leq C k^{\frac{1}{4}}e^{-\sqrt{k}(R-|z|)}.\]
\end{Lemma}
\begin{proof}
Left inequality: By simple calculation, $$\Delta e^{-\sqrt{k}(R-|z|)}\geq k e^{-\sqrt{k}(R-|z|)}$$ and $e^{-\sqrt{k}(R-|z|)}|_{\p D}=1$. By the maximum principle, we have \[e^{-\sqrt{k}(R-|z|)}\leq y_k(z).\]

Right inequality: 
For $|z|\leq \frac{M}{\sqrt{k}}$,  as $k\rightarrow +\infty$, by Lemma \ref{Bessel} 
\begin{eqnarray*}
y_k&\leq& C\frac{I_0(M)}{I_0(\sqrt{k}R)} \\
&\leq& C \frac{e^M}{\sqrt{M}}\frac{k^{\frac{1}{4}}R^{\frac{1}{2}}}{e^{\sqrt{k}R}}\\
&\leq& C k^{\frac{1}{4}}e^{\sqrt{k}(|z|-R)}.\end{eqnarray*}
For $|z|>\frac{M}{\sqrt{k}}$, 
\begin{eqnarray*}
y_k&\leq& C\frac{I_0(\sqrt{k}|z|)}{I_0(\sqrt{k}R)} \\
&\leq& C k^{\frac{1}{4}}e^{\sqrt{k}(|z|-R)}.\end{eqnarray*}
\end{proof}
Our goal is to use such functions $y_k$'s to bound the eigensolutions $w_j$'s by choosing the right $k$. We start by proving a lemma about a more general system.
\begin{Lemma}\label{TodaTechnicalLemma2}
	Suppose $\eta_1,\cdots,\eta_n$ are functions on $D$ satisfying
	\eqtns{\Delta\eta_1=\lambda_1t^{\frac{2}{n}}(\eta_1+f_1)\\
			   \ \ \ \vdots\\
			\Delta\eta_n=\lambda_nt^{\frac{2}{n}}(\eta_n+f_n)}{}
	  with $\lambda_j>0,$ $\lambda_1=min\{\lambda_j\},$ and $f_j=f_j(\eta_1,\cdots,\eta_n).$ Suppose further that there exists $C>0$ such that 
	 \[\begin{dcases}
	 	|f_j|\leq C(\eta_1^2+\cdots+\eta_n^2+y_{k t^{\frac{2}{n}}})& \text{for some}\ k>\lambda_1\\
	 	|\eta_j|\leq Ct^{-\frac{2}{n}}.
	 \end{dcases}\]
	 Then 
	 \[|\eta_1(z)|\leq Ct^{-\frac{3}{2n}}e^{-\sqrt{\lambda_1}t^{\frac{1}{n}}(R-|z|)}.\]
	\end{Lemma}
\begin{proof}
In order to obtain an upper and lower bound for $\eta_1,$ we first prove an upper bound on $\sum\limits_{j=1}^n\eta_j^2,$
\[\Delta\left(\sum\limits_{j=1}^n\eta_j^2\right)=2\sum\limits_j\eta_j\Delta\eta_j+2\sum\limits_j|\nabla \eta_j|^2\geq 2\sum\limits_j\eta_j\Delta\eta_j.\] 
Using our assumptions on $\Delta \eta_j$ and $\lambda_1$ gives
\[\Delta\left(\sum\limits_{j=1}^n\eta_j^2\right)\geq 2t^{\frac{2}{n}}\sum\limits_j\eta_j\lambda_j(\eta_j+f_j)\]
\[\geq 2t^{\frac{2}{n}}\lambda_1\sum\limits_j \eta_j^2-2t^{\frac{2}{n}}\sum\limits_j |\eta_j||f_j|\lambda_j\]
Now with our assumptions on $|f_j|$ and $|\eta_j|,$ we have
\[\Delta\left(\sum\limits_{j=1}^n\eta_j^2\right)\geq (2\lambda_1t^{\frac{2}{n}}-C)\sum\limits_j\eta_j^2-Cy_{k t^{\frac{2}{n}}}\]
Let $\lambda=min\{k, \frac{3}{2}\lambda_1\}$, then
\[\Delta\left(\sum\limits_{j=1}^n\eta_j^2\right)\geq (2\lambda_1t^{\frac{2}{n}}-C)\sum\limits_j\eta_j^2-Cy_{\lambda t^{\frac{2}{n}}}\]
Let $C \geq  \sum\limits_{j=1}^n\eta_j^2|_{\p D}$ and consider $C\cdot y_{\lambda t^{\frac{2}{n}}}$, for $t$ large enough,
\[\Delta C y_{\lambda t^{\frac{2}{n}}}\leq (2\lambda_1t^{\frac{2}{n}}-C)C y_{\lambda t^{\frac{2}{n}}}-Cy_{\lambda t^{\frac{2}{n}}}\]
Hence by the maximum principle, in $D$, 
\[\sum\limits_{j=1}^n\eta_j^2\leq C y_{\lambda t^{\frac{2}{n}}}.\]

To obtain an upper bound for $\eta_1,$ consider
\[\Delta \eta_1=\lambda_1t^{\frac{2}{n}}(\eta_1+f_1)\geq \lambda_1 t^{\frac{2}{n}}\eta_1-C t^{\frac{2}{n}}y_{\lambda t^{\frac{2}{n}}}.\]
Let $b=C t^{-\frac{2}{n}}\geq  \eta_1|_{\p D}$  and consider $by_{\lambda_1 t^{\frac{2}{n}}-C}$, we obtain
\[\Delta by_{\lambda_1 t^{\frac{2}{n}}-C}\leq \lambda_1t^{\frac{2}{n}}by_{\lambda_1 t^{\frac{2}{n}}-C}-Cy_{\lambda t^{\frac{2}{n}}}\]
Similarly, by the maximum principle, on $D$ \[\eta_1\leq by_{\lambda_1 t^{\frac{2}{n}}-C}.\]
For the lower bound, consider $\Delta(-\eta_1)=\lambda_1(-\eta_1)-f_1.$ By the same argument as above, $-\eta_1\leq by_{\lambda_1 t^{\frac{2}{n}}-C}.$
Finally, by the estimate on $y_k$ from lemma \ref{ykboundLemma}, we obtain the desired 
\[|\eta_1(z)|\leq Ct^{-\frac{3}{2n}}e^{-\sqrt{\lambda_1}t^{\frac{1}{n}}(R-|z|)}.\]
\end{proof}
We are now set up to prove Theorem \ref{ErrorEstimateTheorem}.
\begin{proof}(of Theorem \ref{ErrorEstimateTheorem})
Recall from (\ref{wkdef}), $w_k=\frac{1}{\sqrt{n}}\sum\limits_{i\in\Z_n}\zeta\dwn{n}^{ik}(d\up{i}-d\up{i+1});$ 
by Lemma \ref{PropertyEigensolutions}, $w_k$ is real. The proof is by induction on $k.$ 

For the base case, we show that the $w_k$'s satisfy the hypothesis of Lemma \ref{TodaTechnicalLemma2}. 
By part (iii) of Proposition \ref{Perturbedw_kProposition}, 
\[\Delta w_k=4 t^\frac{2}{n}|1-\zeta\dwn{n}^k|^2\sum\limits_{\substack{r\equiv k\\ \text{mod}\ n}}\  \sum\limits_{r_1+\cdots+r_s=r} \frac{1}{s!n^{\frac{s-1}{2}}}\binom{r}{r_1,\cdots,r_s} w_{r_1}w_{r_2}\cdots w_{r_s}.\]
Taking the $r=k$ term out of the sum gives
\[\Delta w_k=4 t^\frac{2}{n}|1-\zeta\dwn{n}^k|^2w_k+t^\frac{2}{n}f_k\]
where $f_k$ is the remainder of the sum. Since the first term has been removed, each term in $f_k$ is a polynomial in the $w_j$'s with no linear or constant term. Recall, by remark \ref{w_ksmall}, each $w_j\sim O\left(t^{-\frac{2}{n}}\right).$ To show $|f_k|\leq C(\sum\limits w_j^2),$ note that  the $s=2$ summand of $f_k$ is  
\[4 t^\frac{2}{n}\sum\limits_{\substack{r\equiv k\\ \text{mod}\ n}}\ \sum\limits_{r_1+r_2=r} \frac{1}{2!\sqrt{n}}\binom{r}{r_1,r_2}|1-\zeta\dwn{n}^k|^2 w_{r_1}w_{r_2}\leq C \sum {w_j}^2.\]
For $s>2$ we use remark \ref{w_ksmall} to conclude
\[|f_k|\leq C\sum\limits_j {w_j}^2(1+O\left(t^{-\frac{2}{n}}\right)+O\left(t^{-\frac{4}{n}}\right)+\cdots)\\\leq C\sum\limits_j {w_j}^2.\]
Furthermore, $|1-\zeta\dwn{n}|^2\leq |1-\zeta\dwn{n}^k|^2$ for all $k<n,$ thus the system of $w_k$'s satisfies the hypothesis of Lemma \ref{TodaTechnicalLemma2} with $\lambda_k=|1-\zeta\dwn{n}^k|^2.$ Thus 
\[w_k(z)= O\left(t^{-\frac{3}{2n}}e^{-2|1-\zeta\dwn{n}|t^\frac{1}{n}(R-|z|)}\right),\] 
proving the base case.

For the induction step, suppose $w_l=O\left(t^{-\frac{3}{2n}}e^{-2|1-\zeta\dwn{n}^l|t^\frac{1}{n}(R-|z|)}\right)$ for all $l\leq k.$ By Lemma \ref{PropertyEigensolutions}, $w_l=w_{n-l},$ and if $l\leq k$ the induction hypothesis implies $w_{n-l}(z)=O\left(t^{-\frac{3}{2n}}e^{-2|1-\zeta\dwn{n}^{n-l}|t^\frac{1}{n}(R-|z|)}\right).$ 

To obtain the estimate for $w_{k+1}$, the assumptions are not enough to directly estimate the terms containing $w_j$'s for $k+1\leq j\leq n-k-1$. Therefore we need use the rest system to do the estimate. Expanding the expression for $\Delta w_{k+i}$, we have
\[\frac{\Delta w_{k+i}}{4t^{\frac{2}{n}}}={|1-\zeta\dwn{n}^{k+i}|}^2 w_{k+i}+\sum\limits_{r_1+\cdots+r_s=k+i}\frac{1}{s!n^{\frac{s-1}{2}}}\binom{k+i}{r_1,\cdots,r_s}w_{r_1}\cdots w_{r_s}\]
\[+\sum\limits_{\substack{r_1+\cdots+r_s\\=k+i+n}}\frac{1}{s!n^{\frac{s-1}{2}}}\binom{k+i+n}{r_1,\cdots,r_s}w_{r_1}\cdots w_{r_s}+\cdots\]

Let $A\subset\{1,\cdots,n-1\}$ be the set $A=\{k+1,\cdots,n-k-1\}$ and $B=\{1,\cdots,n-1\}\setminus A.$
Rewrite the above equation as 
$\dfrac{\Delta w_{k+i}}{4t^\frac{2}{n}}=|1-\zeta\dwn{n}^{k+i}|^2w_{k+i}+E+F+G,
$ where
\eqtns{
	E=\sum\limits_{\substack{r\equiv k+i\\ \text{mod}\ n}}\  \sum\limits_{\substack{r_1+\cdots+r_s=r\\ r_j\in B}} \frac{1}{s!n^{\frac{n-1}{2}}}\binom{r}{r_1,\cdots,r_s}|1-\zeta\dwn{n}^{k+i}|^2 w_{r_1}w_{r_2}\cdots w_{r_s}\\
	F=\sum\limits_{\substack{r\equiv k+i\\ \text{mod}\ n}}\  \sum\limits_{\substack{r_1+\cdots+r_s=r\\ \text{at most one of $r_j$'s in $A$}}} \frac{1}{s!n^{\frac{n-1}{2}}}\binom{r}{r_1,\cdots,r_s}|1-\zeta\dwn{n}^{k+i}|^2 w_{r_1}w_{r_2}\cdots w_{r_s}\\
	G=\sum\limits_{\substack{r\equiv k+i\\ \text{mod}\ n}}\  \sum\limits_{\substack{r_1+\cdots+r_s=r\\  \text{at least two of $r_j$'s in $A$}}} \frac{1}{s!n^{\frac{n-1}{2}}}\binom{r}{r_1,\cdots,r_s}|1-\zeta\dwn{n}^{k+i}|^2 w_{r_1}w_{r_2}\cdots w_{r_s}
}{} 
For $E$, the induction hypothesis applies, for $l\in B$ \[w_{l}=\left(t^{-\frac{3}{2n}}e^{-2|1-\zeta\dwn{n}^{l}|t^\frac{1}{n}(R-|z|)}\right).\] 
Also, by the triangle inequality, for $s\geq 2$,
\[|1-\zeta\dwn{n}^{r_1}|+\cdots+|1-\zeta\dwn{n}^{r_s}|>|1-\zeta\dwn{n}^{r_1+\cdots+r_s}|,\] thus
\[\sum\limits_{r_1+\cdots+r_s=k+1}\frac{1}{s!n^{\frac{n-1}{2}}}\binom{k+1}{r_1,\cdots,r_s}w_{r_1}\cdots w_{r_s}\sim O\left(t^{-\frac{3}{n}}e^{-2|1-\zeta\dwn{n}^{k+1}|t^\frac{1}{n}(R-|z|)}\right).\] 
Hence we have \[E\sim O\left(t^{-\frac{3}{n}}e^{-2|1-\zeta\dwn{n}^{k+1}|t^\frac{1}{n}(R-|z|)}\right).\]
Applying this to the term $F$ yields 
 \[F=\sum\limits_{\substack{r\equiv k+i\\ \text{mod}\ n}}\  \sum\limits_{\substack{r_0+r_1=r\\ \ r_1\in A}} O\left(t^{-\frac{3}{2n}}e^{-2|1-\zeta\dwn{n}^{r_0}|t^\frac{1}{n}(R-|z|)}\right)w_{r_1}.\]
 For $j\in A$, we need the following coarse estimate of the $w_j$'s.
\begin{Claim}\label{claimwk+ibounds}
	\[|w_{k+i}|\leq Ct^{-\frac{3}{2n}}e^{-{\sqrt{2}}|1-\zeta\dwn{n}^{k+1}|t^\frac{1}{n}(R-|z|)}\]
\end{Claim}
\begin{proof}$(of\ Claim)$ We will estimate the sum $\sum\limits_{k+1\leq i\leq n-k-1}w_{k+i}^2$. Observe that
\[\frac{\Delta w_{k+i}^2}{4t^\frac{2}{n}}\geq 2w_{k+i}\dfrac{\Delta w_{k+i}}{4t^\frac{2}{n}}\]
\[=2|1-\zeta\dwn{n}^{k+i}|^2w^2_{k+i}+2w_{k+i}E+2w_{k+i}F+2w_{k+i}G\]
\[\geq \scalemath{.9}{2|1-\zeta\dwn{n}^{k+i}|^2w^2_{k+i} -O\left(t^{-\frac{3}{n}}e^{-2|1-\zeta\dwn{n}^{k+i}|t^\frac{1}{n}(R-|z|)}\right)\left(\sum\limits_iw_{k+i}^2\right)^\frac{1}{2}-\left(O\left(t^{-\frac{3}{2n}}e^{-2|1-\zeta\dwn{n}^{k+1}|t^\frac{1}{n}(R-|z|)}\right)+O\left(t^{-\frac{2}{n}}\right)\right)\sum\limits_iw_{k+i}^2} .\]
Thus, summing over $i$ for $1\leq i\leq n-2k-1$ yields
\[\frac{1}{4}t^{-\frac{2}{n}}\Delta(\sum\limits_i w_{k+i}^2)\geq (2|1-\zeta\dwn{n}^{k+1}|^2-Ct^{-\frac{2}{n}})\sum\limits_iw_{k+i}^2-Ct^{-\frac{3}{n}}e^{-2|1-\zeta\dwn{n}^{k+1}|t^\frac{1}{n}(R-|z|)}(\sum\limits_iw_{k+i}^2)^\frac{1}{2}.\]
Since $w_{k+i}\leq Ct^{-\frac{2}{n}},$ if $b=C t^{-\frac{4}{n}}(n-2k-1)$ then
\[(\sum\limits_i w_{k+i}^2)\left.\right|_{\p D} \leq b.\]

Consider $b\cdot y_{8|1-\zeta\dwn{n}^{k+1}|^2t^{\frac{2}{n}}-8C},$ in $D$ 
\[\frac{1}{4}t^{-\frac{2}{n}}\Delta(b\cdot y_{8|1-\zeta\dwn{n}^{k+1}|^2t^{\frac{2}{n}}-8C})=(2|1-\zeta\dwn{n}^{k+1}|^2-2Ct^{-\frac{2}{n}})b\cdot y_{8|1-\zeta\dwn{n}^{k+1}|^2t^{\frac{2}{n}}-8C}\] 
Using Lemma \ref{ykboundLemma}, for $t$ sufficiently large,
\[\leq (2|1-\zeta\dwn{n}^{k+1}|^2-Ct^{-\frac{2}{n}})b\cdot y_{8|1-\zeta\dwn{n}^{k+1}|^2t^{\frac{2}{n}}-8C}-C t^{-\frac{3}{n}}e^{-|1-\zeta\dwn{n}^{k+1}|t^\frac{1}{n}(R-|z|)}\sqrt{b}(y_{8|1-\zeta\dwn{n}^{k+1}|^2t^{\frac{2}{n}}-8C})^\haf.\]
By the maximum principle,
\[\sum\limits_i w_{k+i}^2\leq b\cdot y_{8|1-\zeta\dwn{n}^{k+1}|^2t^{\frac{2}{n}}-8C}\]
thus 
\[|w_{k+i}(z)|\leq C\sqrt{b}t^{\frac{1}{2n}}e^{-\sqrt{2}|1-\zeta\dwn{n}^{k+1}|t^\frac{1}{n}(R-|z|)},\] and by definition of $b,$
\[|w_{k+i}(z)|\leq Ct^{-\frac{3}{2n}}e^{-\sqrt{2}|1-\zeta\dwn{n}^{k+1}|t^\frac{1}{n}(R-|z|)}.\] \end{proof}
Recall that 
$\frac{1}{4}t^{-\frac{2}{n}}\Delta w_{k+1}=|1-\zeta\dwn{n}^{k+1}|^2w_{k+1}+E+F+G,$ 
and $E=O\left(t^{-\frac{3}{n}}e^{-2|1-\zeta\dwn{n}^{k+1}|t^\frac{1}{n}(R-|z|)}\right),$ thus 
\[\frac{1}{4}t^{-\frac{2}{n}}\Delta w_{k+1}\geq |1-\zeta\dwn{n}^{k+1}|^2w_{k+1} - O\left(t^{-\frac{3}{n}}e^{-2|1-\zeta\dwn{n}^{k+1}|)^\frac{1}{n}(R-|z|)}\right)\]
\[-O\left( t^{-\frac{3}{2n}}e^{-2|1-\zeta\dwn{n}|t^\frac{1}{n}(R-|z|)}\right)(\sum\limits_i w_{k+i}^2)^\haf-C \sum\limits_i w_{k+i}^2.\]
Applying claim \ref{claimwk+ibounds} yields
\[\frac{1}{4}t^{-\frac{2}{n}}\Delta w_{k+1}\geq 
|1-\zeta\dwn{n}^{k+1}|^2w_{k+1}-O\left(t^{-\frac{3}{n}}e^{-2|1-\zeta\dwn{n}^{k+1}|t^\frac{1}{n}(R-|z|)}\right)\]
\[-O\left(t^{-\frac{2}{n}}e^{-(2|1-\zeta\dwn{n}|+{\sqrt{2}}|1-\zeta\dwn{n}^{k+1}|)t^\frac{1}{n}(R-|z|) }\right)-O\left(t^{-\frac{1}{n}}e^{-2\sqrt{2}|1-\zeta\dwn{n}^{k+1}|t^\frac{1}{n}(R-|z|) }\right)\]
On the right hand side of the inequality, the second term dominates the last term, 
\[\frac{1}{4t^{\frac{2}{n}}}\Delta w_{k+1}\geq
|1-\zeta\dwn{n}^{k+1}|^2w_{k+1}-O\left(t^{-\frac{3}{n}}e^{-2|1-\zeta\dwn{n}^{k+1}|t^\frac{1}{n}(R-|z|)}\right)\]
\[-O\left(t^{-\frac{2}{n}}e^{-(2|1-\zeta\dwn{n}|+\sqrt{2}|1-\zeta\dwn{n}^{k+1}|)t^\frac{1}{n}(R-|z|)}\right). \]
For the competition of the second and third term, there are two cases:\\\\
Case I: $|1-\zeta\dwn{n}^{k+1}|\leq |1-\zeta\dwn{n}|+\frac{1}{\sqrt{2}}|1-\zeta\dwn{n}^{k+1}|$\\
	In this case we have 
	\[\frac{1}{4}t^{-\frac{2}{n}}\Delta w_{k+1}\geq 
|1-\zeta\dwn{n}^{k+1}|^2w_{k+1}- O\left(t^{-\frac{3}{n}}e^{-2|1-\zeta\dwn{n}^{k+1}|t^\frac{1}{n}(R-|z|)}\right).\] 
Let $C t^{-\frac{2}{n}} \geq  w_{k+1}|_{\p D}$ and consider $C\cdot y_{\lambda}$, where $\lambda=4t^\frac{2}{n}|1-\zeta\dwn{n}^{k+1}|^2-4C,$ then 
\[\frac{1}{4}t^{-\frac{2}{n}}\Delta(C t^{-\frac{2}{n}}\cdot y_{\lambda})=(|1-\zeta\dwn{n}^{k+1}|^2-Ct^{-\frac{2}{n}})(C t^{-\frac{2}{n}} y_{\lambda})\]
\[\leq |1-\zeta\dwn{n}^{k+1}|^2(C t^{-\frac{2}{n}}y_{\lambda})-O\left(t^{-\frac{3}{n}}e^{-(2|1-\zeta\dwn{n}|+{\sqrt{2}}|1-\zeta\dwn{n}^{k+1}|)t^\frac{1}{n}(R-|z|)}\right).\]
By the maximum principle, in $D$ 
\[w_{k+1}\leq C t^{-\frac{2}{n}}y_{\lambda},\]
thus we have the desired 
\[w_{k+1}= O\left(t^{-\frac{3}{2n}}e^{-2|1-\zeta\dwn{n}^{k+1}|t^\frac{1}{n}(R-|z|)}\right).\]
\\\\
Case II: $|1-\zeta\dwn{n}^{k+1}| > |1-\zeta\dwn{n}|+\frac{1}{\sqrt{2}}|1-\zeta\dwn{n}^{k+1}|,$ then
\[\frac{1}{4}t^{-\frac{2}{n}}\Delta w_{k+1}\geq 
|1-\zeta\dwn{n}^{k+1}|^2w_{k+1}-O\left(t^{-\frac{2}{n}}e^{-(2|1-\zeta\dwn{n}|+\sqrt{2}|1-\zeta\dwn{n}^{k+1}|)t^\frac{1}{n}(R-|z|)}\right).\] 
Let $C t^{-\frac{2}{n}} \geq  w_{k+1}|_{\p D}$ and consider $C\cdot y_k$, where $k={4t^\frac{2}{n}(|1-\zeta\dwn{n}|+\frac{1}{\sqrt{2}}|1-\zeta\dwn{n}^{k+1}|)^2},$ then 
\[\frac{1}{4}t^{-\frac{2}{n}}\Delta(C t^{-\frac{2}{n}}\cdot y_k)=(|1-\zeta\dwn{n}|+\frac{1}{\sqrt{2}}|1-\zeta\dwn{n}^{k+1}|)^2(C t^{-\frac{2}{n}} y_k)\]
\[\leq |1-\zeta\dwn{n}^{k+1}|^2(C t^{-\frac{2}{n}}y_k)-O\left(t^{-\frac{2}{n}}e^{-(2|1-\zeta\dwn{n}|+{\sqrt{2}}|1-\zeta\dwn{n}^{k+1}|)t^\frac{1}{n}(R-|z|)}\right).\]
By the maximum principle, on $D$ 
\[w_{k+1}\leq C t^{-\frac{2}{n}}y_k.\]
Similarly, applying the same process as above, the same bound holds for $-w_{k+i},$ hence 
\[|w_{k+i}|\leq C t^{-\frac{2}{n}}y_k.\] 
Since $|1-\zeta\dwn{n}^{k+i}|\geq |1-\zeta\dwn{n}^{k+1}|,$ we have the bound $|w_{k+i}|\leq C t^{-\frac{2}{n}}y_k.$
Therefore, using Lemma \ref{ykboundLemma} and $k={4t^\frac{2}{n}(|1-\zeta\dwn{n}|+\frac{1}{\sqrt{2}}|1-\zeta\dwn{n}^{k+1}|)^2},$
\begin{equation}\label{precisewki}
	|w_{k+i}(z)|\leq O\left(t^{-\frac{3}{2n}}e^{-(2|1-\zeta\dwn{n}|+\sqrt{2}|1-\zeta\dwn{n}^{k+1}|)t^\frac{1}{n}(R-|z|)}\right).
\end{equation}
Recall that,
\[\frac{1}{4}t^{-\frac{2}{n}}\Delta w_{k+1}=|1-\zeta\dwn{n}^{k+1}|^2w_{k+1}+E+F+G, \]
thus our sharper bounds of $w_{k+i}$ in (\ref{precisewki}) yield
\[\geq |1-\zeta\dwn{n}^{k+1}|^2w_{k+1} - O\left(t^{-\frac{3}{n}}e^{-2|1-\zeta\dwn{n}^{k+1}|t^\frac{1}{n}(R-|z|)}\right)
-O\left(t^{-\frac{2}{n}}e^{-(2|1-\zeta\dwn{n}|+2\sqrt{2}|1-\zeta\dwn{n}^{k+1}|)t^\frac{1}{n}(R-|z|)}\right).
\]
Again, there are two cases to consider:
\begin{enumerate}[\em{\text{Case\ }}\em{}i.]
	\item $|1-\zeta\dwn{n}^{k+1}|\leq |1-\zeta\dwn{n}|+\sqrt{2}|1-\zeta\dwn{n}^{k+1}|$
	\item  $|1-\zeta\dwn{n}^{k+1}|>|1-\zeta\dwn{n}|+\sqrt{2}|1-\zeta\dwn{n}^{k+1}|$
\end{enumerate}
For \em{Case} \em{i. applying an argument similar to Case I. will prove the result. For} \em{Case} \em{ii}, we will obtain a more precise bound on the asymptotics of $w_{k+i},$ with exponent $-(2|1-\zeta\dwn{n}|+2\sqrt{2}|1-\zeta\dwn{n}^{k+1}|)t^\frac{1}{n}(R-|z|).$ 
This will give two more cases, 
\begin{itemize}
	\item $|1-\zeta\dwn{n}^{k+1}|\leq |1-\zeta\dwn{n}|+\frac{3}{2}\sqrt{2}|1-\zeta\dwn{n}^{k+1}|$
	\item  $|1-\zeta\dwn{n}^{k+1}|>|1-\zeta\dwn{n}|+\frac{3}{2}\sqrt{2}|1-\zeta\dwn{n}^{k+1}|$ 
\end{itemize}
Again, in the first case, applying an argument similar to Case I. will prove the result. While the second case will give even more precise bounds, and two more cases. Since this process will terminate,
 we will eventually be able to apply an argument similar to Case I. to obtain the desired,
\[w_{k+1}(z)= O\left(t^{-\frac{3}{2n}}e^{-2|1-\zeta\dwn{n}^{k+1}|t^{\frac{1}{n}}(R-|z|)}\right)\] 

\end{proof}

\begin{Corollary}\label{Derivative}
With the same assumptions Theorem \ref{ErrorEstimateTheorem}, \[
\xymatrix{\Delta w_k(z)= O\left(t^{\frac{1}{2n}}e^{-2|1-\zeta\dwn{n}^k|t^\frac{1}{n}(R-|z|)}\right)&\txt{and}&|w^k_z|= O\left(t^{-\frac{1}{2n}}e^{-2|1-\zeta\dwn{n}^k|t^\frac{1}{n}(R-|z|)}\right)}.\] 
\end{Corollary}
\begin{proof}
Using part (iii) of Proposition \ref{Perturbedw_kProposition}, 
\[\Delta w_k=4|1-\zeta\dwn{n}^k|^2 t^\frac{2}{n}\sum\limits_{\substack{r\equiv k\\ \text{mod}\ n}}\  \sum\limits_{r_1+\cdots+r_s=r} \frac{1}{s!n^{\frac{s-1}{2}}}\binom{r}{r_1,\cdots,r_s} w_{r_1}w_{r_2}\cdots w_{r_s}\]
Applying Theorem \ref{ErrorEstimateTheorem} and the triangle inequality, for $s\geq 2$,
\[|1-\zeta\dwn{n}^{r_1}|+\cdots+|1-\zeta\dwn{n}^{r_s}|>|1-\zeta\dwn{n}^{r_1+\cdots+r_s}|.\] 
Thus, 
\[\Delta w_k(z)= O\left(t^{\frac{1}{2n}}e^{-2|1-\zeta\dwn{n}^k|t^{\frac{1}{n}}(R-|z|)}\right).\]
For $1\leq k\leq,$ consider the functions $\alpha^k$ defined by 
\[\alpha^{k}(z)=w_{k}(t^{-\frac{1}{n}}z),\]
where $w_k(t^{-\frac{1}{n}}z)$ is just a rescaling of $w_k.$
The $\alpha^{k}$'s then satisfy the following two properties
\[\xymatrix{|\alpha^{k}|=|w_k|\leq Ct^{-\frac{3}{2n}}e^{-2|1-\zeta\dwn{n}^k|t^{\frac{1}{n}}(R-|z|)}&\txt{and}&
|\alpha^{k}_{z\bz}|=|t^{-\frac{2}{n}}w^k_{z\bz}|\leq Ct^{-\frac{3}{2n}}e^{-2|1-\zeta\dwn{n}^k|t^{\frac{1}{n}}(R-|z|)}}.\]
Applying Schauder's estimate gives
\[|\alpha^{k}|_{C^1}\leq Ct^{-\frac{3}{2n}}e^{-2|1-\zeta\dwn{n}^k|t^{\frac{1}{n}}(R-|z|)}.\]
Thus $|\alpha^{k}_z|,|\alpha^{k}_{\overline{z}}|\leq Ct^{-\frac{3}{2n}}e^{-2|1-\zeta\dwn{n}^k|t^{\frac{1}{n}}(R-|z|)},$ proving \[|w^k_z|,|w^k_{\overline{z}}|\leq Ct^{-\frac{1}{2n}}e^{-2|1-\zeta\dwn{n}^k|t^{\frac{1}{n}}(R-|z|)}.\] 
\end{proof}
Now we are ready to prove the main error estimate. Recall that it states: If $d(p)$ be the minimum distance from a point $p$ to the zeros of $q_n.$ Then for any $d<d(p),$ as $t\rightarrow +\infty$, the $(k,l)$-entry of $R$ satisfies
\[R_{kl}(p)= O\left(t^{-\frac{1}{2n}}e^{-2|1-\zeta\dwn{n}^{k-l}|t^\frac{1}{n}d}\right).\] 

\begin{proof}(of Theorem \ref{error}) For any $d<d(p)$, choose a disk $D$ of radius $d$ centered at $p$, and denote the $(k,l)$-entry of $R$ by $R_{kl}$. From section \ref{ErrorToda}, we have $R_{kl}=Cw_{k-l}(z)_z+Ct^{-\frac{1}{n}}\Delta w_{k-1}(z)$. By Corollary \ref{Derivative}, as $t\rightarrow +\infty$, the $(k,l)$-entry of R satisfies
\[R_{kl}(p)= O\left(t^{-\frac{1}{2n}}e^{-2|1-\zeta\dwn{n}^{k-l}|t^\frac{1}{n}d}\right).\] 
\end{proof}
\subsection{The $(n-1)$-cyclic case} Recall, from equations (\ref{HitEq(n-1)}) and \ref{Todan-1},
the main difference between the $n$-cyclic case and the $(n-1)$-cyclic case is that the latter contains $\tilde v\up{1}$ which is not part of the perturbed Toda lattice. As a result, $\tilde v\up{1}$ will be estimated separately.
\begin{Lemma} \label{tildev1bounds}On a disc $D$ of radius $R$,
\[|\tilde v\up{1}(z)|= O\left(t^{-\frac{3}{2(n-1)}}e^{-2(2t)^\frac{1}{n-1}(R-|z|)}\right).\] 
\end{Lemma}
\begin{proof}From equation (\ref{HitEq(n-1)}), and $\tilde |v\up{j}|\leq C t^{-\frac{2}{n-1}},$
\[ \Delta \tilde v\up{1}=4(2t)^\frac{2}{n-1}(\haf e^{\tilde v\up{1}-\tilde v\up{2}}-\haf e^{-\tilde v\up{1}-\tilde v\up{2}})=2(2t)^\frac{2}{n-1}(e^{\tilde v\up{1}}-e^{-\tilde v\up{1}})e^{-\tilde v\up{2}}\]
\[\geq 2(2t)^{\frac{2}{n-1}}(2\tilde v\up{1})(1-Ct^{-\frac{2}{n-1}}).\]
By the maximum principle, in $D$, we have
\[\tilde v\up{1}(z)\leq Ct^{-\frac{2}{n-1}}y_{4(2t)^{\frac{2}{n-1}}(1-Ct^{-\frac{2}{n-1}})}\leq C t^{-\frac{3}{2(n-1)}}e^{-2(2t)^\frac{1}{n-1}(R-|z|)}.\]
Similarly, for  $-\tilde v\up{1}$ we obtain
\[-\tilde v\up{1}(z)\leq Ct^{-\frac{2}{n-1}}y_{4(2t)^{\frac{2}{n-1}}(1-Ct^{-\frac{1}{n-1}})}\leq C t^{-\frac{3}{2(n-1)}}e^{-2(2t)^\frac{1}{n-1}(R-|z|)}.\]
\end{proof}
Analogous to Theorem \ref{ErrorEstimateTheorem} for the $n$-cyclic case, the $w_k$'s have the following asymptotics.
\begin{Theorem}
\[|w_{k}(z)|= O\left(t^{-\frac{3}{2(n-1)}}e^{-2|1-\zeta\dwn{n-1}^{k}|(2t)^\frac{1}{n-1}(R-|z|)}\right)\]
\end{Theorem}

\begin{proof}
The proof is similar to that of Theorem \ref{ErrorEstimateTheorem}; the difference between the Toda lattices for the $n$-cyclic case and the $(n-1)$-cyclic case is that latter is perturbed by $f=\haf(e^{-\tilde v\up{1}}+e^{\tilde v\up{1}}-2)e^{-\tilde v\up{2}}$. However, we will show that $f$ has a strong enough decay so that it never effects the estimate of the $w_k$'s.
By Lemma \ref{tildev1bounds},  
\[f=\frac{1}{2}(e^{-\tilde v\up{1}}+e^{\tilde v\up{1}}-2)e^{-\tilde v\up{2}}\]
\[\leq \frac{1}{2}{\tilde v\up{1}(z)}^2(1+Ct^{-\frac{2}{n-1}})\leq C t^{-\frac{3}{n-1}}e^{-4(2t)^\frac{1}{n-1}(R-|z|)}.\]
Since $f\geq 0,$ we have $|f|=O\left(t^{-\frac{3}{n-1}}e^{-4(2t)^\frac{1}{n-1}(R-|z|)}\right)$. 
Since $4\geq 2|1-\zeta\dwn{n-1}^{k}|$, the perturbation $f$ has strong enough decay as not to effect the estimate of the $w_k$'s. Thus, we may apply the process of obtaining estimates in Theorem \ref{ErrorEstimateTheorem}.
\end{proof}
We also have the decoupled equation corollary.
\begin{Corollary}
	For $\phi=\tilde e_1+tq_{n-1}e_{n-2}$, away from the zeros of $q_{n-1},$ the Hitchin equation $F_{A_t}+[\phi,\phi^{*_{h_t}}]=0$ decouples as $t\ra\infty$
\[\begin{dcases}
	F_{A_t}=0\\
	[\phi,\phi^{*_{h_t}}]=0
\end{dcases}\] 
\end{Corollary}
Similar to Corollary \ref{Derivative}, we have 
\begin{Corollary} \label{subDerivative}
With the same condition as above, we have 
 \[\xymatrix{\Delta \tilde v\up{1}(z)= O\left(t^{\frac{1}{2(n-1)}}e^{-2(2t)^\frac{1}{n-1}(R-|z|)}\right)&& \Delta w_k(z)= O\left(t^{\frac{1}{2(n-1)}}e^{-2|1-\zeta\dwn{n-1}^k|(2t)^\frac{1}{n-1}(R-|z|)}\right)}\]and
\[\xymatrix{ |\tilde v\up{1}_z|= O\left(t^{-\frac{1}{2(n-1)}}e^{-2(2t)^\frac{1}{n-1}(R-|z|)}\right)
&&  
|w^k_z|= O\left(t^{-\frac{1}{2(n-1)}}e^{-2|1-\zeta\dwn{n-1}^k|(2t)^\frac{1}{n-1}(R-|z|)}\right)}.\] 
\end{Corollary}
Now we are ready to show Theorem \ref{suberror}.  Recall that it states: If $d(p)$ be the minimum distance from a point $p$ to the zeros of $q_{n-1}.$ Then for any $d<d(p),$ as $t\rightarrow +\infty$, the $(k,l)$-entry of $R$ satisfies
\[R_{kl}(p)=\begin{dcases} O\left(t^{-\frac{1}{2(n-1)}}e^{-2|1-\zeta\dwn{n-1}^{k-l}|(2t)^{\frac{1}{n-1}}d}\right)& k,l\geq 2\\
							0&         k=l=1\\
							O\left(t^{-\frac{1}{2(n-1)}}e^{-2(2t)^{\frac{1}{n-1}}d}\right)& otherwise
\end{dcases}\]
\begin{proof}(of Theorem \ref{suberror}) For any $d<d(p)$, choose a disk $D$ of radius $d$ centered at $p$. From section \ref{ErrorToda}, we have \[R_{kl}=(B\dwn{n-1}^1)_{kl}+(2t)^\frac{1}{n-1}(B\dwn{n-1}^2)_{kl}=\begin{dcases}c(w_{k-l})_z+ct^{-\frac{1}{n-1}}\Delta w_{k-l}+cft^{\frac{1}{n-1}} & k,l\geq 2\\
							0& k=l=1\\
							c\tilde v\up{1}_z+ct^{\frac{1}{n-1}}(e^{\tilde v\up{1}}-e^{-\tilde v\up{1}})e^{-d^{1}}& otherwise
\end{dcases}\]
 Corollary \ref{subDerivative} implies that, as $t\rightarrow +\infty$, $R_{kl}$ has the desired asymptotics.
\[R_{kl}(p)=\begin{dcases} O\left(t^{-\frac{1}{2(n-1)}}e^{-2|1-\zeta\dwn{n-1}^{k-l}|(2t)^\frac{1}{n-1}d}\right)& k,l\geq 2\\
							0& k=l=1\\
							O\left(t^{-\frac{1}{2(n-1)}}e^{-2(2t)^\frac{1}{n-1}d}\right)&otherwise
\end{dcases}\]
\end{proof}

\section{Elementary proof of WKB exponent} \label{WKBsection}
In this section we prove a special case of Theorem \ref{TransportAsymp} which only concerns the highest eigenvalue, or WKB exponent, of the transport operator. We include it here because the proof does not require the precise error estimates obtained in Theorem \ref{ErrorEstimateTheorem}, and hence is more elementary.
\begin{Theorem}\label{WKBAsymp}
	Let $P\in\widetilde{\Sigma}$ be away from the zeros of $q_b,$ choose a neighborhood $\Uu\dwn{p}$ centered at $P$, with coordinate $z,$ so that $q_b=dz^b.$ Any $P'\in\Uu\dwn{P}$ can be written in polar coordinate $P'=Le^{i\theta},$ then as $t\ra\infty$ there exists a constant $K>1$ such that
	\[\frac{1}{K}e^{Lt^\frac{1}{b}\mu}\leq||T_{P,P'}^{-1}(t)||\leq Ke^{Lt^\frac{1}{b}\mu}\]
\begin{enumerate}[1.]
		\item For the $n$-cyclic case $b=n,$ and $\mu=max\{2cos(\theta+\frac{2\pi j}{n})\}$.
		\item For the $(n-1)$-cyclic case $b=n-1$, and $\mu=max\{2cos(\theta+\frac{2\pi j}{n-1})\}$.
	\end{enumerate}
\end{Theorem}
\begin{Remark}
When $P$ and $P'$ both project to the same point in $\Sigma,$ the projected path is a loop. In this case, the above asymptotics correspond to the largest eigenvalue of the associated family of representations on the homotopy class of the loop. Note, considering the inverse path from $P'$ to $P,$ we obtain asymptotics of the smallest eigenvalue of $T_{P,P'}.$
\end{Remark}
The proof is a generalization of arguments of Loftin \cite{flatmetriccubicdiff,LoftinNotes} in which he deals with $n=3$ situation. To remain self-contained, we include the proof for the case $\phi=\tilde e_1+q_n e_{n-1}$ here.
\begin{proof}
	Since we are only proving the $n$-cyclic case, we drop all $b$ subscripts. Using the notation of section \ref{ODESetup}, we need to estimate $\Phi$ solving the initial value problem
	\[\Phi(0)=I\ \ \ \ \ \ \ \dfrac{d\Phi}{ds}=\left[t^{\frac{1}{n}} \mtrx{\mu_1&&\\&\ddots&\\&&\mu_n}+(b_{ij})
\right]\Phi\]
where $|b_{ij}|=O(t^{-\frac{1}{n}}).$ We will assume $\mu_1=\mu.$

Write $\Phi=(\phi_{ij})$, and consider the first column $(\phi_{11},\phi_{21},\dots,\phi_{n1})$, which satisfies the linear system \begin{eqnarray*}
\phi_{11}(0)=1,&  \frac{d}{ds}\phi_{11}=(t^{\frac{1}{n}}\mu_1+b_{11})\phi_{11}+b_{12}\phi_{21}+\dots+b_{1n}\phi_{n1},\\
\phi_{21}(0)=0,&  \frac{d}{ds}\phi_{21}=b_{21}\phi_{11}+(t^{\frac{1}{n}}\mu_2+b_{22})\phi_{21}+\dots+b_{2n}\phi_{n1},\\
\vdots&\\
\phi_{n1}(0)=0,&  \frac{d}{ds}\phi_{n1}=b_{n1}\phi_{11}+b_{n2}\phi_{21}+\dots+(t^{\frac{1}{n}}\mu_n+b_{nn})\phi_{n1}\end{eqnarray*}
Each of the above differential equations is first-order linear, and so we must have 
\begin{eqnarray*}
\phi_{11}&=&e^{t^{\frac{1}{n}}\mu_1 s}e^{\int\limits_0^ s b_{11}}\left[1+\int\limits_0^ s e^{-t^{\frac{1}{n}}\mu_1\tau-\int\limits_0^\tau b_{11}}(b_{12}\phi_{21}+\dots+b_{1n}\phi_{n1})d\tau\right]\\
\phi_{21}&=&e^{t^{\frac{1}{n}}\mu_2 s}e^{\int\limits_0^ s b_{22}}\int\limits_0^ s e^{-t^{\frac{1}{n}}\mu_2\tau-\int\limits_0^\tau b_{22}}(b_{21}\phi_{11}+b_{23}\phi_{31}+\dots+b_{2n}\phi_{n1})d\tau\\
&&\vdots\\
\phi_{n1}&=&e^{t^{\frac{1}{n}}\mu_n s}e^{\int\limits_0^ s b_{nn}}\int\limits_0^ s e^{-t^{\frac{1}{n}}\mu_n\tau-\int\limits_0^\tau b_{nn}}(b_{n1}\phi_{11}+\dots+b_{n,n-1}\phi_{n-1,1})d\tau
\end{eqnarray*}
The above $n$ equations can be see as a map $\Ff$ from the ${R}^n$-valued function $(\phi_{11},\phi_{21},\dots,\phi_{n1})$ to the right-hand side.\\
Now let $N\gg 1$ be a constant independent of $t$, and consider the Banach space $\mathcal{B}_{t}$ of continuous $R^n$-valued functions with norm \[||(f_1,f_2,\dots,f_n)||_{\mathcal{B}_t}=\sup\limits_i \sup\limits_{s\in[0,L]}|f_i(s)|e^{-t^{\frac{1}{n}}\mu_1 s}.\]
 Let $\mathcal{B}_t(N)$ be the closed ball of radius $N$ centered at the origin in $\mathcal{B}_t$.
\begin{Claim}\label{claim} For $t$ large enough, the solution $(\phi_{1j},\phi_{2j},\dots,\phi_{nj})$ to the ODE system, must lie in $\mathcal{B}_t(N)$ for all $1\leq j\leq n$.\end{Claim}
Thus, for sufficiently large $t,$ $|\phi_{ij}|\leq Ne^{t^\frac{1}{n}\mu_1L}.$  Applying this to the above equation system yields
\[\phi_{11}=e^{t^\frac{1}{n}\mu_1L}(1+O(t^{-\frac{1}{n}})),\ \ \ \phi_{21}=e^{t^\frac{1}{n}\mu_1L}O(t^{-\frac{1}{n}}),\ \ \,\cdots,\ \ \ \phi_{n1}=e^{t^\frac{1}{n}\mu_1L}O(t^{-\frac{1}{n}}).\]
For the second column of $\Phi,$ we have the equations
\begin{eqnarray*}
\phi_{12}&=&e^{t^{\frac{1}{n}}\mu_1 s}e^{\int\limits_0^ s b_{11}}\int\limits_0^s e^{-t^{\frac{1}{n}}\mu_1\tau-\int\limits_0^\tau b_{11}}(b_{12}\phi_{22}+\dots+b_{1n}\phi_{n2})d\tau\\
\phi_{22}&=&e^{t^{\frac{1}{n}}\mu_2 s}e^{\int\limits_0^ s b_{22}}\left[1+\int\limits_0^ s e^{-t^{\frac{1}{n}}\mu_2\tau-\int\limits_0^\tau b_{22}}(b_{21}\phi_{12}+b_{23}\phi_{32}+\dots+b_{2n}\phi_{n2})d\tau\right]\\
&&\vdots\\
\phi_{n2}&=&e^{t^{\frac{1}{n}}\mu_n s}e^{\int\limits_0^ s b_{nn}}\int\limits_0^ s e^{-t^{\frac{1}{n}}\mu_n\tau-\int\limits_0^\tau b_{nn}}(b_{n1}\phi_{12}+\dots+b_{n,n-1}\phi_{n-1,2})d\tau
\end{eqnarray*}
Thus, applying $|\phi_{ij}|\leq Ne^{t^\frac{1}{n}\mu_1s}$ yields
\begin{eqnarray*}
|\phi_{12}|&\leq &e^{t^\frac{1}{n}\mu_1s}O(t^{-\frac{1}{n}})\\
|\phi_{22}-e^{t^\frac{1}{n}\mu_2s}|&\leq &e^{t^\frac{1}{n}\mu_1s}O(t^{-\frac{1}{n}})\\
&\vdots&\\
|\phi_{n2}|&\leq &e^{t^\frac{1}{n}\mu_1s}O(t^{-\frac{1}{n}}),
\end{eqnarray*}
Similarly, $|\phi_{ij}-\delta_{ij}e^{t^\frac{1}{n}\mu_is}|\leq e^{t^\frac{1}{n}\mu_1s}$ for all $i,j.$ To determine the largest eigenvalue of $(\phi_{ij})$ asymptotically, we make use of the trace. Suppose $\xi_1\geq \xi_2\geq\cdots\geq \xi_n$ are the eigenvalues, then since $|\phi_{ij}|\leq N e^{t^\frac{1}{n}\mu_is},$ there is a constant $C$ so that
\[max(\xi_j)\leq C e^{t^\frac{1}{n}\mu_1s}.\]
Also, 
\[\xi_1+\cdots+\xi_n=Tr(\Phi)=e^{t^\frac{1}{n}\mu_1s}(1+O(t^{-\frac{1}{n}}))+e^{t^\frac{1}{n}\mu_2s}+\cdots+e^{t^\frac{1}{n}\mu_ns},\]
thus
\[\xi_1\geq\frac{1}{n}(\xi_1+\cdots+\xi_n)\geq \frac{1}{n}e^{t^\frac{1}{n}\mu_1s}(1+O(t^{-\frac{1}{n}})).\]
Hence
\[\frac{1}{n}e^{t^\frac{1}{n}\mu_1s}(1+O(t^{-\frac{1}{n}}))\leq \xi_1\leq Ce^{t^\frac{1}{n}\mu_1s}(1+O(t^{-\frac{1}{n}})),\]
and for $t$ sufficiently large, there exist constant $K> 0$ so that 
\[\frac{1}{K}e^{t^\frac{1}{n}\mu_1s}\leq \xi_1\leq Ke^{t^\frac{1}{n}\mu_1s}.\]
Since the highest eigenvalue $\xi_1=||T_{PP'}^{-1}(t)||,$ the result follows.
\begin{proof} \textit{(of Claim \ref{claim})} We will show that for $t$ large enough, $\Ff$ is a contraction map from $\mathcal{B}_t(N)$ to itself, and thus the solution $(\phi_{11},\phi_{21},\dots,\phi_{n1})$ to the ODE system, which is the fixed point of $\Ff$, must lie in $\mathcal{B}_t(N)$. For the rest of columns, the proof is identical.
Consider $F=(f_1,f_2,\dots,f_n), G=(g_1,g_2,\dots,g_n)\in \mathcal{B}_t(N)$. Then the first component of $\Ff(F)-\Ff(G)$ is given by 
\[e^{t^{\frac{1}{n}}\mu_1 s}e^{\int\limits_0^ s b_{11}}\int\limits_0^ s e^{-t^{\frac{1}{n}}\mu_1\tau-\int\limits_0^\tau b_{11}}[b_{12}(f_2-g_2)+\dots+b_{1n}(f_n-g_n)d\tau]\] 
Now assume $|b_{ij}|\leq R$ and recall $s\leq L$. 
A straightforward calculation shows that the first component of $\Ff(F)-\Ff(G)$ is pointwise bounded by 
\[e^{t^{\frac{1}{n}}\mu_1 s}e^{2RL}2RL ||F-G||_{\mathcal{B}_t}.\]
For $t$ sufficiently large, since $R\sim t^{-\frac{1}{n}},$ we may assume $e^{2RL}2RL<1$. Essentially the same calculation shows that $\Ff:\mathcal{B}_t(N)\rightarrow \mathcal{B}_t(N)$ for large $t$, since $N\gg 1$. The other $n-1$ components of $\Ff$ behave the same way, thus $\Ff$ is a contraction map.

Since $\Ff$ is a contraction map on the complete metric space $\mathcal{B}_t(N)$, the unique solution $(\phi_{11},\phi_{21},\dots,\phi_{n1})$ to the ODE system is the fixed point, and so must be in $\mathcal{B}_t(N)$ for all $t$ sufficiently large. 

\end{proof}
\end{proof}

\section{Harmonic maps into symmetric spaces}\label{HarmonicSection}
We continue to work in the universal cover $\widetilde\Sigma$ of $\Sigma,$ all objects should be pulled back from the surface. As in previous sections, we will use a subscript $b$ to work with the two cases $\phi=\tilde e_1+q_ne_{n-1}$ and $\phi=\tilde e_1+q_{n-1}e_{n-2}$ simultaneously.

A Hermitian metric $h$ on a flat bundle $E$ gives rise to an equivariant map to the symmetric space $SL(n,\C)/SU(n)$. 
To see this, fix a positively oriented unitary frame $\{x_i(P)\}$ over a base point $P\in \widetilde{\Sigma}$. With respect to the flat connection, parallel transport of the frame $\{x_i(P)\}$ gives a global frame $\{x_i\}.$
Define a $\pi_1(\Sigma)$-equivariant map by, 
\begin{eqnarray*}
f: \widetilde{\Sigma}&\longrightarrow& SL(n,\C)/SU(n)\\
P'&\longmapsto& \{h(x_i(P'),x_j(P'))\}.
\end{eqnarray*}
By Corlette's Theorem \cite{canonicalmetrics}, the family of harmonic metrics $h_t$ considered above, gives a family of $\rho_t$-equivariant harmonic maps 
\[f_t:\widetilde{\Sigma}\rightarrow SL(n,\C)/SU(n).\]
\begin{Remark}
The image of the family $f_t$ all lie in a copy of the real symmetric space \[SL(n,\R)/SO(n,\R)\subset SL(n,\C)/SU(n).\] This is because the family of representations $\rho_t$ has image in the real group $SL(n,\R).$
\end{Remark}

Pick a base point $P\in\widetilde{\Sigma}$ away from zeros of the differential $q_b.$ Recall that $\Uu_P$ is a local coordinate such that $q_b=dz^b,$ and $F_b=F=(\sigma_1,\dots,\sigma_{\frac{n}{2}},\sigma_{\frac{n}{2}}^*,\dots,\sigma_1^*)$ is a rescaled holomorphic frame (\ref{RescaledFrame}). By Remark \ref{FmetricRemark}, we can choose a unitary and orthogonal (with respect to the orthogonal structure $Q$) basis $N(P)$ at $P$ so that  
\[F(P)=N(P)(1+O(t^{-\frac{2}{b}})).\] 

Using the flat connection, parallel transport the unitary basis $N(P)$ to obtain a frame $N.$ 
Note that $N$ is not a unitary frame since the flat connection does not have holonomy in $SU(n);$ however, it retains its $SL(n,\R)$ symmetry. 
As a result, the image of $f_t$ is contained in a copy of $SL(n,\R)/SO(n,\R)\hook SL(n,\C)/SU(n).$ 
The inclusion is determined by the inclusion of $SO(n,\R)\subset SU(n)$
given by $Q$-orthogonal unitary matrices, and the intersection of $Q$-symmetric matrices with determinant 1 Hermitian matrices.

At a point $P',$ denote the $j^{th}$ column of $N$  by $N^j(P').$ 
Recall from equation (\ref{TvPsi}), that the parallel transport of the rescaled holomorphic frame $F$ at $P$ has been denoted by $G,$ and 
\[G(P')=N(P')(1+O(t^{-\frac{2}{b}})).\]
If we denote the $j^{th}$ column of $G(P')$ by $G^j(P'),$ we have
\[f_t(P')=\{ h_t(P')(N^i(P'),N^j(P')) \}\]
\[= \{h_t(P')\left(G^i(P') ,G^j(P')\right) (1+O(t^{-\frac{2}{b}}))\}.\]
By Proposition \ref{FmetricProp}, we understand $h_t(P')$ in the frame $F,$ thus, we change coordinates $\Psi_t(P')G(P')=F(P').$
In terms of columns, we have
\[G^i(P')=\Psi_t^{-1}(P')_{ik}F^{k}(P').\]
Thus $f_t(P')$ is given by 
\[f_t(P')=\{\ h_t(P')\left(\Psi_t^{-1}(P')_{ik}F^{k}(P') \ ,\ \Psi_t^{-1}(P')_{jl}F^{l}(P')\right) (1+O(t^{-\frac{2}{b}}))\ \}.\]
In the frame $F,$ the metric $h_t$ is diagonal, thus 
\[f_t(P')=\Psi_t^{-1}(P')^T\ h_t^F\ \overline{\Psi_t^{-1}}(P')\ (1+O(t^{-\frac{2}{b}}))\]
where $h_t^F$ denotes the metric in the rescaled holomorphic frame $F.$ 
Now, using Theorem \ref{metrictheorem} and Remark \ref{FmetricRemark}, we have 
\[f_t(P')=\  \Psi_t^{-1}(P')^T\cdot(1+O(t^{-\frac{2}{b}}))\cdot \overline{\Psi_t^{-1}}(P')\ .\]


Therefore, by applying estimates for $\Psi(L)$ in Theorems \ref{TransportAsymp} and \ref{subTransportAsymp}, for any $P'=L e^{i\theta}=\gamma(L)$ with the property that, for all $s,$
$s<d(\gamma(s)):=min\{d(\gamma(s), z_0)| \  \text{for all zeros $z_0$ of $q_n$}\},$ as $t\rightarrow \infty$ 
 \begin{eqnarray}\label{asympft}
f_t(P')= \left(Id+O\left(t^{-\frac{1}{2b}}\right)\right)\overline{S} \mtrx{
		e^{-2Lt^{\frac{1}{b}}\mu_1}&&&\\
		&e^{-2Lt^{\frac{1}{b}}\mu_2}&&\\
		&&\ddots&\\
		&&&e^{-2Lt^{\frac{1}{b}}\mu_n}}
S^T\left(Id+O\left(t^{-\frac{1}{2b}}\right)\right). \end{eqnarray}
 Here $S$ and the $\{\mu_j\}$'s satisfy the same conditions as in Theorems \ref{TransportAsymp} and \ref{subTransportAsymp}.
 By Remark \ref{smallnbhd}, the above equation can be interpreted as saying that for all such $P,$ there exists a neighborhood $\Uu_P,$ so that the $\rho_t$-equivariant maps $f_t:\tilde \Sigma\ra SL(n,\R)/SO(n,\R)$ send $\Uu_P$ asymptotically into a \em{flat} \em{}of the symmetric space.

 Given two points $P,P'$ in the symmetric space $SL(n,\R)/SO(n,\R)$, the vector distance between them is defined by $\overset{\ra}{d}(P,P')=P-P',$ where the difference is taken in a \em{flat}\em{} (isometric to $\A^{n-1}$) containing both points. One can show $\overset{\ra}{d}(P,P')$ is independent of the choice of flat. For example, in the standard flat of $SL(n,\R)/SO(n,\R)$ consisting of all diagonal matrices of determinant 1, the vector distance is defined by
\[\overset{\ra}{d}\left(\mtrx{1&&\\&\ddots&\\&&1},\mtrx{e^{\lambda_1}&&\\&\ddots&\\&&e^{\lambda_n}}\right)=(\lambda_1,\cdots,\lambda_n).\]
Since all flats in $SL(n,\R)/SO(n,\R)$ are conjugate to the standard flat, the vector distance can be defined in a similar way. 

The asymptotic expression (\ref{asympft}) for $f_t,$ together with the definition of vector distance, gives the following theorem. 
 \begin{Theorem}\label{harmMapTheorem} With the same assumptions as the parallel transport asymptotitcs Theorem \ref{TransportAsymp}, for a path $\gamma$ satisfying
\[ s<d(\gamma(s)):=min\{d(\gamma(s), z_0)| \  \text{for all zeros $z_0$ of $q_n$}\},\]
we have
\[\lim\limits_{t\ra\infty}\frac{1}{t^\frac{1}{n}}\overset{\ra}{d}(f_t(\gamma(0)),f_t(\gamma(1)))=\left(-2L\cos\left(\theta\right), -2L\cos\left(\theta+\frac{2\pi }{n}\right),\dots,-2L\cos\left(\theta+\frac{2\pi {(n-1)}}{n}\right)\right).\]
\end{Theorem}

A similar theorem holds in the $(n-1)$-cyclic case. As mentioned in the introduction, the above result answers the conjecture in \cite{harmonicbuildingWKB} on the asymptotics of the `Hitchin WKB problem' in these special cases.
In \cite{harmonicbuildingWKB}, a similar result is proven for the asymptotics of the complex WKB problem, i.e. the family of holomorphic flat connections $\nabla_t=\nabla_0+t\phi$ (\ref{CWKB}).

To close, we briefly discuss the behavior of the `limit map' $f_\infty$ associated to the family $f_t$, studied extensively in \cite{harmonicbuildingWKB}.
To obtain better information about the behavior of the maps $f_t$ as $t\ra \infty,$ we rescale the metric on the symmetric space and consider the family of maps
\[f_t:\widetilde\Sigma\ra \left(SL(n,\R)/SO(n,\R),\frac{1}{t^\frac{1}{b}}d\right). \]
The limit of $\left(SL(n,\R)/SO(n,\R),\frac{1}{t^\frac{1}{b}}d\right)$ as $t\ra\infty$ is not well defined, however, by the work of Kleiner-Leeb \cite{KeinerLeeb} and Parreau \cite{CompactificationHit}, a Gromov limit of $\left(SL(n,\R)/SO(n,\R),\frac{1}{t^\frac{1}{b}}d\right)$ is an affine building modeled on $\A^{n-1}.$ 
The limit construction depends on the choice of ultrafilter $\omega$ on $\R$ with countable support; with this choice, the limit is called the asymptotic cone and is denoted $Cone_{\omega}.$ 
In \cite{CompactificationHit}, Parreau showed that, given a diverging family of representations $\rho_t$, the limit of the vector length spectra of $\rho_t$ arises from the length spectrum of a limit action $\rho_\omega$ on $Cone_\omega.$ 
This gives a harmonic map 
\[f_{\omega}:\widetilde{\Sigma}\rightarrow Cone_{\omega},\] 
which is equivariant for the limiting action $\rho_{\omega}$ of $\pi_1(S)$ on $Cone_{\omega}.$ 

In this language, the asymptotic expression (\ref{asympft}) of $f_t$ implies that for the families of rays 
\[(\Sigma,0,\cdots,0,tq_n),\ (\Sigma,0,\cdots,tq_{n-1},0)\in Hit_n(S)\] 
and for any $P$ away from the zeros of $q_n$ and $q_{n-1},$ there exists a neighborhood $\mathcal{U}_P$ so that the $\rho_{\omega}$-equivariant map 
\[f_{\omega}:\widetilde{\Sigma}\rightarrow Cone_{\omega}\] 
sends $\mathcal{U}_P$ into a single apartment of the building $Cone_{\omega}$. 


\bibliographystyle{amsalpha}
\bibliography{bib}{}

\end{document}